\documentclass[11pt]{amsart}
\usepackage{amsmath}
\usepackage{amsfonts}
\usepackage{amssymb}
\usepackage{graphicx}
\usepackage{hyperref}
\usepackage{mathrsfs}
\usepackage{pb-diagram}
\usepackage{epstopdf}
\usepackage{amsmath,amsfonts,amsthm,enumerate,amscd,latexsym,curves}
\usepackage{bbm}
\usepackage[mathscr]{eucal}
\usepackage{epsfig,epsf,epic}
\usepackage{caption}
\usepackage{subcaption}
\usepackage{multirow}
\usepackage{color}
\usepackage[all]{xy}
\usepackage{fourier}
\usepackage{tikz-cd}
\usepackage{calc}
\usepackage{float}
\usepackage{comment}
\usepackage{gensymb}
\usepackage{longtable}

\graphicspath{{Figures/} }

\newtheorem{thm}{Theorem}[section]
\newtheorem{lemma}[thm]{Lemma}

\newtheorem{cor}[thm]{Corollary}
\newtheorem{prop}[thm]{Proposition}
\newtheorem{defn}[thm]{Definition}

\newtheorem{remark}[thm]{Remark}

\numberwithin{equation}{section}

\newcommand{\tr}{\mathrm{tr} \,}

\newcommand{\AI}{A_\infty}

\newcommand{\Z}{\mathbb{Z}}
\newcommand{\R}{\mathbb{R}}
\newcommand{\C}{\mathbb{C}}

\newcommand{\var}{\mathrm{var}}

\begin{document}
\title{Geometric models of simple Lie algebras via singularity theory}
 \author[Cho]{Cheol-Hyun Cho}
\address{Department of Mathematics\\ Postech\\Pohang\\ Republic of Korea}
\email{chocheol@postech.ac.kr}
\author[Jeong]{Wonbo Jeong}
\address{Department of Mathematics and Center for Nano Materials\\ Sogang University\\ 35 Baekbeom-ro\\ Mapo-gu\\ Seoul 04107\\ Republic of Korea}
\email{wonbo.jeong@gmail.com}
\author[Kim]{Beom-Seok Kim}
\address{Department of Mathematics\\ Postech\\Pohang\\ Republic of Korea}
\email{bskim.math@gmail.com}

\begin{abstract}
It is well-known that $ADE$ Dynkin diagrams classify both the simply-laced simple Lie algebras and simple singularities. 
We introduce a polygonal wheel in a plane for each case of $ADE$, called the Coxeter wheel.  We show that  equivalence classes of edges and spokes of a Coxeter wheel form a geometric root system isomorphic to the classical root system of the corresponding type. 
This wheel is in fact derived from the Milnor fiber of corresponding simple singularities of two variables, and the bilinear form on the geometric root system is the negative of its symmetrized Seifert form. Furthermore, we give a completely geometric definition of simple Lie algebras using arcs, Seifert form and variation operator of the singularity theory.
\end{abstract}

\maketitle
\tableofcontents

\section{Introduction}
\subsection{Dynkin diagrams, simple Lie algebras and simple singularities}

Dynkin diagrams were introduced in the theory of Lie groups and Lie algebras to describe the root system
and classify semi-simple Lie algebras over an algebraically closed field.
Killing \cite{Ki} has shown that for a given semi-simple Lie algebra  $\mathfrak{g}$, there exists an abelian subalgebra  $\mathfrak{h}$, called  {\em Cartan subalgebra}, which acts on  $\mathfrak{g}$ by the adjoint representation so that
$\mathfrak{g}$ admits an eigenspace decomposition:
$$\mathfrak{g} =\mathfrak{h} \oplus \left(\bigoplus_{\alpha \in \Phi} \mathfrak{g}_\alpha\right).$$
Here $\Phi \subset \mathfrak{h}^*$ is the set of roots so that for $\alpha \in \Phi$, $g \in  \mathfrak{g}_\alpha$ and $h \in \mathfrak{h}$, we have
$$[h, g] = \alpha(h)\cdot g.$$


A Lie algebra comes with a symmetric bilinear form, called Cartan--Killing form,  defined by  $(x,y) = \mathrm{tr(ad(x) \circ ad(y))}$. 
For $ADE$ simple Lie algebras, the Cartan--Killing forms on $\mathfrak{g}$, $\mathfrak{h}$ and $\mathfrak{h}^*$ are non-degenerate.
One can choose a simple basis of $\mathfrak{h}^*$ so that the Cartan--Killing form is given by the Cartan matrix $C$. Namely, $C = 2 I_{n \times n} - A$,
where $A=(a_{ij})$ is the adjacency matrix of the corresponding Dynkin diagram, and $a_{ij}$ is the number of edges between $i$-th and $j$-th vertices.

Simple Lie algebras are classified using the Dynkin diagrams:
there are $A_k, D_k$, $E_6$, $E_7$, and $E_8$ simple Lie algebras whose Dynkin diagrams are simply-laced (no multiple edges).
There are $B_k,C_k$, $F_4$, and $G_2$ simple Lie algebras whose Dynkin diagrams are not simply-laced.
It is well-known that non-simply-laced simple Lie algebras appear as Lie subalgebras of simply-laced ones.

Dynkin diagrams also appeared surprisingly in the classification of hypersurface singularities by Arnold \cite{Arnold72}.
In particular, simple singularities are classified by Dynkin diagrams of $A_k, D_k$, $E_6,E_7$, and $E_8$ type.
For the case of two variables, they are given by
$$x^{k+1}+y^2, x^{k-1}+xy^2, x^3+y^4, x^3+xy^3,\mbox{ and }x^3+y^5$$
and adding $z_1^2+\dots +z_{n-2}^2$ to each of the above gives the classification for $n$ variables.

Given a simple singularity, there are two different, algebraic or symplectic geometric approaches to  the Dynkin diagram.
For an algebro-geometric approach, the singular fiber $f^{-1}(0)$ of an $ADE$ singularity $f(x,y,z)$ of three variables, can be identified with the quotient $\mathbb{C}^2/G$ of a finite subgroup $G$ of $SU(2)$
(whose classification is also given by $ADE$),
and  Du Val \cite{Du} have shown that the configuration of exceptional divisors in its minimal resolution is exactly described by the corresponding Dynkin diagram. 

For a symplectic geometric approach, given an $ADE$ singularity $f$, there exist a Morsification of $f$ and the collection
of vanishing paths so that its distinguished collection of vanishing cycles in the Milnor fiber has an intersection pattern as in the
associated Dynkin diagram. Milnor fiber has an exact symplectic structure and vanishing cycles are Lagrangian submanifolds,
which can be used to define Fukaya--Seidel category. Our construction is related to the symplectic approach, but we do not use any symplectic geometry in this paper.

\subsection{Known relations between simple Lie algebras and simple singularities}
Let us briefly recall two major developments linking simple Lie algebras and simple singularities.

\subsubsection{From the theory of simple Lie algebras to simple singularities}
There are classical deep results that construct (versal deformation of) a simple singularity from a Lie algebra.
Given a Lie algebra $\mathfrak{g}$ with a cartan subalgebra $\mathfrak{h}$ and the associated Weyl group $W$, there is a map 
$$\gamma : \mathfrak{g} \to \mathfrak{h}/W$$
sending $x \in \mathfrak{g}$ to the conjugacy class of the semi-simple part of $x$. 
The fiber $\gamma^{-1}(\gamma(0))$ is called the nilpotent variety $\mathcal{N}$ of $\mathfrak{g}$, which is singular.
Slodowy \cite{Slodowy} proved that there is a simultaneous resolution of $\gamma$, generalizing a Springer resolution of $\mathcal{N}$.
Associated algebraic group $G$ acts on $\mathcal{N}$ via adjoint representation, and one can choose the transversal slice $X$ (to the $G$-orbit) at a subregular
element $y$ of $\mathcal{N}$. Grothendieck \cite{Gro} conjectured (see Brieskorn \cite{brie}, Esnault \cite{esnault} for a proof) that
$(X \cap \mathcal{N},y)$  is the corresponding simple singularity and its extention to $\gamma:(X,y) \to (\mathfrak{h}/W,0)$ is its versal deformation. We refer readers to the  survey article of L\^e and Tosun \cite{Lesimple} for more details.

\subsubsection{From simple singularity to the root system}
Let us explain how to obtain a root system associated to the three-variable $ADE$ simple singularities $f(x,y,z)$, following 
\cite{brie70}, \cite{le69}, \cite{Ga73}, \cite{Eb84}.
Let $M$ be a Milnor fiber of the singularity $f$. Then $H_2(M;\Z)$ is generated by  distinguished collection of vanishing cycles, 
which also has a symmetric bilinear form given by the intersection form. It is known that vanishing cycles can be chosen so that
the intersection form equals the negative of the standard Cartan matrix $C$.

Furthermore, monodromy of $H_2(M;\Z)$ around a single Morse critical value corresponding to the vanishing cycle $\Delta$  is given by the Picard--Lefschetz formula \cite{Pham65} (for $n$-variable case)
\begin{equation*}\label{eq:PL}
 a \mapsto a+ (-1)^{\frac{n(n+1)}{2}}(a \bullet \Delta) \Delta.
\end{equation*}
When $n=3$, this can be identified with  the reflection in the associated Weyl group of the Lie algebra.
Therefore, the monodromy group of the simple singularity $f(x,y,z)$ can be identified with the Weyl group,
and this gives the correspondence between the set of homology classes of vanishing cycles of $f$
and the roots in the corresponding simple Lie algebra.

In summary, the root system from the $ADE$ simple Lie algebras can be identified with the root system
of the corresponding simple singularity $f(x,y,z)$, given by the set of vanishing cycles and intersection form.

Given a finite root system, there is a way to construct a unique semi-simple Lie algebra up to isomorphism \cite{Ser}.
This process is algebraic and one also needs to make a choice of structure constants for the Lie bracket. 

In this paper, we will work with simple singularities of two variables. Then the intersection form on middle dimensional homology
$H_1(M;\Z)$ of the Milnor fiber $M$ is skew-symmetric, hence the correspondence with the Cartan--Killing form and Weyl group action fails to hold.
However, we will give a different definition of geometric root system and find a way to visualize them. Moreover, we will be able to
give a completely geometric definition of the corresponding Lie algebras. It turns out that  this geometric definition that we found, can be modified to work in all dimensions.

\subsection{Basic singularity theory}
We recall the notion of monodromy, variation operator, and Seifert form from  \cite{AGV2} and \cite{Mil}.
Let $f:\C^n \to \C$ be an isolated singularity at the origin. 
For simplicity, let us further assume that $f$ is a weighted homogeneous polynomial: for any $t \in \C^*$, there exist $a_1,\dots, a_n, h \in \mathbb{N}$ so that
$$f(t^{a_1}z_1,\dots, t^{a_n}z_n) = t^h f(z_1,\dots,z_n).$$
$f$ is said to have weight $(a_1,\dots, a_n;h)$. In this case, Milnor fiber $M$ can be defined by the solution set $\{f=1\}$ inside the closed ball $\overline{B}_R(0)$ of large radius $R$.
Since $f$ is non-singular away from the origin, one can define a parallel transport between fibers of $f$ along the unit circle in the base $\C$, and 
this defines a map from $M$ to itself, called the geometric monodromy $\rho$.  
The induced isomorphism $\rho_*$ on the homology group of $M$ is called the monodromy operator, or simply the monodromy.
In our case, $\rho$ can be chosen to be the weight action 
\begin{equation}\label{eq:wa}
(z_1,\dots,z_n) \mapsto \big(e^{2\pi i a_1/h} z_1, \dots, e^{2\pi i a_n/h}z_n\big).
\end{equation}

The boundary $\partial M$ is diffeomorphic to the link $\{f=0\} \cap \partial B_R(0)$ of the singularity.
There is a weight flow on the link connecting the point and its weight action image. 
By adding an untwisting near the boundary using this weight flow, we can define a geometric monodromy $\rho'$ that is identity on the boundary $\partial M$,
and isotopic to $\rho$. Milnor have shown that $M$ is homotopy equivalent to the $\mu$-bouquet of spheres $S^{n-1}$, where $\mu$ is the Milnor number of $f$.
Hence $H_{n-1}(M;\Z)$ is isomorphic to $\Z^\mu$.
\begin{thm}[{\cite[Theorem 2.2]{AGV2}}] 
The variation operator $\mathrm{var}$ is a linear isomorphism
$$ \mathrm{var}: H_{n-1}(M,\partial M;\mathbb{Z}) \to H_{n-1}(M;\mathbb{Z}): \mathrm{var} ([K]) = [\rho'(K) - K]$$
sending relative cycles to the cycles without boundary in $M$.
\end{thm}

For an isolated singularity, there is a bilinear form $\mathcal{L}$ on $H_{n-1}(M;\Z)$, called the Seifert form, which is always non-degenerate.
Seifert form is defined by a linking number, but  the following proposition provides an alternative definition.
\begin{thm}[{\cite[Theorem 2.3]{AGV2}}]\label{thm:SA}
For $a, b \in H_{n-1}(M;\Z)$,
$$\mathcal{L}(a,b) = \mathrm{var}^{-1}(a) \bullet b.$$
\end{thm}
Here $\bullet$ denotes the intersection pairing on $H_{n-1}(M, \partial M;\Z) \otimes H_{n-1}(M;\Z)$.
Since  variation operator is an isomorphism, we have the corresponding pairing on $H_{n-1}(M,\partial M;\Z)$ as well.

\subsection{Main result 1: Geometric root system}
Let $f(x,y)$ be a simple singularity  of type $\Gamma =A_k,\,D_k,\,E_6,\,E_7,E_8$.
We will explain how to construct   a  geometric Lie algebra $\mathfrak{g}_{geo}$ of type $\Gamma$.
Milnor fiber $M$ of $f(x,y)$  is a Riemann surface with boundary, and  $H_1(M;\Z)$ and $H_1(M,\partial M;\Z)$ are dual to each other via intersection pairing. We take these as geometric analogues of $\mathfrak{h}$ and $\mathfrak{h}^*$.
But the duality between them in this paper will be from the pairing in Definition \ref{defn:S} rather than the intersection pairing.
 \begin{defn}[Geometric Cartan subalgebra $\mathfrak{h}_{geo}$]
We  take  $H_1(M;\mathbb{Z})$ as Cartan subalgebra $\mathfrak{h}_{geo}$ of our geometric Lie algebra $\mathfrak{g}_{geo}$  (or $H_1(M;\mathbb{Z})\otimes \C$ for complex Lie algebra).  Therefore, roots of $\mathfrak{g}_{geo}$ will be a collection of relative cycles in $H_1(M,\partial M;\mathbb{Z})=\mathfrak{h}_{geo}^*$.
\end{defn}
\begin{remark}
This is different from the classical approach, where vanishing cycles were taken as roots. Since $\mathfrak{h}$ and $\mathfrak{h}^*$ are isomorphic to each other, 
two conventions can be regarded to be equivalent. In contrast, our setup permits the use of the variation operator itself, rather than its inverse, in defining $\mathfrak{sl}_2$-triples. Also we can regard the roots as oriented arcs in the Milnor fiber, which is essential for their geometric visualization.
\end{remark}

The intersection pairing on $H_1(M;\Z)$ is skew-symmetric, but the desired Cartan--Killing forms are symmetric. 
We choose the pairing to be the negative of the symmetrized  Seifert form
on $\mathfrak{h}_{geo}$ or $\mathfrak{h}_{geo}^*$.
\begin{defn}[Geometric pairing on $\mathfrak{h}_{geo}^*$]\label{defn:S}
For $\alpha,\beta \in H_1(M,\partial M;\mathbb{Z})$, we define
$$(\alpha,\beta):= \mathrm{var}(\alpha) \bullet \beta  + \mathrm{var}(\beta)  \bullet  \alpha.$$
This equals $-\mathcal{L}^t - \mathcal{L}$ on $\mathfrak{h}_{geo}^*$,  a {\em negative of the symmetrized Seifert pairing} on $H_1(M,\partial M;\mathbb{Z})$.
\end{defn}
\begin{remark}\label{re:fu}
There is a categorification of this pairing.  Namely, our pairing is the symmetrized Euler pairing of
the Fukaya category $\mathcal{F}_\rho$ of an isolated singularity being developed in \cite{BCCJ}.
Also note that Fukaya--Seidel category gives a categorification of Seifert form somewhat indirectly.
In \cite{BCCJ}, only the morphism spaces of the category are defined, but their $\AI$-structures will be defined
in the forthcoming work following the construction in \cite{CCJ}.
\end{remark}

\begin{thm}\label{geometric roots}
    Let $f(x,y)$ be a simple curve singularity of type $\Gamma$.
    Then 
    $$\Phi_\Gamma := \{\alpha\in H_1(M,\partial M;\Z) \;|\;  (\alpha,\alpha)=2 \} $$ 
    satisfies the axioms of a root system,  and we call it the geometric root system of type $\Gamma$.
    It is isomorphic to  a classical root system of type $\Gamma$.
\end{thm}
In particular, the geometric pairing can be identified with the classical Cartan--Killing form.
Moreover, we have the following geometric realizations of roots.
\begin{thm}\label{root_equiv}
    Let $M$ be the Milnor fiber of a simple curve singularity of type $\Gamma$.
    Then the following subsets of $H_1(M,\partial M;\Z)$, which are denoted by $\Phi^{(1)}_\Gamma,\ldots,\Phi^{(4)}_\Gamma$,  are equal to each other.
    \begin{itemize}
        \item[(i)]  $ \{\alpha\in H_1(M,\partial M;\Z) \;|\;  (\alpha,\alpha)=2 \} $, namely $\Phi_\Gamma$.
        \item[(ii)] $\{\alpha\in H_1(M,\partial M;\Z) \;|\; (\alpha,\alpha) \le 2 \}$.
        \item[(iii)] $\{\alpha\in H_1(M,\partial M;\Z) \;|$ $\alpha$ has an embedded connected arc representative $L$ such that  $L \cap \rho(L) = \emptyset$ $\}$.
        \item[(iv)] $\{\alpha\in H_1(M,\partial M;\Z) \;|\;$ $\var(\alpha)$ is a member of a distinguished collection of vanishing cycles of $f$ $\} $.
    \end{itemize}
\end{thm}

(iii) and (iv) provide two geometric characterization of roots.
First, (iii) provides an easy geometric way to determine whether $\alpha$ is a root or not using the monodromy $\rho$.
We remark that for any isolated curve singularity $f(x,y)$, geometric monodromy $\rho'$ can be chosen so that 
$\rho'$ does not have any fixed point away from the boundary.  

(iv) demonstrates how the conventional notion that vanishing cycles are roots is also connected to the current approach.
But let us point out that we are not taking all vanishing cycles, but rather the proper subset given by
the members of distinguished collection of vanishing cycles.

This difference is largely because we are working with  2-variable rather 3-variable singularities (c.f. \cite{Her24}).
Roughly speaking,  a  vanishing cycle  from an embedded vanishing path belongs to $\Phi^{(4)}_\Gamma$.
But general vanishing cycles may come from immersed vanishing paths (due to the Picard--Lefschetz transformation).
To distinguish them, we call the former {\em em vanishing cycle}.

In \cite{GZ80}, it is proven that every vanishing cycle of a 3-variable simple singularity is an em vanishing cycle,
but this fails to hold for 2-variables. In this case the set of all vanishing cycles are much bigger than that of em vanishing cycles.

The geometric condition $L \cap \rho(L) = \emptyset$ (iii) has its origin in  \cite{BCCJP}, where the notion of vanishing arcs and arcsets are introduced.
Under the assumption that it is not A, D type, and $g >4$, this condition was sufficient to guarantee the arc to be a geometric vanishing arc (not necessarily em type).  
Since this assumption is false in our cases, such a finding is not applicable here.

%
%
%

\subsection{Main result 2: Coxeter wheels and geometric roots}
The above theorem allows us to realize  roots as oriented embedded connected arcs in the Milnor fiber $M$ of a given simple singularity $f(x,y)$.
On the Milnor fiber $M$, a Riemann surface with boundary, we realize all vanishing arcs that correspond to the roots.
It is known that for $A_k, D_k, E_6, E_7$, and $E_8$ type, the number of (geometric) roots are 
$$ k(k+1), 2k(k-1), 72, 126,\mbox{ and } 240.$$
\begin{figure}[h]
\begin{subfigure}{0.32\textwidth}
\includegraphics[scale=0.22]{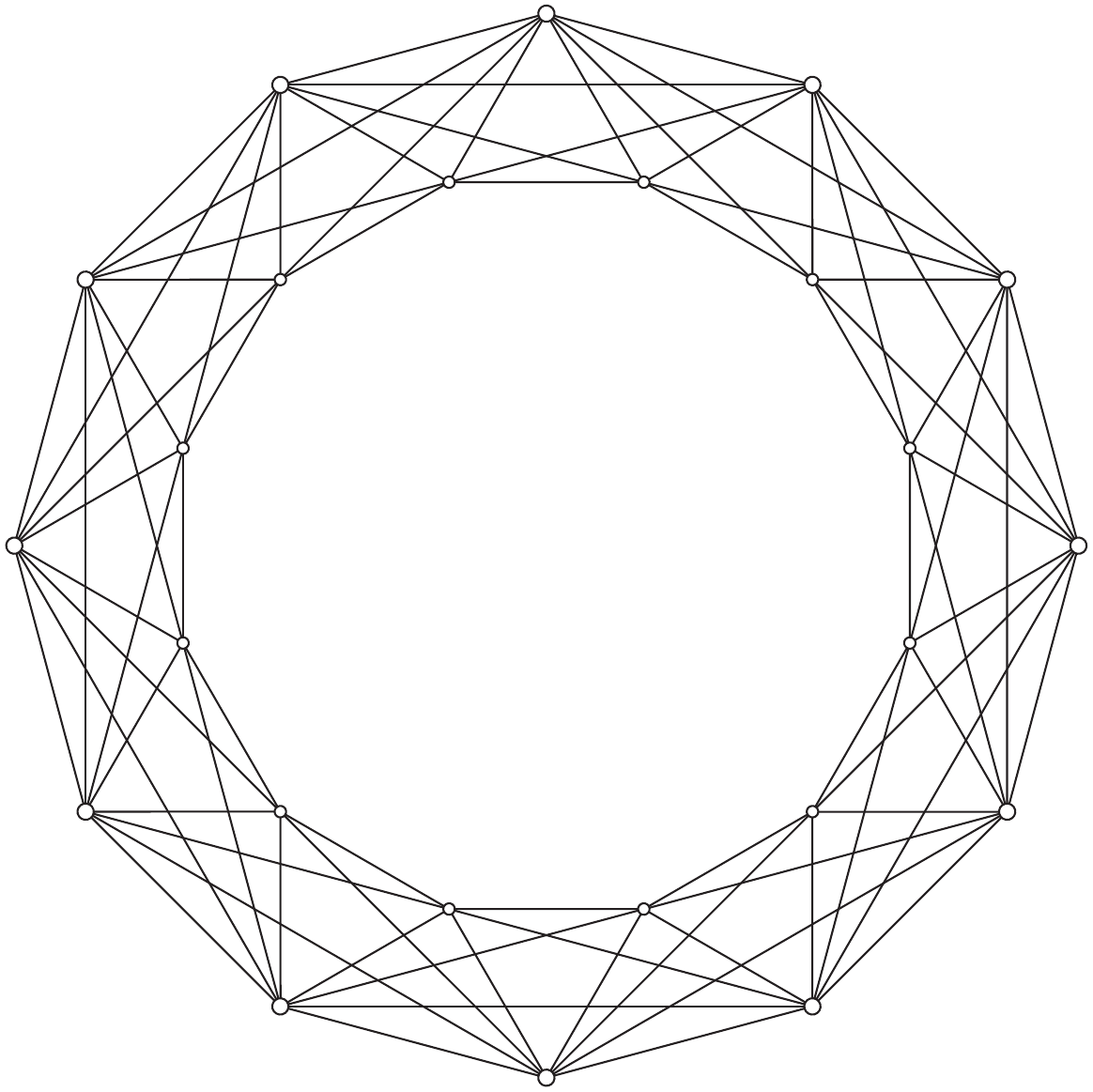}
\centering
\caption{Coxeter wheel for $E_{6}$.}
\end{subfigure}
\begin{subfigure}{0.32\textwidth}
\includegraphics[scale=0.23]{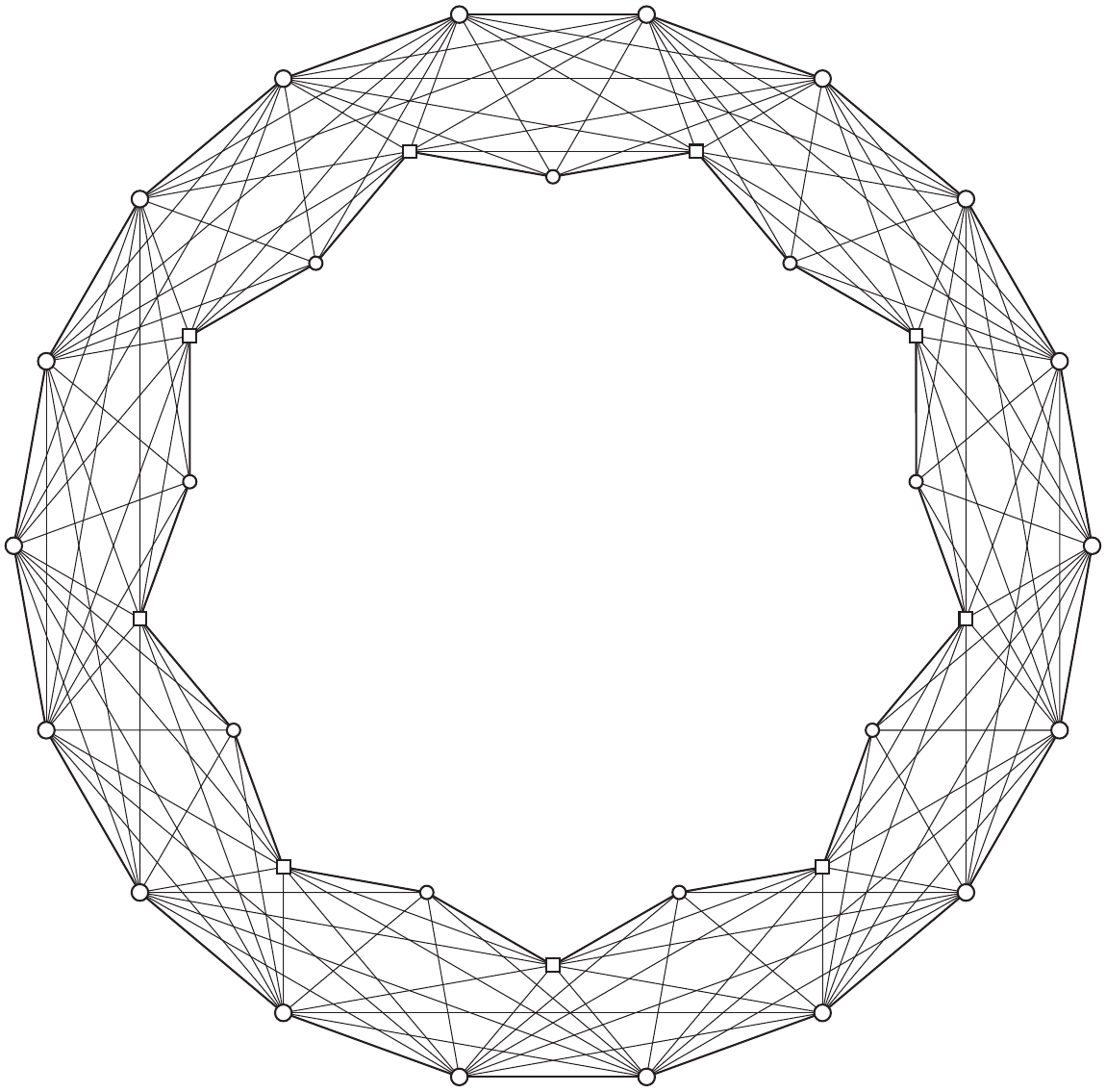}
\centering
\caption{Coxeter wheel for $E_{7}$.}
\end{subfigure}
\begin{subfigure}{0.32\textwidth}
\includegraphics[scale=0.24]{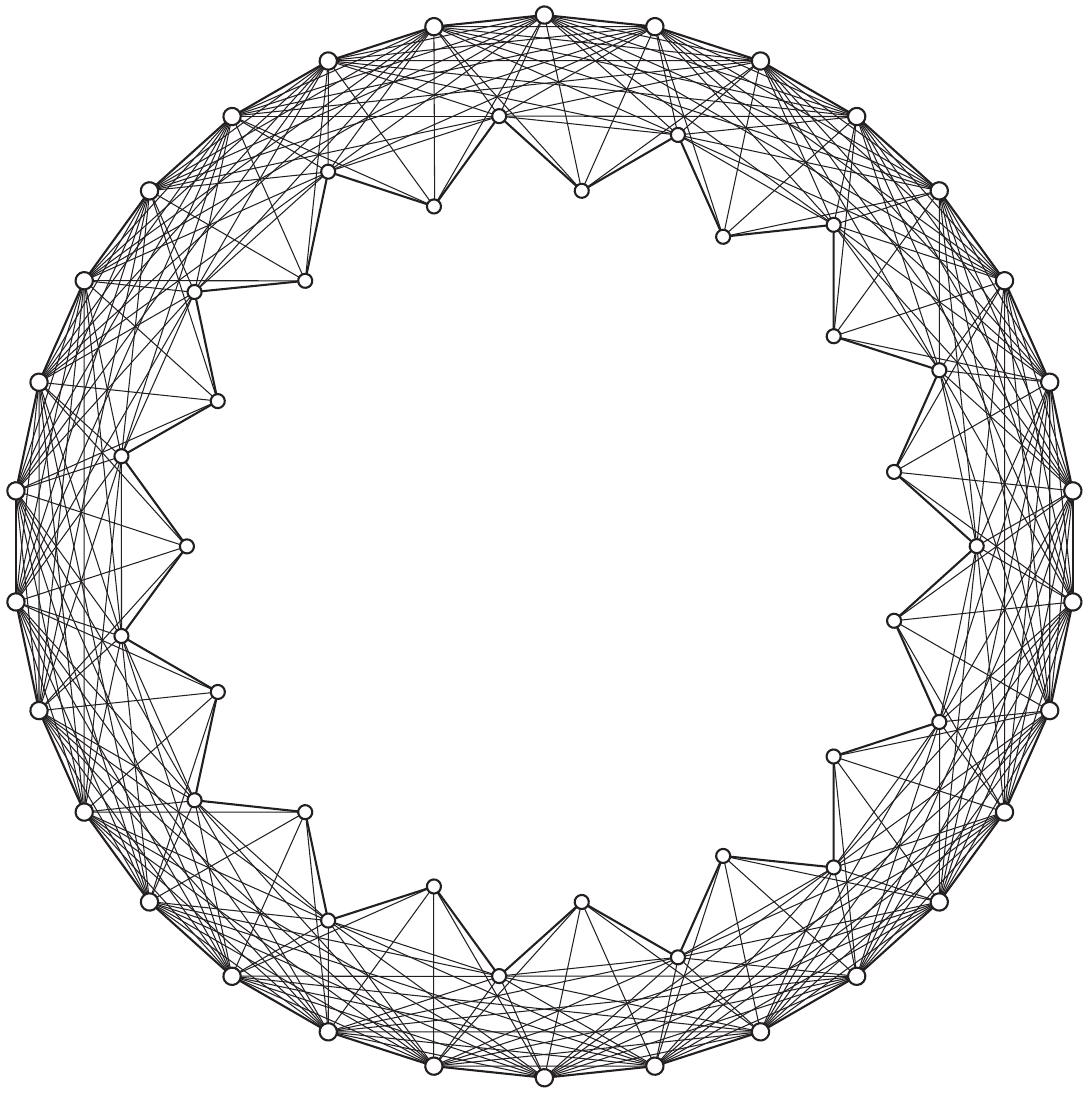}
\centering
\caption{Coxeter wheel for $E_{8}$.}
\end{subfigure}
\caption{ Geometric roots and the Coxeter wheels for $E$-types.}
\label{fig:E678Coxeter}
\end{figure}
Now, we introduce the Coxeter wheel, which is made from copies of polygonal pieces of a Milnor fiber.
Most important property of the wheel is that  geometric roots (vanishing arcs) are given by the equivalence classes of oriented edges and line segments (called spokes) in the wheel.
See Figure \ref{fig:E678Coxeter} for $E$-type wheels.

$A_k$-wheel is a regular $(k+1)$-gon whose vertices are punctures. $D_k$-wheel is a regular $2(k-1)$-gon whose vertices and the center point are punctures (see Figure \ref{fig:ADroots}).

For a wheel $W \subset \mathbb{R}^2$ of type $\Gamma$ , we have a natural map $\pi: W \to M$ to the Milnor fiber of type $\Gamma$, where
$\pi$ is an inclusion for $A_k$, $\pi$ is a surjection given by identifying opposite edges for $D_k$, and $\pi$ is very close to a covering map for $E$-type.
\begin{defn}
Given two edges or spokes in $W$, we will say that they are {\em equivalent } if their images under $\pi$ are in the same homology class of $H_1(M,\partial M;\Z)$.
\end{defn}

Here are the properties of $E$-type wheels.
\begin{thm}\label{Etype wheel}
Let $M$ be the Milnor fiber of a simple singularity of  $E$ type. There exists a polygonal wheel $W \subset  \mathbb{R}^2$ whose vertices are punctures, satisfying the following properties.
\begin{enumerate}
\item There exists a surjective map to the Milnor fiber $\pi: W \to M$.
\item There exists a decomposition of $M$ into polygons $\{M_i\}$ such that for each $M_i$,
two dimensional components of $\pi^{-1}(M_i)$ are polygons in $\mathbb{R}^2$ all of which can be identified with each other by Euclidean translation.
\item Geometric monodromy $\rho$ of Milnor fiber $M$ lifts to the rotation of $W$ given by half rotation and one more tilting in the counter-clockwise direction.
\end{enumerate}
\end{thm}
\begin{remark}\label{rmk:transl}
The property (2) makes $E_k$-Milnor fibers translation surfaces. So are $A_k$ and $D_k$-Milnor fibers.
In particular, they carry flat metrics. We learned that A'campo and Portilla Cuadrado have shown in their work in progress that all Milnor fibers of curve singularities carry flat metrics.
\end{remark}
In the case of $D_k$-wheel, geometric monodromy $\rho$ acts as in $E$-type. For $A_k$-wheel, geometric monodromy $\rho$ behaves differently. For details of geometric monodromy, see Section \ref{sec:2} and \ref{sec:3}. For descriptions of the monodromy operator $\rho_*$ of each type, see Section \ref{sec:monodromy}.

Now, we call them as the Coxeter wheels, because of the surprising relation to the Coxeter plane.
Recall that given a Weyl group of type $\Gamma$ and the Coxeter element $c$ 
with order $h$, there is a two dimensional subspace of $\mathfrak{h}^*$, called Coxeter plane on which $c$ acts
by the rotation of $\frac{2\pi}{h}$. Orthogonal projections of roots in $\mathfrak{h}^*$ to the Coxeter plane are called
root projections.

It turns out that our wheel can be placed on the Coxeter plane so that vertices of the wheel lie exactly on root projections
(see Figure \ref{fig:E67Coxeter} and Figure \ref{fig:E8Coxeter}).  It is not a coincidence.
\begin{thm}
Equivalence classes of oriented edges and spokes on the wheel of type $\Gamma$(under the map $\pi$) can be identified with geometric roots $\Phi_\Gamma$.
If an oriented edge or spoke is equivalent to another oriented edge or spoke in the wheel, they have the same length and direction.
\end{thm}

For the case of $A_2$ root system, $A_2$-Milnor fiber is a regular hexagon with 3 punctures (with opposite sides identified) as in Figure \ref{fig:A2roots}.  $A_2$-wheel is a regular triangle whose vertices are these punctures.
There are six geometric roots given by $\alpha_1,\alpha_2,\alpha_3$ and their orientation reversals (whose homology classes are $-\alpha_1,-\alpha_2,-\alpha_3$).
    \begin{figure}[h]
    \centering
    \begin{subfigure}{0.32\textwidth}
    \includegraphics[scale=0.7]{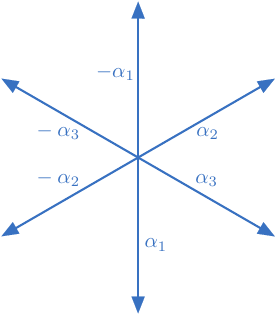}
    \centering
    \caption{Classical roots.}
    \end{subfigure}
    \begin{subfigure}{0.32\textwidth}
    \includegraphics[scale=0.7]{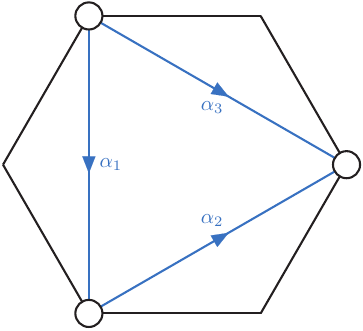}
    \centering
    \caption{Geometric roots.}
    \end{subfigure}
  \caption{$A_2$ root systems from Lie theory and geometry.}
        \label{fig:A2roots}
    \end{figure}
    
 In this case, the classical $A_2$ root system lies in $\mathbb{R}^2$, and the identification is almost immediate.
 
\begin{remark}\label{multi}
The converse of the last statement of the theorem holds for $A_{2k}$, $E_7$ and $E_8$ but not for others types. In fact the same phenomenon happens in Coxeter plane too.
The projection of roots to the root projections in the Coxeter plane are one to one only for the above types and can have multiplicities in general. 
One can check that regarding oriented edges/spokes as vectors in $\mathbb{R}^2$ have exactly the same multiplicities.
\end{remark}

Let us explain why it is not  a coincidence. On a Coxeter plane, two root projections are joined by an edge if they come from neighboring roots. In this case, the difference of two neighboring roots (which is a root also) project to the difference vector of two root projections. Our edge/spoke as a geometric root
corresponds to this root of the difference.

In the case of $A_k$ and $D_k$-wheels, they can be also related to Coxeter planes, but we will not do so because the wheel itself is simple enough.
For the $A_k$-wheel given by the regular $(k+1)$-gon, number of oriented edges and diagonals between the vertices are exactly $2 {k \choose 2} = k(k+1)$, they are  exactly the geometric roots for $A_k$.

For the $D_k$-wheel, opposite edges of $2(k-1)$-gon are identified, and from this, two parallel line segments that do not meet the center puncture and having  the same orientation  become homologous to each other. Number of such equivalence classes are ${2k-2 \choose 2} - (k-1)$. Number of oriented spokes  that meet the center is $2 \times 2(k-1)$.
In total, we have $2k(k-1)$ oriented edges/spokes, which equals the number of geometric roots.

\begin{figure}[h]
\begin{subfigure}{0.47\textwidth}
\includegraphics[scale=0.7]{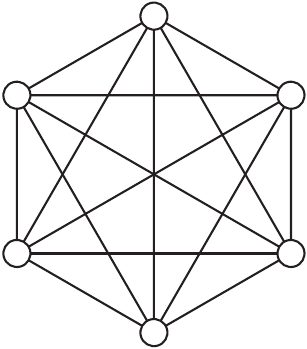}
\centering
\caption{Geometric roots in the $A_{5}$-wheel.}
\end{subfigure}
\begin{subfigure}{0.47\textwidth}
\includegraphics[scale=0.7]{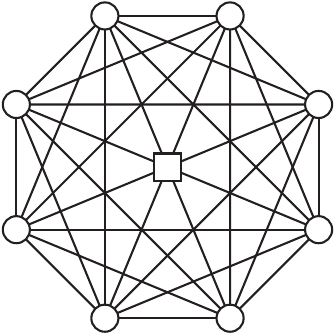}
\centering
\caption{Geometric roots in the $D_{5}$-wheel.}
\end{subfigure}
\centering
\caption{Examples of geometric roots.}
\label{fig:ADroots}
\end{figure}


The construction of $E$-type wheels and the equivalence relations between edges/spokes are more subtle. 
The polygonal description for $E$-type Milnor fiber of a simple singularity was given in \cite{CCJ2}.  
We will find a systematic way of gluing copies of these polygonal pieces to form the desired wheel.

\begin{remark}\label{qui}
These geometric roots are also related to other classification problems. 
One is related to Berglund--H\"ubsch  homological mirror symmetry of invertible curve singularities. In \cite{CCJ2}, Lagrangian arcs in the Milnor fiber quotient that correspond to indecomposable matrix factorzations of $ADE$ polynomial $f(x,y)$ were found. Their lifts in the Milnor fiber correspond to these geometric roots and in fact this is where this paper is originated from. 

Also recall that roots are dimension vectors of  indecomposable quiver representations of $ADE$ quivers.
The first two authors plan to construct an $A_\infty$-functor that sends geometric roots to indecomposable quiver representations in the work in preparation, using the Fukaya category being developed in \cite{CCJ}, \cite{CCJ2} and \cite{BCCJ}.
\end{remark}
\begin{remark}
How to define geometric roots beyond $ADE$ is a very interesting question. Saito has introduced a generalized root system of vanishing cycles and we refer readers to \cite{Sai} for more details and historical accounts. 
\end{remark}

\subsection{Main result 3: Geometric construction of Lie algebra}

Now we are ready to define the Lie algebra for each simple singularity.
It turns out that they admit a completely geometric definition, only using geometric roots, variation operator and  Seifert form.

We start with an  example of $A_1$-singularity and construct $\mathfrak{sl}_2$ Lie algebra to illustrate the idea.
Recall that $\mathfrak{sl}_2$ has a basis given by the triple  $\big(e,f,h\big)$, satisfying $[e,f] = -h$. 
On the other hand, the Milnor fiber $M$ of $A_1$-singularity $x^2+y^2$ is a cylinder (or  $T^*S^1$).
Note that $H_1(M,\partial M;\Z)$ is generated by an oriented arc (a cotangent fiber) $\alpha$ connecting two boundaries, and $H_1(M;\Z)$ by the vanishing circle $V$ (zero section).
In this case, the monodromy of $A_1$-singularity is given by the right-handed Dehn twist and the variation operator takes the relative cycle $\alpha$ to the compact cycle $V$
(this determines the orientation of $V$).

\begin{figure}[h]
\includegraphics[scale=0.8]{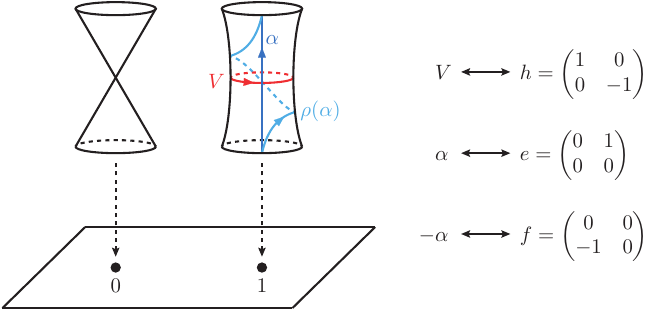}
\centering
\caption{$\mathfrak{sl}_2$ triple for $A_1$-singularity.}
\label{fig:A1ex}
\end{figure}

We choose the geometric triple corresponding to  $\big(e,f,h\big)$, to be 
$$\big(g_{\alpha},g_{-\alpha}, V= \mathrm{var}(\alpha)\big).$$
Here $g_\alpha, g_{-\alpha}$ are formal generators corresponding to $\alpha,-\alpha$, respectively.
We define the Lie bracket as
$$[g_\alpha, g_{-\alpha}]  = - \mathrm{var}(\alpha) \in \mathfrak{h}_{geo}.$$
As an analogue of $[h,e] = 2e, [h,f]=-2f$, we define the Lie bracket as
$$[V,g_\alpha] := ( V \bullet (\alpha + \rho_*(\alpha)) ) g_\alpha = 2 g_\alpha.$$
This definition makes  $\big(g_{\alpha},g_{-\alpha},\mathrm{var}(\alpha)\big)$ the $\mathfrak{sl}_2$-triple $\big(e,f,h\big)$.

Now we give the geometric definition in the general case so that for any geometric root $\beta$, we have a 
corresponding $\mathfrak{sl}_2$ triple $\big(g_{\beta},g_{-\beta},\mathrm{var}(\beta)\big)$.

\begin{defn}\label{Lie bracket}
Given a simple curve singularity $f$ of type $\Gamma$,
we define the geometric Lie algebra $\mathfrak{g}_{\textrm{geo}}$ as a $\Z$-module 
$$\mathfrak{g}_{\textrm{geo}}:=\mathfrak{h}_{geo} \oplus \bigoplus_{\alpha \in \Phi_{\Gamma}} \Z \langle g_\alpha \rangle$$
where $\Phi_{\Gamma}$ is the set of geometric roots of type $\Gamma$, and $\mathfrak{h}_{geo}:=H_1(M;\Z)$.
Its Lie bracket is defined as follows.
\begin{enumerate}

\item  For $h_1,h_2 \in H_1(M;\Z)$, 
$$[h_1,h_2]=0.$$

\item For $h \in H_1(M;\Z), \alpha \in \Phi_{\Gamma}$,
$$[h, g_\alpha] = -[g_\alpha,h] = ( h \bullet \alpha + h \bullet \rho_*(\alpha)) g_\alpha.$$
In particular, 
\begin{equation} \label{bracket_identity}
    [\mathrm{var}(\alpha), g_\beta] = (\alpha,\beta) g_\beta
\end{equation}
where  $(\alpha,\beta)$ is the negative symmetrized Seifert form in Definition \ref{defn:S}.

\item For each geometric root $\alpha \in \Phi_\Gamma$,  
$$[g_\alpha, g_{-\alpha}] = - \mathrm{var}(\alpha) \in H_1(M;\Z).$$

\item For any two roots $\alpha, \beta \in \Phi_\Gamma$ such that $\alpha  \neq  - \beta$,  we set
$$[g_\alpha, g_\beta] = \begin{cases}  N_{\alpha,\beta} g_{\alpha+\beta} & \textrm{if } \; \alpha + \beta \in \Phi_\Gamma,
\\ 0 & \textrm{if } \;\alpha + \beta \notin \Phi_\Gamma. \end{cases}$$
Here the sign $N_{\alpha,\beta} \in \{+1,-1\}$ is defined as
$$N_{\alpha,\beta} = (-1)^{\mathrm{var}(\beta) \bullet \alpha}.$$
\end{enumerate}
\end{defn}

In part (2) of Definition \ref{Lie bracket}, the second equation \eqref{bracket_identity} is equivalent to the preceding equation with the substitution $h=\var(\alpha)$, due to the identity (see \cite[Corollary of Lemma 1.1]{AGV2})
$$\big( \rho_*(\alpha) \bullet \var(\beta) \big) + \big( \var(\alpha) \bullet \beta \big) = 0.$$
Furthermore, \eqref{bracket_identity} also suffices for the definition of Lie bracket, since $\var(\Phi_\Gamma)$ spans $H_1(M;\Z)$.

\begin{thm}
For a simple curve singularity $f$ of type $\Gamma$, the above definition indeed defines the Lie algebra $\mathfrak{g}_{\textrm{geo}}$ that is isomorphic to the 
 simple Lie algebra of the type $\Gamma$.
\end{thm}
This shows that geometry has a preferred sign function $N_{\alpha,\beta}$ for the  Lie bracket of simple Lie algebras.


\begin{remark}
There are well-known categorical constructions of Lie agebras by Peng-Xiao
\cite{PX90}, \cite{PX00}  via Hall algebras based on the works of Ringel \cite{Ri90}, \cite{Ri90h}.
It would be very interesting to find a connection to our work. 
As an evidence, in the case of $A$, all Lie brackets of type  $[g_\alpha, g_\beta] = g_\gamma$, the arc for $\gamma$ is the unique extension of arc for $\alpha$ by arc for $\beta$ in the Fukaya category $\mathcal{F}_\rho$.
\end{remark}

\subsection{Extension to $n$-variable cases}
The $n$-variable simple singularities can be obtained by stabilizations of simple curve singularities. We explain how to extend our construction to these cases.
$$\Tilde{f}(x,y,z_1,\dots,z_{n-2}):= f(x,y)+z_1^2+\dots + z_{n-2}^2.$$

Recall that given two isolated singularities $g_1:(\C^{n_1},0)\rightarrow (\C,0)$ and $g_2:(\C^{n_2},0)\rightarrow (\C,0)$, 
their Thom--Sebastiani sum  $g_1 \oplus g_2: (\C^{n_1+n_2},0)\rightarrow (\C,0)$ is defined by
$(g_1 \oplus g_2)(\textbf{z}_1,\textbf{z}_2) = g_1(\textbf{z}_1) + g_2(\textbf{z}_2)$. 
Thom--Sebastiani theorem \cite{ST71} says that the join $M_{g_1} * M_{g_2}$ of two Milnor fibers can be embedded into the Milnor fiber of the Thom--Sebastiani sum, and the inclusion map $\iota:M_{g_1} * M_{g_2} \rightarrow M_{g_1 \oplus g_2}$ is a homotopy equivalence.
Hence, we get an isomorphism
\[
\iota_*: \widetilde{H}_{n_1-1}(M_{g_1};\Z) \otimes \widetilde{H}_{n_2-1}(M_{g_2};\Z) \xrightarrow{\cong} \widetilde{H}_{n_1+n_2-1}(M_{g_1\oplus g_2};\Z).
\]

In addition, there is a result of P. Deligne which determines the Seifert form on the Milnor fiber of a Thom--Sebastiani sum: for $a_1,b_1\in H_{n_1-1}(M_{g_1};\Z)$ and $a_2,b_2\in H_{n_2-1}(M_{g_2};\Z)$,
\[
\mathcal{L}_{g_1 \oplus g_2}\big(\iota_*(a_1 \otimes a_2),\iota_*(b_1 \otimes  b_2)\big) = (-1)^{n_1 n_2} \mathcal{L}_{g_1}(a_1,b_1) \mathcal{L}_{g_2}(a_2,b_2).
\]
See page 57 of \cite{AGV2} (or \cite{Sak74}) for the proof of it.
Consider the one-variable singularity $z \mapsto z^2$ whose Milnor fiber $M_{z^2}$ is composed with two points. For the generator $\Delta\in \widetilde{H}_0(M_{z^2};\Z)$, we have $\mathcal{L}(\Delta,\Delta) = -1$ from the Picard--Lefschetz formula. In addition, $\iota_*(\cdot \otimes \Delta):\widetilde{H}_{n_1-1}(M_{g_1};\Z) \rightarrow \widetilde{H}_{n_1}(M_{g_1 \oplus z^2};\Z)$ turns out to be an isomorphism in this case.
Thus, the Seifert form of the singularity stabilized by $z^2$ is given by
\[
\mathcal{L}_{g_1 \oplus z^2}\big(\Tilde{a_1},\Tilde{b_1}\big) = (-1)^{n_1} \mathcal{L}_{g_1}(a_1,b_1)
\]
where $\Tilde{a}_1,\Tilde{b}_1$ denotes the image of $a_1,b_1$ under the isomorphism $\iota_*(\cdot \otimes \Delta)$.
Taking the stabilization $n-2$ times to $f$, we obtain $\Tilde{f}$. Above discussion leads us to the following definition.
\begin{defn}
For the $n$-variable simple singularity $\Tilde{f}$, we define the pairing $(\cdot,\cdot)_n$ on $H_{n-1}(M,\partial M;\Z)$ by 
$$(\cdot,\cdot)_n:=(-1)^{n(n+1)/{2}}\big(\mathcal{L}^t+\mathcal{L}\big).$$
\end{defn}
Now, it is straightforward to verify the following theorem.
\begin{thm}
The lattice $\big(H_{n-1}(M_{\Tilde{f}},\partial M_{\Tilde{f}};\Z),(\cdot,\cdot)_n \big)$ is identical to the lattice $\big(H_1(M_f,\partial M_f;\Z),(\cdot,\cdot)_2 \big)$.
In the definition of the Lie bracket (Definition \ref{Lie bracket}), we only replace the bilinear form $(\cdot,\cdot)$ by $(\cdot,\cdot)_n$.
This generalizes the construction of the Lie algebra to the case of $n$-variables $\Tilde{f}$. Namely,  Theorem \ref{thm:Lie bracket} still holds.
\end{thm}

%
%

\subsection{Main result 4: Description of Weyl group reflections and the Coxeter element}
Our model also provides a geometric interpretation of the Weyl group action associated with the geometric root system $\Phi_\Gamma$.
As an application, we present a geometric realization of a theorem of Kostant concerning the Coxeter element.

We begin by observing that the correspondence between Weyl group reflections and Picard–Lefschetz transformations, valid in the setting of three variables, does not hold in the case of two variables.
Hence, the Weyl group associated with $\Phi_\Gamma$ does not coincide with the monodromy group arising from the simple curve singularity of the same type.
Rather, it admits a combinatorial description.
\begin{thm}
Weyl group reflection $s_\alpha$ corresponding to the geometric root $\alpha$ can be described as a transformation that flips the edges and spokes of the Coxeter wheel $W$ involved with $\alpha$.
\end{thm}
See Section \ref{sec:9} for a detailed explanation of the flipping rule.

A Weyl group has a distinguished element $c$, called Coxeter element, obtained by the product of reflections corresponding to simple roots.
Its conjugacy class is independent of the choice of simple roots and ordering of the product. 
Its order $h$ is called the Coxeter number.
The following theorem concerning the action of Coxeter element on root system was given in \cite{Kos} (Kostant refer it to Coleman \cite{Cole}).
\begin{thm}[{\cite[Corollary 8.2]{Kos}}]\label{Kostant_original}
Coxeter element $c$ acts on the set of roots $\Phi$ with $k$ orbits, each containing $h$ element.
\end{thm}

Here, $k$ denotes the dimension of the ambient Euclidean space of the root system.
In our framework, the geometric realization of this theorem is given as follows.

Recall from  Proposition 2.4 of Hubery--Krause \cite{HKra} that 
Coxeter element equals the negative of the monodromy. This equals $\overline{\rho}_*$, which is  
the orientation reversal composed with $\rho_*$.
\begin{thm}
The group $\langle \;\overline{\rho}_* \rangle$ acts freely on $\Phi_\Gamma$.
Moreover, this action  decomposes $\Phi_\Gamma$ into exactly $k$ orbits, each containing $h$ elements.
\end{thm}

To finish, we remark on the relation to the Coxeter plane.
On the Coxeter plane, the Coxeter element $c$ acts on root projections by the rotation of $\tfrac{2\pi}{h}$.
But on Coxeter wheels (except $A$-type), monodromy $\rho_*$ rotates the wheel by $\tfrac{2\pi}{h} + \pi$.
Now, orientation reversal rotates each geometric root as a vector by $\pi$, and hence $\;\overline{\rho}_*$ rotates each geometric root by an angle of $\tfrac{2\pi}{h}$.

But orientaiton reversal is not always realized by the overall $\pi$-rotation of the wheel, and hence
the action of $\;\overline{\rho}_*$ is not given by a uniform cyclic rotation of the wheel. 

For example, for outer oriented edges or their sums, $\overline{\rho}_*$ acts by tilting the wheel counter-clockwise once. 
But for spokes or inner edges, it tilts differently.  This is not a contradiction because root projections in the Coxeter plane correspond  to vectors in our wheel not the vertices.

%

\subsection{On non-simply-laced simple Lie algebras}
We give a brief discussion for non-simply-laced
simple Lie algebras $B_k, C_k, F_4$, and $G_2$. 
It is well-known that these Lie algebras can be embedded into simply-laced ones (into $D_{k+1}, A_{2k-1}, E_6$, and $D_4$, respectively) as Lie subalgebras.
Hence, our construction gives a geometric model for these Lie algebras in an indirect way.

This process of embedding can be described in terms of folding of root systems (see \cite{Gin} page 47).
Consider two set of simple roots $\mathcal{A}, \mathcal{A}'$ with Cartan matrices $C, C'$. Then a projection $\pi: \mathcal{A}' \to \mathcal{A}$ is called a folding
if it satisfies the following two properties.
\begin{enumerate}
\item $C_{\alpha'\beta'}=0$ if $\pi(\alpha') =\pi(\beta')$ and $\alpha'\neq\beta'$.
\item $C_{\alpha\beta}= \sum_{\alpha' \in \pi^{-1}(\alpha)} C_{\alpha'\beta'}$ if $\pi(\beta')= \beta$.
\end{enumerate}
This induces an embedding (in the reverse direction) of Chevalley generators $\{h_{\alpha}, e_{\alpha}, e_{-\alpha}\}$  into  $\{h_{\alpha'}', e_{\alpha'}', e_{-\alpha'}'\}$ given by
$$  h_{\alpha} \mapsto \sum_{\alpha' \in \pi^{-1}(\alpha)} h_{\alpha'}', \quad e_{\alpha} \mapsto \sum_{\alpha' \in \pi^{-1}(\alpha)} e_{\alpha'}', \quad e_{-\alpha} \mapsto \sum_{\alpha' \in \pi^{-1}(\alpha)} e_{-\alpha'}' .$$

One may ask if the projection $\pi$ has a geometric meaning in our cases.
For the folding of $D_{k+1}$ into $B_k$, two simple roots $\alpha_1$ and $\alpha_2$   in Figure \ref{fig:simple} (C) are folded (i.e., have the same image under $\pi$).
Note that these two simple roots are different geometric roots, but are the same as vectors on $\R^2$.
For the folding of $E_6$ into $F_4$,  two roots $\alpha_1$ and $\alpha_3$ are folded and two roots $\alpha_2$ and $\alpha_4$ are folded as well (see  Figure \ref{fig:simple} (E)).
Again, $\alpha_1$ and $\alpha_3$ define the same vectors and so does $\alpha_2$ and $\alpha_4$.
Similarly, for the folding of $D_4$ into $G_2$, and for the folding of $A_{2k-1}$ into $C_k$, we can choose simple roots so that folded roots define
the same vectors. 
As in the remark \ref{multi}, projection of geometric roots to vectors corresponds to the root projections in the Coxeter plane.
In particular, the root projections have multiplicities in $A_{2k-1}, D_{k}$, and $E_6$-cases, and these are exactly the cases that we have foldings.
Also, one can check that the Coxeter planes for $\mathfrak{g}$ and $\mathfrak{g}'$ are isomorphic if  $\mathfrak{g}$ and $\mathfrak{g}'$ are related by a folding.

After the paper was put on arXiv, we have learned that there are
related earlier works by \cite{QZ25} and \cite{CQZ24} in a different context.
The authors studied total stability conditions of quiver representations and
found a combinatorial polygonal model relating indecomposable quiver representations with roots.

More precisely, in \cite{QZ25}, Qiu--Zhang defined the notion of a $h$-gon: a polygon $P$ in $\C$
and a vector $z_i \in \C$ for each boundary edge of $P$ and it was
required that $z_i$'s satisfy certain linear relations. Then they
defined the notion of stable $h$-gons and showed that these correspond
to total stability conditions on Dynkin cases. For a symmetric stability condition (called a Gepner point), these
$z_i$'s were identified with root projections on the Coxeter plane. 

There are similarities and differences between their combinatorial model and ours wheels.
Outer boundary of our $E_8$-wheel agrees with their symmetric
$h$-gon, but boundaries of our $E_6,E_7$-wheels do not.
Inside the $h$-gons, they choose two convex overlapping $h/2$-gons, and
consider the edges between some of the vertices. Thus the interiors
of $h$-gons and our wheels are different.

Another major difference is the multiplicity of vectors. The edges in
the stable $h$-gon are regarded as vectors in $\C$ and polygons represent
which sums of vectors vanish. But recall that some of the distinct arcs in our Coxeter wheel can be the same as vectors
as explained in the Remark \ref{multi}. 

The difference comes from the fact that our wheels are made of Milnor fibers and hence our edges/spokes are
actual arcs in the Milnor fiber, whereas the notion of stable $h$-gon itself is rather
combinatorially defined and interiors of their polygons represent relation between vectors.
For example, their polygon for $D_k$ has two
interior punctures, whereas our $D_k$-wheel has only one puncture;
our geometric roots $\alpha_1,\alpha_2$ are the same as vectors in
$\C$ but are different as geometric roots.

In \cite{CQZ24}, Chang--Qiu--Zhang further gave a geometric model of quiver representations based on these constructions and settled Reineke's conjecture. This should be related to our paper by the Remark \ref{qui} (see also Subsection \ref{ssec:simple}).

\subsection{Organization of the paper}
In Section \ref{sec:2}, we recall a combinatorial description of Milnor fibers, following the exposition in \cite{CCJ2}.
In Section \ref{sec:3}, we present a detailed construction of the Coxeter wheel.
In Section \ref{sec:4}, we define the geometric root system and verify that it satisfies the axioms of a root system.
In Section \ref{sec:monodromy}, we show that the geometric roots are realized by equivalence classes of edges and spokes on the Coxeter wheel.
In Section \ref{sec:6}, we provide several equivalent geometric realizations of the geometric root system.
In Section \ref{sec:7}, we construct Lie algebras associated with each simple singularity.
In Section \ref{sec:8}, we offer a geometric interpretation of the Lie bracket in terms of the edges and spokes on the Coxeter wheel.
In Section \ref{sec:9}, we describe the Weyl group reflections associated with the geometric root system and analyze the action of the Coxeter element.

\subsection*{Acknowledgements}
First author would like to thank Nan-Kuo Ho for introducing him the Coxeter plane.
We would like to thank Dongwook Choa and Hanwool Bae for the collaborations and discussions where several ideas in this paper are originated from.
Cho and Kim are  supported by NRF grant funded
by the Korea government(MSIT)  No.2020R1A5A1016126 and 0450-20250061.
Jeong was supported by G-LAMP program of the NRF grant funded by the Ministry of Education (RS-2024-00441954).

\section{Milnor fibers of simple curve singularities} \label{sec:2}
In this section, we review a particular combinatorial description of the Milnor fiber $M$ of simple curve singularities following \cite{CCJ2}.
For the case of $A$ and $D$, this will be enough for our geometric construction of Lie algebras.
For the case of $E_6, E_7, E_8$, further rearrangements of Milnor fibers into Coxeter wheels are required and will be explained in the next section.

We recall the following theorem.
\begin{thm}[{\cite[Theorem 1.7]{CCJ2}}]\label{thm:M1}
Milnor fiber of a simple curve singularity  can be described as follows.
\begin{enumerate}
\item $A_{k}$-case : $(2k+2)$-gon with edges identified as $\pm 3$ pattern.
\item $D_{k}$-case : $(4k-4)$-gon with edges identified as $\pm (2k-3)$ pattern.
\item $E_{6}$-case : $16$-gon with edges identified as $\pm 5$ pattern.
\item $E_{7}$-case : $18$-gon with edges identified as $\pm 5$ pattern.
\item $E_{8}$-case : $20$-gon with edges identified as $\pm 5$ pattern.
\end{enumerate}

For $D_k$, one may equivalently regard the Milnor fiber as a $(2k-2)$-gon with opposite edges identified.

Here, edges are identified as $\pm m$ pattern means the following.
A boundary edge $E$ has two boundary points, say $p$ and $q$. Suppose $p$ is a puncture. 
Travel along the boundary of the polygon starting from this edge in the direction of the vector $\overrightarrow{pq}$, to the $m$-th neighboring edge $E'$. We identify $E$ and $E'$ so that the resulting surface is orientable.
\end{thm}

\begin{figure}[h]
\begin{subfigure}{0.47\textwidth}
\includegraphics[scale=0.9]{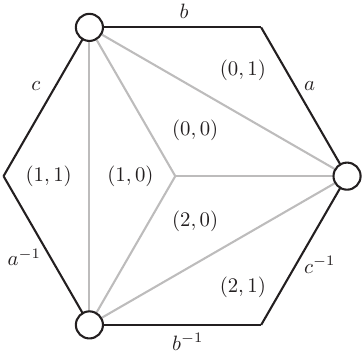}
\centering
\caption{Milnor fiber for $A_{2}$-case.}
\end{subfigure}
\begin{subfigure}{0.47\textwidth}
\includegraphics[scale=0.9]{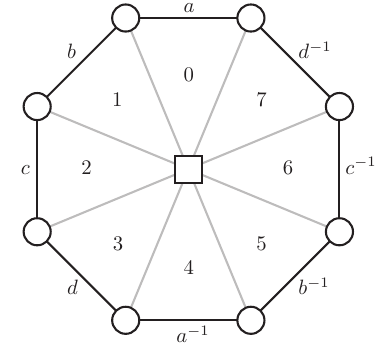}
\centering
\caption{Milnor fiber for $D_{5}$-case.}
\end{subfigure}
\centering
\caption{Milnor fibers of $A_{k}$ and $D_{k}$-singularities.}
\label{fig:ADMil}
\end{figure}

\begin{figure}[h]
\begin{subfigure}{0.47\textwidth}
\includegraphics[scale=0.8]{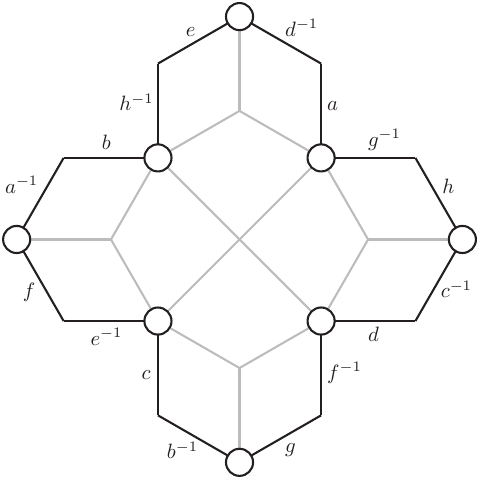}
\centering
\caption{Milnor fiber for $E_{6}$-case.}
\end{subfigure}
\begin{subfigure}{0.47\textwidth}
\includegraphics[scale=0.75]{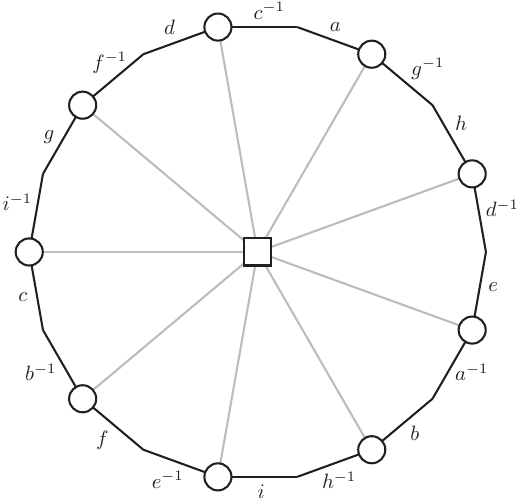}
\centering
\caption{Milnor fiber for $E_{7}$-case.}
\end{subfigure}
\centering
\caption{Milnor fibers of $E_{6}$ and $E_{7}$-singularities.}
\label{fig:E67Mil}
\end{figure}

\begin{figure}[h]
\includegraphics[scale=0.8]{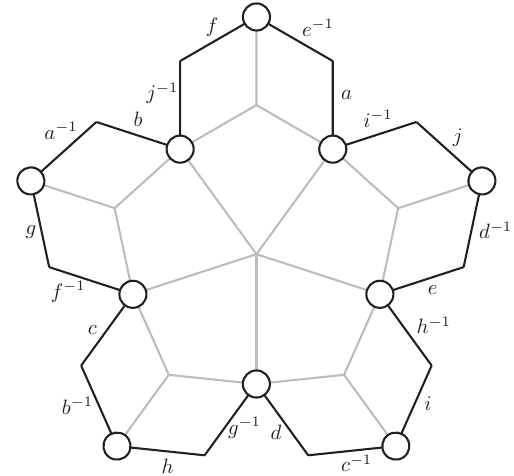}
\centering
\caption{Milnor fiber of $E_{8}$-singularity.}
\label{fig:E8Mil}
\end{figure}

In \cite{CCJ2}, the above theorem was proved in more general setting of invertible polynomials of two variables.
Up to coordinate changes, any invertible polynomial of two variables is one of the following:
$$F_{p,q} = x^{p} + y^{q}, \; C_{p,q}=x^{p} + xy^{q}, \mbox{ and } L_{p,q} = x^{p}y + xy^{q}.$$
Thus $A_{k}$, $D_{k}$, $E_{6}$, $E_{7}$, and $E_{8}$-singularities are  $F_{k+1,2}, C_{k-1,2}, F_{4,3}, C_{3,3}$, and $F_{5,3}$ respectively.
Furthermore, these are weighted homogeneous polynomials.


Recall that the Milnor fiber of a weighted homogeneous isolated singularity is defined by $\{f=1\} \cap \overline{B}_R(0)$.
The proof of Theorem \ref{thm:M1} is based on analyzing the Riemann surface  $f(x,y) =1$.
For example, in the case of $F_{p,q}$, the Riemann surface of $y = \sqrt[q]{1-x^p}$ is given by gluing $q$-copies of $x$-plane with
$p$-branch cuts from $p$-th root of unity to the infinity. 
For the case of $C_{p,q}$, the Riemann surface of $y=\sqrt[q]{\frac{1-x^p}{x}}$ is given by gluing $q$-copies of $x$-plane with
additional branch cut from the origin (for the pole). We refer readers to \cite{CCJ2} for more details.

Milnor fibers above have symmetry groups comming from the diagonal symmetry group of  invertible polynomials: define the maximal symmetry group $G_f$ of $f$ as
$$ G_f = \{ (\lambda, \mu) \in \mathbb{C}^2 \mid f(\lambda x, \mu y) = f(x,y)\}.$$
\begin{lemma}
The maximal symmetry group $G_{f}$ for simple curve singularties  $A_{k}$, $D_{k}$, $E_{6}$, $E_{7}$, and $E_{8}$ are
$$\Z/(k+1)\Z \oplus \Z/2\Z,  \;\; \Z/2(k-1)\Z,\;\; \Z/4\Z \oplus \Z/3\Z, \;\;\Z/9\Z,\mbox{ and } \;\Z/5\Z \oplus \Z/3\Z.$$
\end{lemma}

In the Fermat case $F_{p,q}$, $\Z/p\Z \oplus \Z/q\Z$ action on the Milnor fiber can be easily seen.
Milnor fiber consists of $p$ copies of $2q$-gon (with $q$-punctures). Then, $\Z/p\Z$ acts by cyclic rotation of $2\pi/p$ of the whole Milnor fiber and
$\Z/q\Z$-action is given by cyclic rotation of $2\pi/q$ of each $2q$-gons.
Note that each edge belongs to two different $2q$-gons and in fact, boundary identification rule is equivalent to the well-definedness of the latter action.
In the case of $D_k$, Milnor fiber has the cyclic $\Z/(2k-2)\Z$-action, where the generator $1$ act by multiplication of $\left(e^{\frac{2\pi  i}{k-1}}, e^{-\frac{\pi  i}{k-1}}\right)$.
In the case of $E_7$, Milnor fiber has the cyclic $\Z/9\Z$-action, where the generator $1$ act by multiplication of $\left(e^{\frac{2\pi  i}{3}}, e^{-\frac{2\pi  i}{9}}\right)$.
If we fix one fundamental domain of $G_f$-action on the Milnor fiber $M$, then $M$ has a tessellation that are labed by the elements of $G_f$.
We refer readers to \cite{CCJ2} for more details.

Geometric monodromy for the Milnor fibers for weighted homogeoneous polynomials are given by the weight action \eqref{eq:wa}, and in our cases, this action is a part of the maximal symmetry group $G_f$. 
\begin{lemma}\label{lem:mo}
The geometric monodromy $\rho$ for a simple curve singularity $f$ is given by the group element $(1,1) \in G_f$ for $A_{k}$, $E_{6}$, and $E_{8}$, by $k\in \Z/(2k-2)\Z$ for $D_k$, and by $7 \in \Z/9\Z$ for $E_{7}$-case.
\end{lemma}

\section{Coxeter wheel from the Milnor fiber} \label{sec:3}
In this section, we define the Coxeter wheel $W \subset \mathbb{R}^2$ by rearranging Milnor fiber of each simple singularity. 
Later, the edges and spokes on this wheel (up to equivalence) will correspond to roots.


The wheel is easy to describe for the cases of $A_k$ and $D_k$.
\begin{itemize}
\item $A_k$-wheel:
Recall that $A_k$-Milnor fiber $M$ is obtained from the regular $(2k+2)$-gon  by identifying boundary edges,
and there are $(k+1)$ punctures at every other vertices. We consider a regular $(k+1)$-gon in $M$ whose vertices are given by these $(k+1)$-punctures. 
This is the wheel $W$. Note that $A_k$-wheel is only a half of the Milnor fiber $M$ in this case,
and geometric monodromy $\rho$ takes the wheel to the other half.

\item $D_k$-wheel: The $D_k$-Milnor fiber $M$ itself is a wheel $W$ in some sense.  Using the description in Theorem \ref{thm:M1}, the regular $(2k-2)$-gon with a puncture at the origin and boundary vertices is the wheel $W$. 
The Milnor fiber $M$ can be obtained by gluing opposite sides of $D_k$-wheel.
The geometric monodromy (in Lemma \ref{lem:mo}) on $D_k$-wheel is given by the counterclockwise rotation of the wheel by $\frac{k}{k-1}\pi$.
\end{itemize}

The construction of the wheel for exceptional cases are more elaborate.
In the rest of the section, we will describe a construction of the $E$-type wheels that proves Theorem \ref{Etype wheel}.

\subsection{$E_{6}$-singularity}
We decompose the Milnor fiber of $E_6$-singularity  as in Figure \ref{fig:E6domain2} (c.f., Figure \ref{fig:E67Mil} (A)).
\begin{figure}[h]
\includegraphics[scale=0.8]{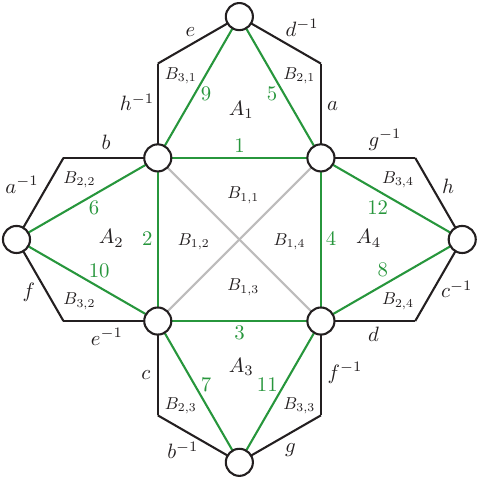}
\centering
\caption{Another decomposition of $E_{6}$-Milnor fiber.}
\label{fig:E6domain2}
\end{figure}
Namely, this cuts $M$ into the 4  triangles $A_{1},\dots,A_{4}$ and 12 triangles  $\{B_{i,j}\}$ for $1 \leq i \leq 3$ and $1\leq j \leq 4$. 
Note that 4 triangles $\{B_{i,1}, B_{i,2}, B_{i,3}, B_{i,4}\}$ are glued (in $M$) to form a 4-gon, which we call $B_i$ for $i=1,2,3$.
See Figure \ref{fig:E6pieces}.  Also note that $B_1$ is the square in the center of Figure  \ref{fig:E6domain2}.
\begin{figure}[h]
\includegraphics[scale=0.8]{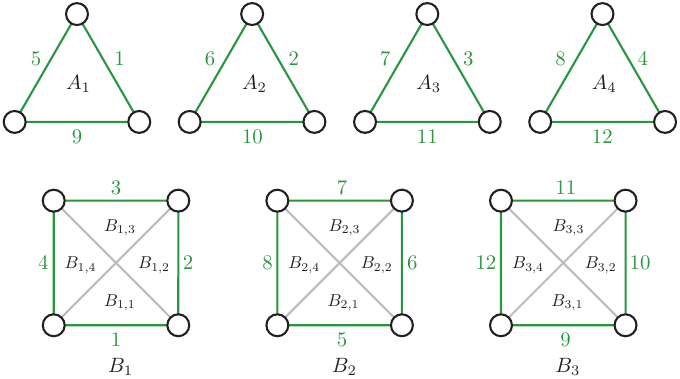}
\centering
\caption{Four regular triangles and three squares consisting of $E_{6}$-Milnor fiber.}
\label{fig:E6pieces}
\end{figure}
In fact, we have chosen an embedding of these pieces in $\mathbb{R}^2$ in Figure \ref{fig:E6pieces}  such that each $A$-triangle is a regular triangle  and
each $B_{ij}$ is a  right isosceles triangle. 

Now, the wheel $W$ is defined by gluing 3 copies of each $A$-triangles and 2 copies of each $B$-triangles in a particular way
as described in Figure \ref{fig:E6Milnor}. Each piece in Figure \ref{fig:E6Milnor} is obtained from the one in Figure \ref{fig:E6pieces} modulo rotation and translation.
\begin{figure}[h]
\includegraphics[scale=0.5]{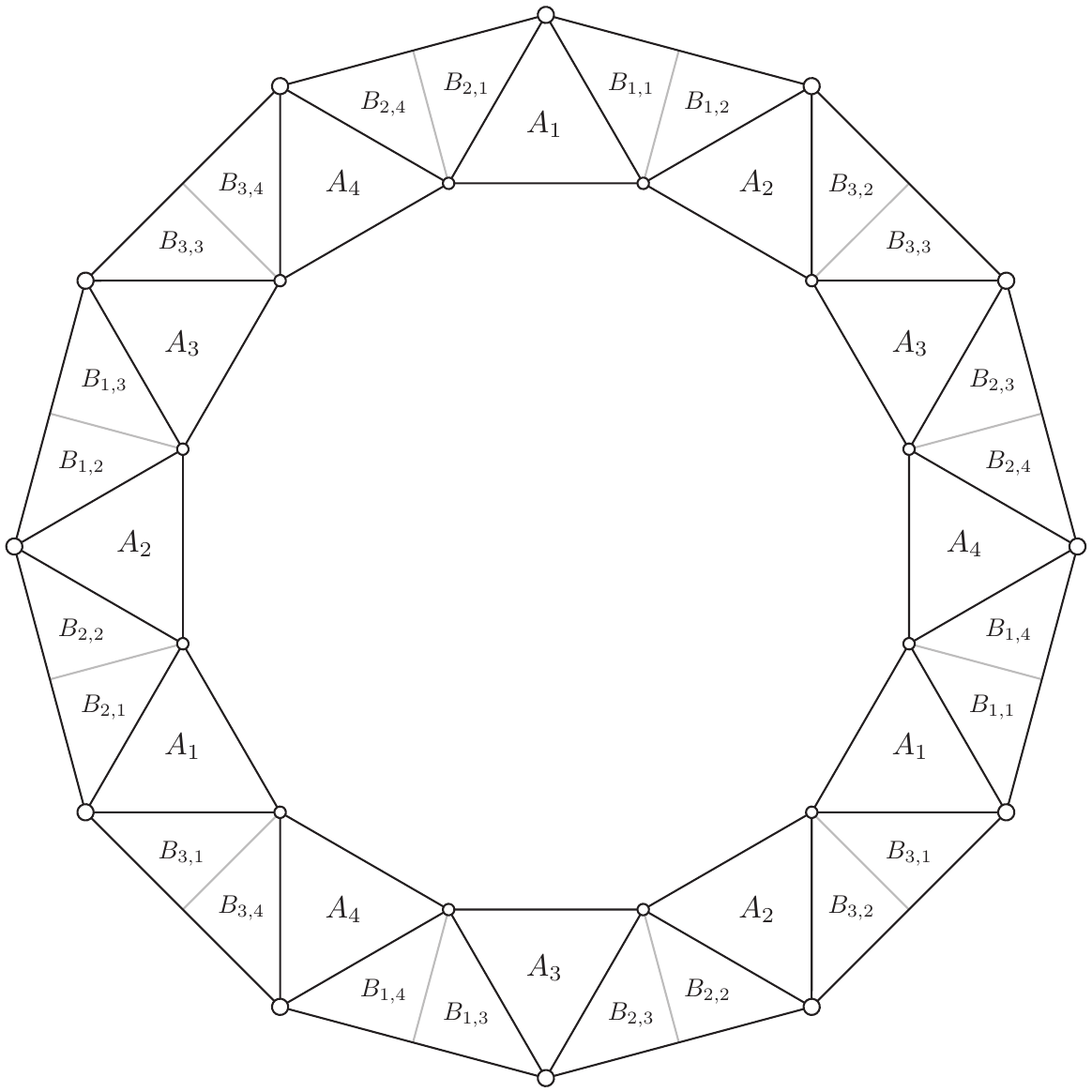}
\centering
\caption{The wheel $W$ for the $E_{6}$-singularity.}
\label{fig:E6Milnor}
\end{figure}
Construction of $W$ can be seen in Figure \ref{fig:E6domain2} as follows.
We start from the union of triangles $A_1, B_{1,1}, B_{1,2}$. These three shares a vertex and form a 4-gon in $M$. 
Adjacent to the last triangle $B_{1,2}$, we have $A_2$ which is then adjacent to $B_{3,2}$.  
 From the gluing rule, $B_{3,2}$ is adjacent to $B_{3,3}$ in $M$. Note that the triangles $B_{1,2}, A_2, B_{3,2}$ shares a vertex
 and so does the triple $A_2, B_{3,2}, B_{3,3}$. These give a part of the wheel consisting of
 $$A_1, B_{1,1}, B_{1,2}, A_2, B_{3,2}, B_{3,3}.$$
 After continuing in this way, we return to the original  triple $A_1, B_{1,1}, B_{1,2}$ forming the  wheel $W$.

Thus wheel is a region bounded by two 12-gons, and one can check that the projection map $\pi: W \to M$ is continuous and surjective.
In fact, we may also consider a bigger wheel $\widetilde{W}$ so that the projection map $\widetilde{\pi}:\widetilde{W} \to M$ becomes a 3-fold covering.
Namely, for each $A$-triangles in the wheel $W$, the interior edge is adjacent to a $B$-triangle and we can attach these 12 $B$-triangles to $W$ and obtain $\widetilde{W}$.

Alternatively the Milnor fiber decomposes into four $\Z/3$-invariant hexagon with one $A$-triangle and 3 adjacent $B$-triangles (realized as right isosceles) and $\widetilde{W}$ is
the union of 12 hexagons, obtained by gluing 3 copies of each hexagons in a particular pattern. 
We choose to work with $W$ instead of $\widetilde{W}$ since all the vertices of $W$ are punctures unlike $\widetilde{W}$.

Recall that the monodromy of Milnor fiber is the same as the group action by $(1,1) \in \Z/4\Z \oplus \Z/3\Z$.
We now show that this can be identified with  the counterclockwise rotation of $W$ by $\frac{6+1}{12} \times 2\pi$.

We find that $(1,1)$ action takes $A_j$ to $A_{j+1}$ 
and takes $B_{1,j}$ to $B_{2,j-1}$, $B_{2,j}$ to $B_{3,j+1}$, $B_{3,j}$to $B_{1,j-1}$ where indices are taken modulo 4.
Then one can check that this can be identified with  the counterclockwise rotation of $W$ by $\frac{7}{6}\pi$.
%
%
%
%

\subsection{$E_{7}$-singularity}
We decompose the  $E_7$-Milnor fiber (in Figure \ref{fig:E67Mil} (B)) as in Figure \ref{fig:E7domain2} (A)
into 9 triangles $A_{1,1},\dots, A_{3,3}$ and another 9 triangles $B_{1,1},\dots, B_{1,9}$.
Note that under the gluing rule, three $A$-triangles $A_{i,1}$, $A_{i,2}$, $A_{i,3}$ glue to a triangle $A_i$ for each $i=1,2,3$ and 9 $B$-triangle glues to a punctured 9-gon.
Also note that each $B_{1,j}$ is adjacent to only one $A$-triangle.

\begin{figure}[h]
\begin{subfigure}[t]{0.42\textwidth}
\includegraphics[scale=0.65]{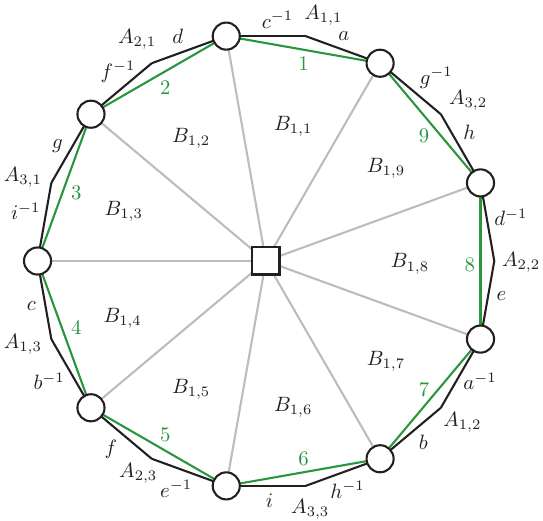}
\centering
\caption{Another decomposition of $E_{7}$-Milnor fiber.}
\end{subfigure}
\begin{subfigure}[t]{0.53\textwidth}
\includegraphics[scale=0.7]{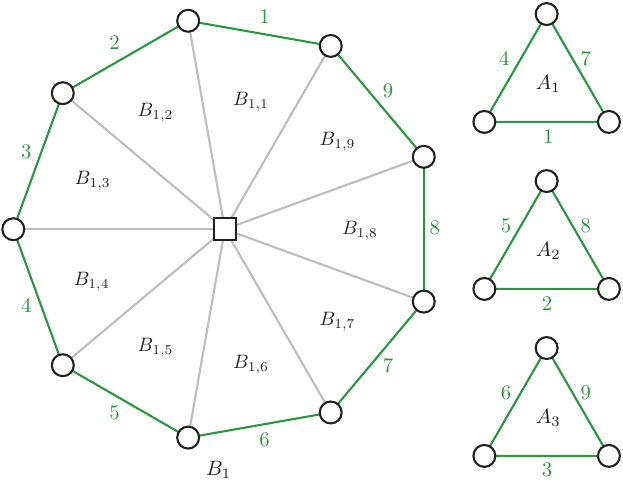}
\centering
\caption{One regular $9$-gon and three squares consisting of $E_{7}$-Milnor fiber.}
\end{subfigure}
\centering
\caption{Milnor fiber of $E_{7}$-singularity.}
\label{fig:E7domain2}
\end{figure}

\begin{figure}[h]
\includegraphics[scale=0.5]{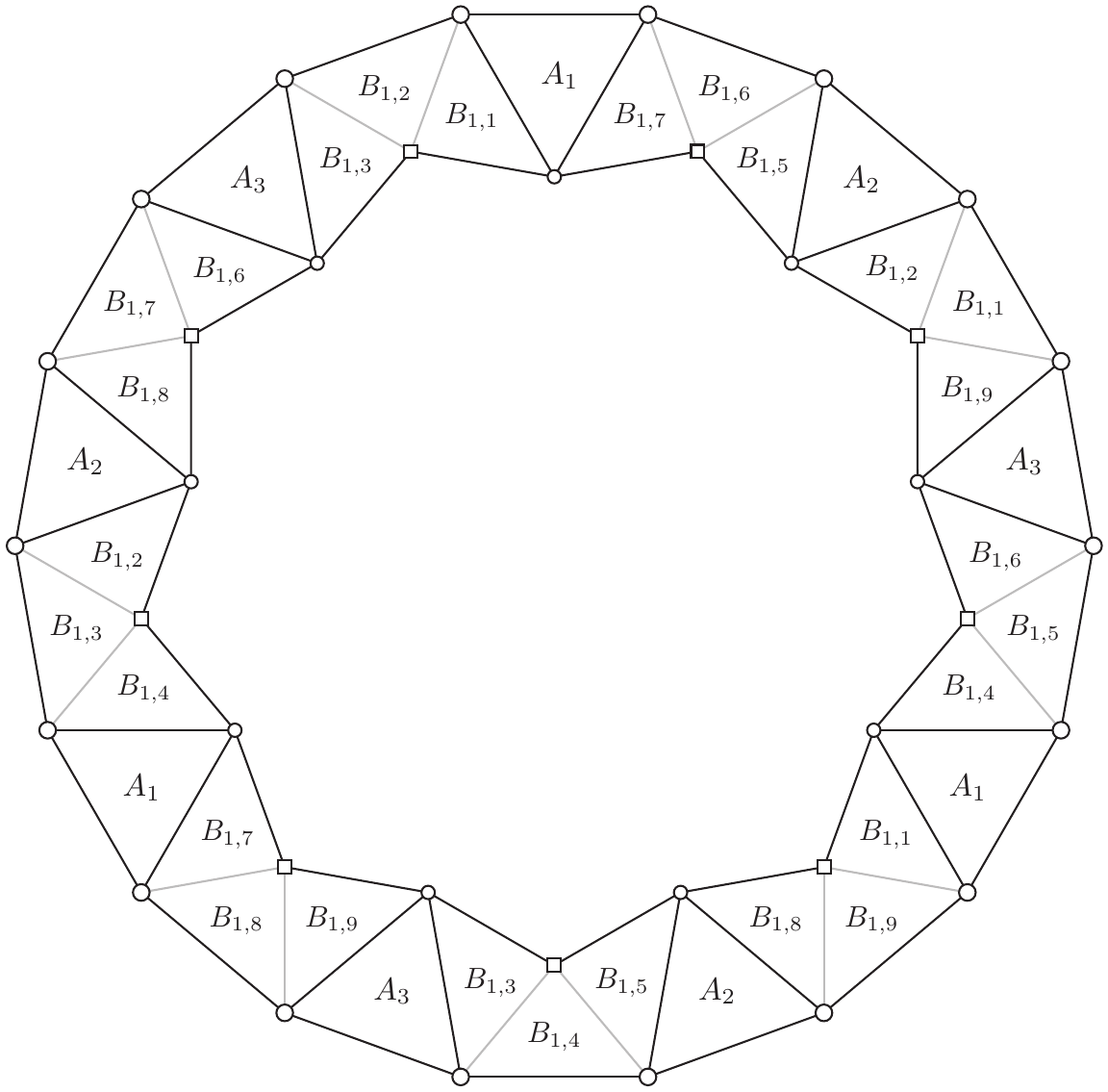}
\centering
\caption{The wheel $W$ for the $E_{7}$-singularity.}
\label{fig:E7Milnor}
\end{figure}

%
To make the wheel, we first choose  embeddings of isosceles triangles $B_{1,j}$ in $\mathbb{R}^2$ so that its interior angles are $80\degree, 50\degree, 50\degree$.  
(This is {\em not} the angles of $B_{1,j}$ drawn in Figure \ref{fig:E7domain2} (B) which are $40\degree, 70\degree,70\degree$.)

Then the wheel in Figure \ref{fig:E7Milnor} can be constructed as follows.
Start with 3 consecutive $B$-triangles (with the above angles), say $B_{1,1}, B_{1,2}, B_{1,3}$ and continue with the unique $A$-triangle $A_3$ attached with $B_{1,3}$.
From $A_3$, we can continue with 3 consecutive $B$-triangles $B_{1,6}, B_{1,7}, B_{1,8}$ and the last one meets the $A_2$. Continuing in this way, we come back to
the original configuration $B_{1,1}, B_{1,2}, B_{1,3}$ and this gives the polygonal wheel. Angles are chosen to make this work.

The wheel is a region bounded by two 18-gons, and  the projection map to the Milnor fiber  $W \to M$ is a 3-fold covering.
Given an $A$ or $B$ triangle, its preimage has 3 isomorphic copies that are identified by translations in $\mathbb{R}^2$.

We show that the geometric monodromy of $E_7$-Milnor fiber lifts to the counterclockwise rotation of $18$-gon by $\frac{9+1}{18}\times 2\pi$.
Recall that $7 \in \Z/9\Z$ for $E_{7}$-case gives the monodromy action. Hence one can check that 
it sends $A_{i}$ to $A_{i+1}$ (modulo $3$) and $B_{1,j}$ to $B_{1,j+7}$ (modulo $9$).
On the wheel this is  the counterclockwise rotation of $W$ by $\frac{10}{9}\pi$.
%
%
%
 
\subsection{$E_{8}$-singularity}
We decompose the Milnor fiber $M$ of $E_8$-singularity  as in Figure \ref{fig:E8domain2}.

\begin{figure}[h]
\includegraphics[scale=0.8]{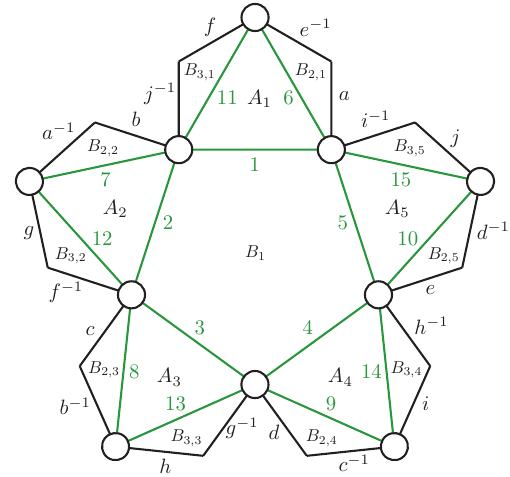}
\centering
\caption{Another decomposition of $E_{8}$-Milnor fiber.}
\label{fig:E8domain2}
\end{figure}
 
Namely, this cuts $M$ into 5  triangles $A_{1},\dots,A_{5}$, a pentagon $B_1$ and 10 triangles $\{B_{i,j}\}$.
Note that 5 triangles $\{B_{i,1},\dots, B_{i,5}\}$ are glued (in $M$) to form a 5-gon, which we call $B_i$ for $i=2,3$. Hence we have 3 pentagons $B_1,B_2,B_3$. 
We choose an embedding of $A$-triangles as regular triangles and $B$-pentagons as regular pentagons as in Figure \ref{fig:E8pieces}.

 \begin{figure}[h]
\includegraphics[scale=0.8]{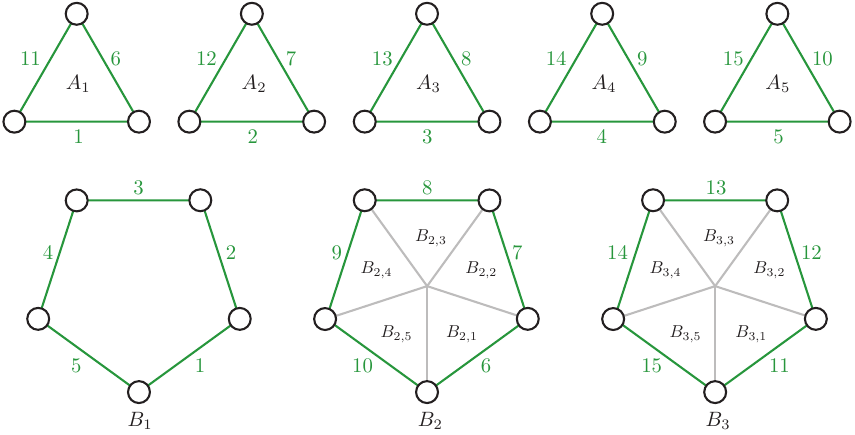}
\centering
\caption{5 regular triangles and 3 regular pentagons consisting of $E_{8}$-Milnor fiber.}
\label{fig:E8pieces}
\end{figure}
The wheel $W$ is constructed by gluing $A_1, A_3, A_5, A_2, A_4$ triangles and $B_1, B_3, B_2$-pentagons in an alternating pattern as in Figure \ref{fig:E8Milnor}.
The wheel is a region bounded by two 30-gons.
Each $A$-triangle is used 3 times and each $B$-pentagon 5 times. Thus the projection map $\pi:W \to M$ is not a covering map.
We remark that  we can extend $W$ to $\widetilde{W}$ by attaching $1/3$ of $A$-triangles to the interior boundary edges of the wheel.
Then, $\widetilde{W}$ after identifying corresponding parallel boundary edges become a 5-fold covering of $M$.

\begin{figure}[h]
\includegraphics[scale=0.5]{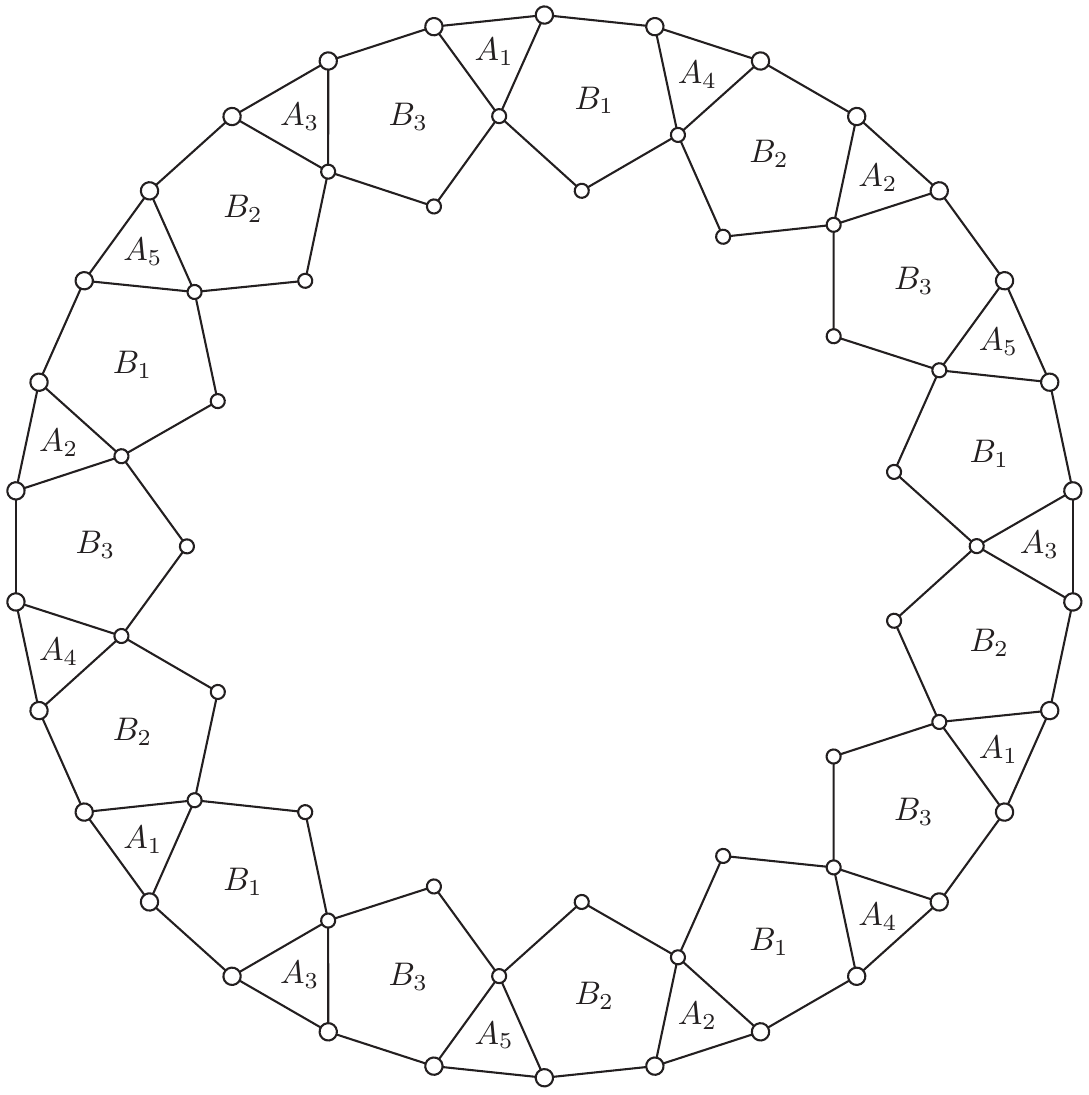}
\centering
\caption{The wheel $W$ for the $E_{8}$-singularity.}
\label{fig:E8Milnor}
\end{figure}
In this case, the geometric monodromy of Milnor fiber is the group action by $(1,1) \in \Z/5\Z \oplus \Z/3\Z$. It sends $A_{j}$ to $A_{j+1}$ (modulo $5$) and $B_{i}$ to $B_{i+1}$ (modulo $3$).
The geometric monodromy of $E_{8}$-Milnor fiber can be lifted to the counterclockwise rotation of the wheel by $\frac{15+1}{30} \times 2\pi = \frac{16}{15}\pi$.
%
%
%

\subsection{Wheels and Coxeter planes}
Coxeter plane of a complex simple Lie algebra $\mathfrak{g}$ is a certain real two dimensional plane in $\mathfrak{h}^*$
with the dots given by orthogonal projection of roots. Coxeter called it the most symmetric projection of roots to a plane.
On the Coxeter plane, a choice of  Coxeter element acts on the plane as the rotation by $\frac{2 \pi}{h}$ ($h$ is the Coxeter number). 
We refer readers to Appendix B of \cite{GH19} and references therein. 

Surprisingly, there exists an interesting relation between the wheel and the Coxeter planes.
Namely, we can place our $E_{6}$, $E_{7}$, and $E_{8}$-wheels on the Coxeter planes so that vertices are on the root projections!
See Figure \ref{fig:E67Coxeter} and Figure \ref{fig:E8Coxeter}.
From this observation, we will call these wheels as Coxeter wheels.

 
Often, root projections in the Coxeter planes are decorated with line segments between them.
The edges in those figures join pairs of roots that are nearest neighbors, 
i.e., two projections of roots $\alpha$ and $\beta$ are connected by a line segment when $\alpha-\beta$ is also in the root system \cite{WangZhao}, \cite{Stembridge}.
The root projection images of $E$-type root systems available on the webpage \cite{Stembridge} demonstrate a large number of edges.
Later, we will find that geometric roots are realized as edges and spokes lying on the wheels. 
Given that our wheel diagram is aligned with the Coxeter plane, interpreting the edges and spokes as roots is quite natural.

\begin{figure}[h]
\begin{subfigure}{0.47\textwidth}
\includegraphics[scale=0.35]{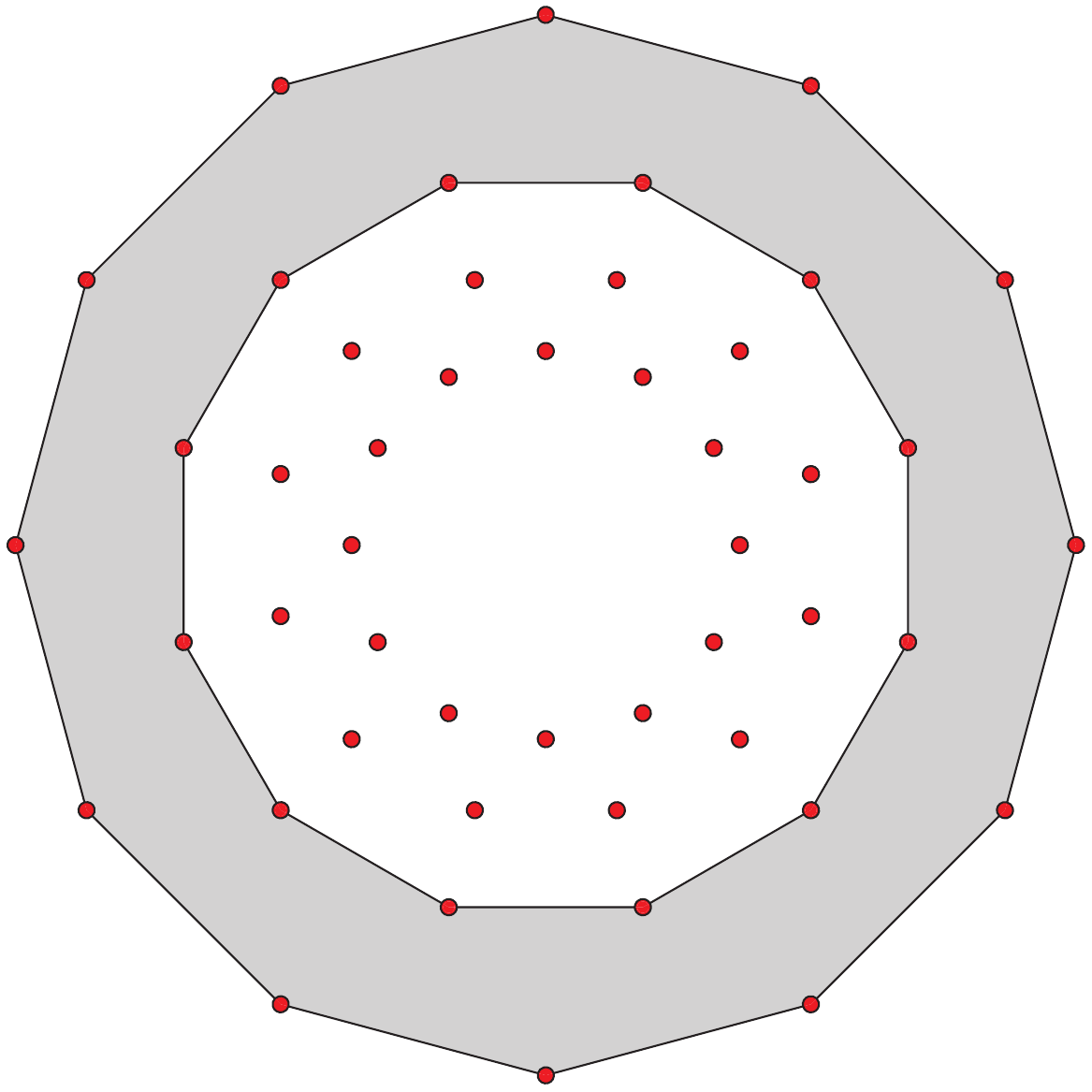}
\centering
\caption{Coxeter wheel for $E_{6}$.}
\end{subfigure}
\begin{subfigure}{0.47\textwidth}
\includegraphics[scale=0.35]{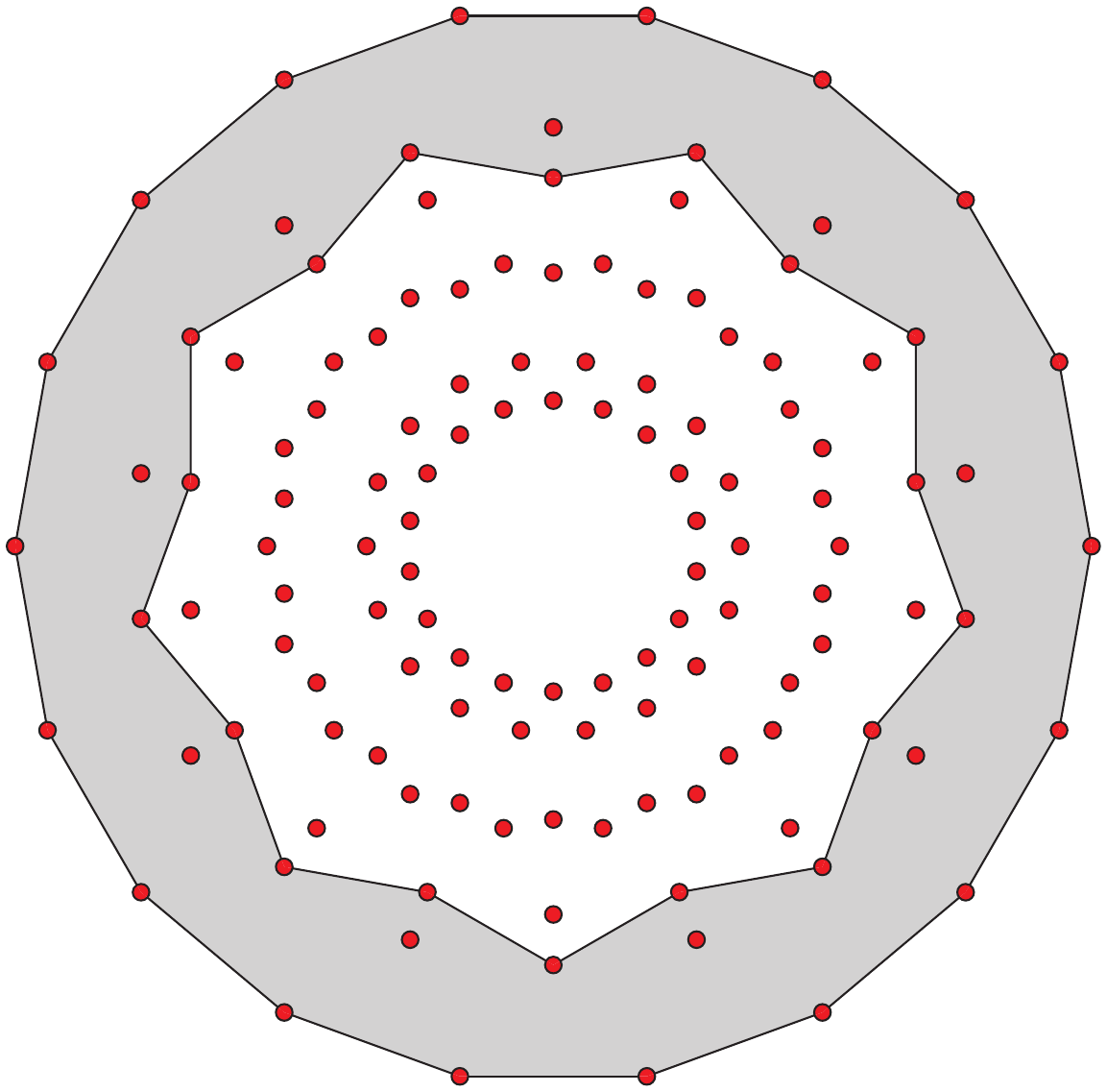}
\centering
\caption{Coxeter wheel for $E_{7}$.}
\end{subfigure}
\centering
\caption{Coxeter wheels.}
\label{fig:E67Coxeter}
\end{figure}

\begin{figure}[h]
\includegraphics[scale=0.35]{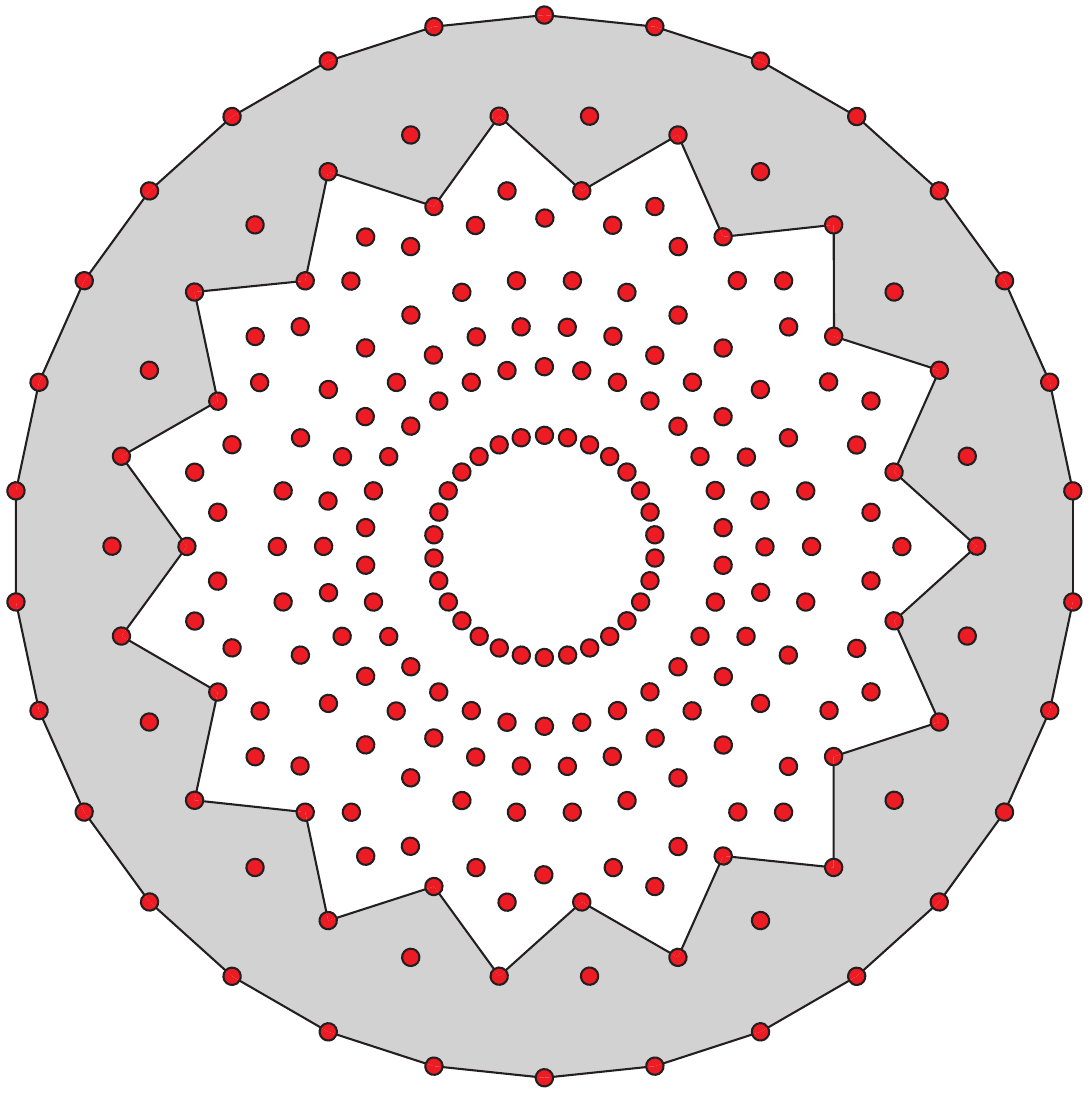}
\centering
\caption{Coxeter wheel for $E_{8}$.}
\label{fig:E8Coxeter}
\end{figure}

%
%

\section{Geometric root systems} \label{sec:4}
Let $f(x,y)$ be a simple curve singularity of type $\Gamma=A_k,\,D_k,\,E_6,\,E_7, E_8$ and $M$ be its Milnor fiber.
We have a bilinear form $(\cdot,\cdot)$ on $H_1(M,\partial M;\Z)$ given by the negative symmetrized Seifert form
in Definition \ref{defn:S}.  
Theorem \ref{geometric roots} claims that the set of geometric roots defined by 
$$\Phi_\Gamma := \{\alpha\in H_1(M,\partial M;\Z) \;|\;  (\alpha,\alpha)=2 \} $$ 
forms a root system.   We first recall the definition of an abstract root system.

\begin{defn}\label{defi: Root system}
Let $E$ be a Euclidean space with an inner product $(\cdot, \cdot)$.
For $\alpha, \beta \in E$, let us denote by $ \langle \beta, \alpha \rangle  = 2 \frac{(\beta,\alpha)}{(\alpha,\alpha)}$. 
A subset $\Phi \subset E$  is called a root system if the following axioms are satisfied.
\begin{enumerate}
\item $\Phi$ is finite, spans $E$, and does not contain $0$.
\item If $\alpha \in \Phi$, then the only multiples of $\alpha$ in $\Phi$ are $\pm \alpha$.
\item If $\alpha  \in \Phi$, then the reflection map $s_\alpha:\beta \to \beta - \langle \beta, \alpha \rangle \alpha$ leaves $\Phi$ invariant.
\item If $\alpha,\beta \in \Phi$, then $ \langle \beta, \alpha \rangle \in \mathbb{Z}$.
\end{enumerate}
\end{defn}
 
\subsection{Proof of Thm \ref{geometric roots}}
We first show that in the Eulclidean space $E = H_1(M,\partial M;\Z)\otimes \R$, a subset $\Phi_\Gamma$ of $E$ with  $(\cdot,\cdot)$  satisfies the axioms of a root system.
Note that axioms (2), (3), and (4) follow directly from our definition of $\Phi_\Gamma$ and standard computations.
It remains to show that the first axiom holds and that $(\cdot,\cdot)$ is an inner product defined by the  standard Cartan matrix of type $\Gamma$.

In order to show this, we will find a `geometric simple basis' of $H_1(M,\partial M;\Z)$ so that the bilinear form $(\cdot,\cdot)$ is
represented by the standard Cartan matrix of type $\Gamma$. 
\begin{prop}\label{cartan matrix}
    For each type of simple singularity $f$, there exists a basis $\{\alpha_1,\dots,\alpha_k\}$  of $H_1(M,\partial M;\Z) \cong \Z^k$ such that
    the matrix $\big(B_{ij} \big) = \big( \mathrm{var}(\alpha_i) \bullet \alpha_j\big)$ is of the following form. In particular, the basis consists of
    geometric roots (i.e., $(\alpha_i,\alpha_i)=2$ for all $i$).

\begin{tabular}{p{5em} p{18em} p{5em} p{18em}}
(a) $A_{k}$ & $ \begin{pmatrix}
                1 & -1 & 0 & 0 & & & \\
                0 & 1 & -1 & 0 & &\dots &  \\
                0 & 0 & 1 & -1 & & &  \\
                & &  & \ddots & \ddots & & \\
                & \vdots &  & & 1  & -1  & 0  \\
                & &  & &  & 1  & -1  \\
                & &  &  &  &  & 1
            \end{pmatrix}$
& (b) $D_{k}$ & $ \begin{pmatrix}
                1 & 0 & -1 & 0 & & & \\
                0 & 1 & -1 & 0 & &\dots &  \\
                0 & 0 & 1 & -1 & & &  \\
                & &  & \ddots & \ddots & & \\
                & \vdots &  & & 1  & -1  & 0  \\
                & &  & &  & 1  & -1  \\
                & &  &  &  &  & 1
            \end{pmatrix}$
\end{tabular} \vspace{5mm}

\begin{tabular}{p{6.5em} p{16.5em} p{5em} p{18em}}
(c) $E_{6}$ & $\begin{pmatrix}
                1 & -1 & 0 & 0 & 0 & 0 \\
                0 & 1 & 0 & 0 & -1 & 0 \\
                0 & 0 & 1 & -1 & 0 & 0 \\
                0 & 0 & 0 & 1 & -1 & 0 \\
                0 & 0 & 0 & 0 & 1 & -1  \\
                0 & 0 & 0 & 0 & 0 & 1
            \end{pmatrix}$ &
(d) $E_{7}$ & $\begin{pmatrix}
                1 & 0 & 0 & 0 & -1 & 0 & 0 \\
                0 & 1 & -1 & 0 & 0 & 0 & 0 \\
                0 & 0 & 1 & -1 & 0 & 0 & 0 \\
                0 & 0 & 0 & 1 & -1 & 0 & 0 \\
                0 & 0 & 0 & 0 & 1 & -1 & 0 \\
                0 & 0 & 0 & 0 & 0 & 1 & -1 \\
                0 & 0 & 0 & 0 & 0 & 0 & 1
            \end{pmatrix}$
\end{tabular} \vspace{3mm}

\begin{center}
\begin{tabular}{p{5em} p{18em}}
(e) $E_{8}$ & $\begin{pmatrix}
                1 & 0 & 0 & -1 & 0 & 0 & 0 & 0 \\
                0 & 1 & -1 & 0 & 0 & 0 & 0 & 0 \\
                0 & 0 & 1 & -1 & 0 & 0 & 0 & 0 \\
                0 & 0 & 0 & 1 & -1 & 0 & 0 & 0 \\
                0 & 0 & 0 & 0 & 1 & -1 & 0 & 0 \\
                0 & 0 & 0 & 0 & 0 & 1 & -1 & 0 \\
                0 & 0 & 0 & 0 & 0 & 0 & 1 & -1 \\
                0 & 0 & 0 & 0 & 0 & 0 & 0 & 1
            \end{pmatrix}$
\end{tabular}
\end{center}

\end{prop}
\begin{proof}
Theorem \ref{geometric roots} follows from Proposition \ref{cartan matrix} in the following way.
The set $\Phi_\Gamma$ spans $E$, because $\{\alpha_1,\dots,\alpha_k\}$ forms a basis of $H_1(M,\partial M;\Z)$.
The bilinear form $(\cdot,\cdot)$ defines an inner product on $E$, as the matrix $B+B^t$ is symmetric, positive-definite,
and coincides with the standard Cartan matrix of type $\Gamma$.
The finiteness of $\Phi_\Gamma$ follows from positive-definiteness of $(\cdot,\cdot)$.
Thus, $\Phi_\Gamma$ defines a root system that is isomorphic to the standard root system of type $\Gamma$.
\end{proof}

\subsection{Simple basis of arcs}
In this subsection, we prove Proposition \ref{cartan matrix}.
    We will  find the elements $\alpha_1,\dots,\alpha_k$ of which the matrix $B=\big(B_{ij} \big) = \big( \mathrm{var}(\alpha_i) \bullet \alpha_j\big)$ is of the form described in the proposition.
    Note that this implies that  $\{\alpha_i\}_{i=1}^k$ forms an integral basis of $H_1(M,\partial M;\Z)$; $\mathrm{var}$ is linear and intersection form is bilinear, so if they were not linearly independent,
    the matrix $B$ would not have full rank. Also, if   $\{\alpha_i\}_{i=1}^k$ is not an integral basis, we have an integral vector $v = \sum_{i=1}^k c_i \alpha_i$ such that $c_i \in \mathbb{Q}$
    for any $i$ and not all of them lie on $\mathbb{Z}$. Let $i_0$ be the smallest index such that $c_i \notin \mathbb{Z}$, then $v^tB\vec{e}_{i_0} \notin \Z$, which is a contradiction.
   
 Recall that given an arc $K$ in the Milnor fiber, its monodromy image $\rho'(K)$ and the orientation reversal $\overline{K}$ meet at the boundary $\partial M$, and smoothing out the boundary
 gives a simple closed curve in the Milnor fiber $M$, which is the variation image of $K$. The wheel readily illustrates this, as the geometric monodromy $\rho$ is given by a cyclic rotation of the wheel (except type $A$).
    
   The desired simple basis will be given by the following arc representatives $\alpha_1,\dots,\alpha_k\in H_1(M,\partial M;\Z)$ as depicted on Figures \ref{fig:simple} (A), (C), (E), (G), and (I). It can be directly checked that the given $\alpha_j$'s induce the desired matrix $B$. We have illustrated some on Figures \ref{fig:simple} (B), (D), (F), (H), and (J).
   
   From the Lie theory, $\{\alpha_i\}_{i=1}^k$ is a simple basis. Namely, any other root is either a positive or negative linear combination of them.
For geometric roots, this will be checked in Subsection \ref{ssec:simple}.

           \begin{figure}[H]
            \centering
\begin{subfigure}[t]{0.47\textwidth}
\includegraphics[scale=1]{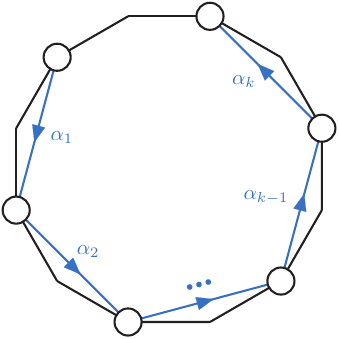}
\centering
\caption{The choice of $\alpha_1,\ldots,\alpha_k$ in $A_k$-Milnor fiber.}
\end{subfigure}
\begin{subfigure}[t]{0.47\textwidth}
\includegraphics[scale=1]{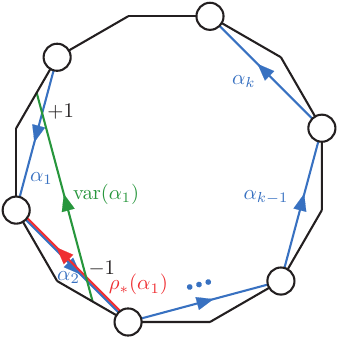}
\centering
\caption{An illustration for computation of $\var(\alpha_1)\bullet \alpha_j$.}
\end{subfigure} \vspace{10mm}

\begin{subfigure}[t]{0.47\textwidth}
\includegraphics[scale=1]{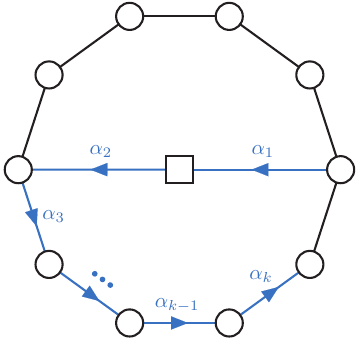}
\centering
\caption{The choice of $\alpha_1,\ldots,\alpha_k$ in $D_k$-Milnor fiber.}
\end{subfigure}
\begin{subfigure}[t]{0.47\textwidth}
\includegraphics[scale=1]{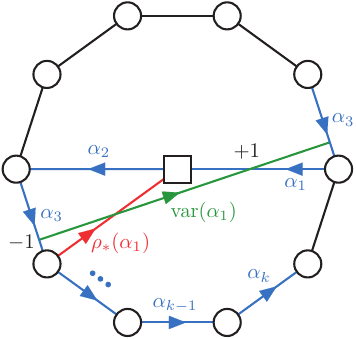}
\centering
\caption{An illustration for computation of $\var(\alpha_1)\bullet \alpha_j$.}
\end{subfigure} \vspace{3mm}
\caption{Examples of simple basis.}
        \end{figure}
        
        \begin{figure}[H] \ContinuedFloat
\begin{subfigure}[t]{0.47\textwidth} 
\includegraphics[scale=0.75]{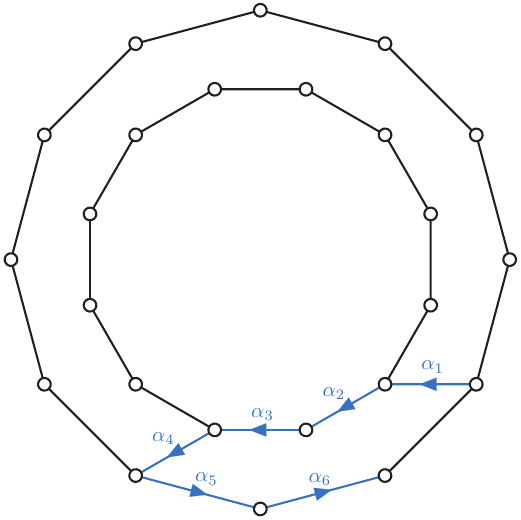}
\centering
\caption{The choice of $\alpha_1,\ldots,\alpha_6$ in $E_6$-Milnor fiber.}
\end{subfigure}
\begin{subfigure}[t]{0.47\textwidth}
\includegraphics[scale=0.75]{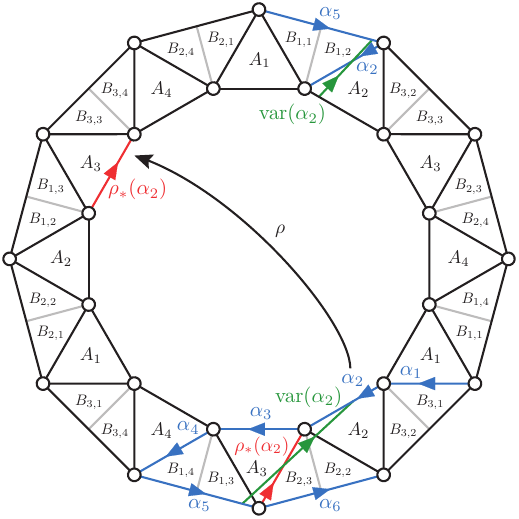}
\centering
\caption{An illustration for computation of $\var(\alpha_2)\bullet \alpha_j$.}
\end{subfigure}

\begin{subfigure}[t]{0.47\textwidth}
\includegraphics[scale=0.75]{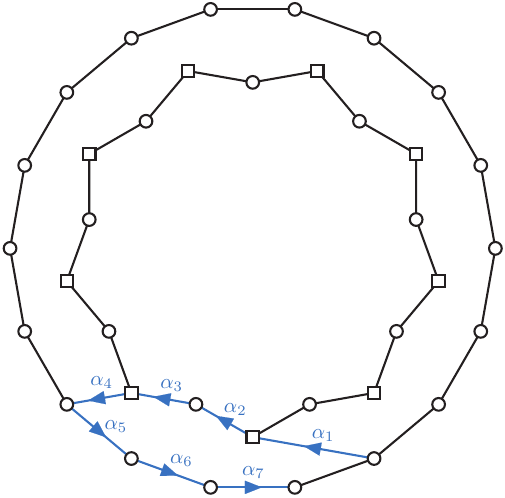}
\centering
\caption{The choice of $\alpha_1,\ldots,\alpha_7$ in $E_7$-Milnor fiber.}
\end{subfigure}
\begin{subfigure}[t]{0.47\textwidth}
\includegraphics[scale=0.75]{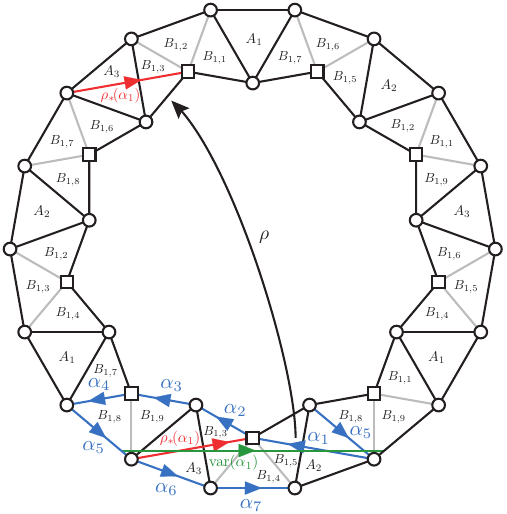}
\centering
\caption{An illustration for computation of $\var(\alpha_1)\bullet \alpha_j$.}
\end{subfigure}

\begin{subfigure}[t]{0.47\textwidth}
\includegraphics[scale=0.75]{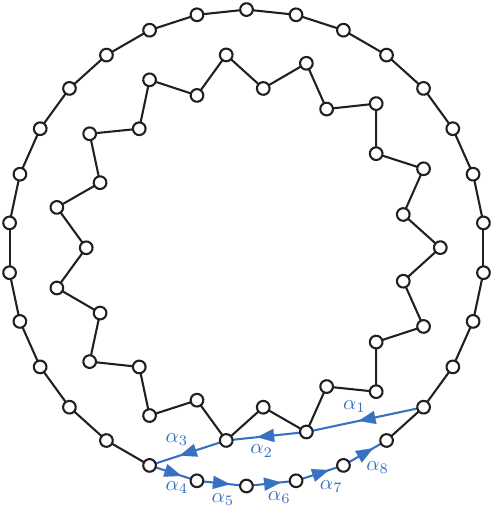}
\centering
\caption{The choice of $\alpha_1,\ldots,\alpha_8$ in $E_8$-Milnor fiber.}
\end{subfigure}
\begin{subfigure}[t]{0.47\textwidth}
\includegraphics[scale=0.75]{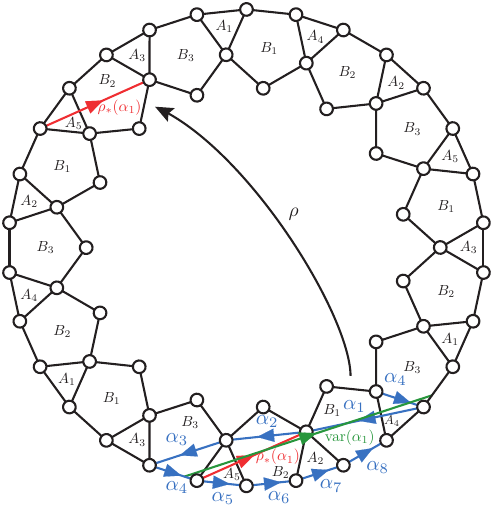}
\centering
\caption{An illustration for computation of $\var(\alpha_1)\bullet \alpha_j$.}
\end{subfigure}
\centering
\caption{Examples of simple basis.}
\label{fig:simple}
        \end{figure}

\section{Monodromy action on  the geometric root system}\label{sec:monodromy}
In the previous section, we have found a `simple' basis of geometric roots in the wheel and showed that
geometric root system and classical root system coincide for $ADE$.
In this section, we show that all geometric roots are realized as edges/spokes in the Coxeter wheel.
The converse statement will be proved in the next section.

To locate all geometric roots in the wheel, we will first find a different basis, called `projective' basis of geometric roots,
 also given by a collection of edges/spokes.   
Then, cyclic rotations, which are powers of monodromy, take edges/spokes in the projective basis to other edges/spokes. 
The upshot is that $\rho_*$ sends geometric roots to geometric roots, and its induced action on $\Phi_\Gamma$ is free. 
In particular, this yields as many geometric roots as there are classical roots.

Recall that $ADE$ simple singularties are weighted homogeneous, and  the monodromy operator $\rho_*$ on $H_1(M,\partial M;\Z)$ or on $H_1(M;\Z)$ is induced from a weight action.
In particular, $\rho_*$ is of finite order, say $m$  with $\rho_*^m = \mathrm{id}$. The following is easy to check.
\begin{lemma}
The variation operator $\var$ commutes with the monodromy. Hence we have
$(\rho_* (\alpha),\rho_* (\beta)) = (\alpha,\beta)$ for any $\alpha, \beta \in H_1(M,\partial M;\Z)$.
\end{lemma}
In particular, $\rho_*$ takes geometric roots to geometric roots. 
Therefore the cyclic group $\Z/m\Z$ generated by $\rho_*$ acts on $\Phi_\Gamma$, and we call this group action  the {\em monodromy action}.
%

We can divide $\Phi_\Gamma$ into disjoint union of orbits of the monodromy action. We will show that the monodromy action is free, and
hence each orbit has $m$ elements, and the number of orbits in $\Phi_\Gamma$ is given by dividing the cardinality of $\Phi_\Gamma$ by $m$:
\begin{center}
\begin{tabular}{ c||c|c|c|c|c|c|c } 
    Type ($\Gamma$) & $A_k$ ($k$ even) & $A_k$ ($k$ odd) & $D_k$ ($k$ even) & $D_k$ ($k$ odd) & $E_6$ & $E_7$ & $E_8$ \\
    \hline
    Number of roots ($|\Phi_\Gamma|$) & \multicolumn{2}{|c|}{$k(k+1)$} & \multicolumn{2}{|c|}{$2k(k-1)$} & $72$ & $126$ & $240$ \\
    \hline
    Order of action ($|\rho_*|=m$) & $2(k+1)$ & $k+1$ & $k-1$ & $2(k-1)$ & $12$ & $9$ & $15$ \\ 
    \hline
    Number of orbits ($|\Phi_\Gamma|/m$) & $k/2$ & $k$ & $2k$ & $k$ & $6$ & $14$ & $16$
\end{tabular}
\end{center}
We will exhibit the proof for each type of simple singularity.
\subsection{The case of $A_k$}
Recall that the Milnor fiber $M$ of $A_k$-singularity is a regular $2(k+1)$-gon with $(k+1)$ punctures with edges identified $\pm 3$ pattern,
and regular $(k+1)$-gon inside is the $A_k$-wheel.
\begin{lemma}\label{lemma_Ak root}
    The geometric roots of type $A_k$ have one-to-one correspondence with the oriented line segments connecting pairs of punctures on the $A_k$-wheel.
\end{lemma}

\begin{proof}
    We label the vertices of the wheel as $v_1,\dots,v_{k+1}$ so that $\alpha_i$ in the basis $\{\alpha_1,\dots,\alpha_k\}$ of $H_1(M,\partial M;\Z)$ in Proposition \ref{cartan matrix}
is an arrow from the vertex $v_i$ to $v_{i+1}$. An oriented line segment from $v_i$ to $v_j$ $(i<j)$ has the homology
    \[
         (\underbrace{0,\dots,0}_{i-1},\underbrace{1,\dots,1}_{j-i},\underbrace{0,\dots,0}_{k-j+1})
    \]
    with respect to the basis $\{\alpha_i\}$.
    Therefore each oriented line segment  in $A_k$-wheel has distinct homology class in $H_1(M,\partial M;\Z)$, say $\alpha$.
    A direct computation shows $(\alpha,\alpha)=2$, and hence they are geometric roots.
    The number of such lines segments are $\tfrac{1}{2}k(k+1)$, and with orientation, it is $k(k+1)$. 
    It equals the cardinality of the $A_k$ root system.
\end{proof}

\begin{prop}\label{free_Ak}
    The monodromy action on $\Phi_{A_k}$ is free.
\end{prop}
\begin{proof}
Monodromy for $A_k$-Milnor fiber is given by $(1,1) \in \Z/(k+1)\Z \oplus \Z/2\Z$, and for an oriented line segment $L$ inside the $A_k$-wheel, its monodromy image $\rho(L)$ lies outside the $A_k$-wheel.
    However, $\rho(L)$ is homologous to the  line obtained by rotating $L$ through an angle of $\tfrac{2\pi}{k+1}$ and inverting its orientation.
    This is because the monodromy acts in this manner on the basis elements $\alpha_1,\dots,\alpha_k$.

    With the rule of the monodromy operator on the $A_k$-wheel, it can be readily proved that the action is free.
    Let $\alpha$ be the relative homology class of $L$.
    
    First suppose that $k$ is an even number. 
    The monodromy images $\rho_*(\alpha),\dots,\rho_*^{k}(\alpha)$ are distinct from $\alpha$. 
    Note that $\rho_*^{k+1}(\alpha) = -\alpha$ corresponds to the orientation reversal of the line $L$. 
    Similarly, $\rho_*^{k+2}(\alpha),\dots,\rho_*^{2k+1}(\alpha)$ are orientation reversals of the former.
    Therefore, the monodromy action on $\Phi_{A_k}$ ($k$ even) is a free $\Z/(2k+2)\Z$-action.

    Suppose that $k$ is an odd number. The monodromy images $\rho_*(\alpha),\dots,\rho_*^{k}(\alpha)$ are distinct from $\alpha$, and we have $  \rho_*^{k+1}(\alpha) = \alpha$.   
    Therefore, the monodromy action on $\Phi_{A_k}$ ($k$ odd) is a free $\Z/(k+1)\Z$-action.
\end{proof}

\subsection{The case of $D_k$}
Recall that $D_k$-wheel is a regular $2(k-1)$-gon with punctures at each of its vertices, along with an additional puncture at the center. 
$D_k$-Milnor fiber is obtained by identifying opposite edges of the $D_k$-wheel.
\begin{lemma}\label{Dk roots}
    The geometric roots of $D_k$ are realized by the set of equivalence classes of oriented line segments connecting pairs of punctures on the $D_k$-wheel.
\end{lemma}
\begin{proof}
 With respect to  the basis $\{\alpha_1,\dots,\alpha_k\}$ of Proposition \ref{cartan matrix},
   consider the following homology classes in $H_1(M,\partial M;\Z)$:
    \begin{itemize}
            \item [(i)] $\pm(0,0,\underbrace{0,\dots,0}_{k_1},\underbrace{1,\dots,1}_{k_2},\underbrace{0,\dots,0}_{k_3})$ \quad\quad ($k_1+k_2+k_3=k-2$, $k_2\ge 1$).
            \item [(ii)] $\pm(1,1,\underbrace{2,\dots,2}_{k_1},\underbrace{1,\dots,1}_{k_2},\underbrace{0,\dots,0}_{k_3})$ \quad\quad ($k_1+k_2+k_3=k-2$, $k_2\ge 1$).
            \item [(iii)] $\pm(1,0,\underbrace{1,\dots,1}_{k_1},\underbrace{0,\dots,0}_{k_2})$ \quad\quad ($k_1+k_2=k-2$).
            \item [(iv)] $\pm(0,1,\underbrace{1,\dots,1}_{k_1},\underbrace{0,\dots,0}_{k_2})$ \quad\quad ($k_1+k_2=k-2$).
    \end{itemize}

    \begin{figure}[h]
    \centering
    \begin{subfigure}{0.23\textwidth}
    \includegraphics[scale=0.75]{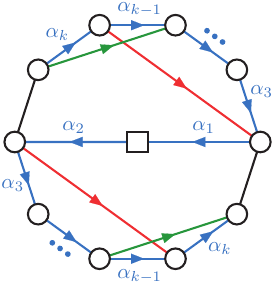}
    \centering
    \caption{Type (i).}
    \end{subfigure}
    \begin{subfigure}{0.23\textwidth}
    \includegraphics[scale=0.75]{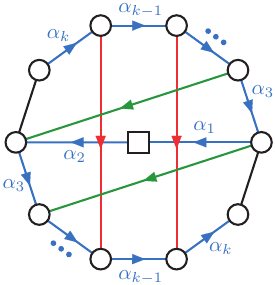}
    \centering
    \caption{Type (ii).}
    \end{subfigure}
    \begin{subfigure}{0.23\textwidth}
    \includegraphics[scale=0.75]{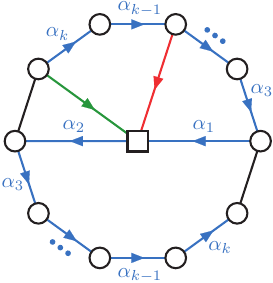}
    \centering
    \caption{Type (iii).}
    \end{subfigure}
    \begin{subfigure}{0.23\textwidth}
    \includegraphics[scale=0.75]{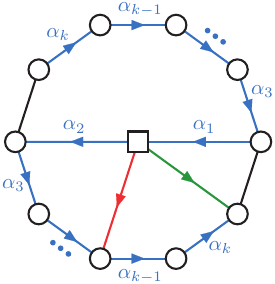}
    \centering
    \caption{Type (iv).}
    \end{subfigure}
    \caption{Geometric roots on $D_k$-wheel.}
    \label{fig:Dkroots}
    \end{figure}
    
    These homology classes are all represented by lines segments of  the $D_k$-wheel.
    Recall that  two line segments are considered equivalent if they determine the same relative homology class. 
    Unlike $A_k$, we will find that there are non-trivial equivalence relations (for type (i) and (ii)).
      
    Now, a homology class of type (i) corresponds to a line segment lying in one of the convex regions enclosed by the lines $\alpha_3,\dots,\alpha_k$.
    As we are identifying opposite edges of the $2(k-1)$-gon, each homology class of type (i) can be represented by one of the two parallel lines as in the Figure \ref{fig:Dkroots} (A).
    
    Type (ii) includes lines in the wheel that connect pairs of vertices of the $2(k-1)$-gon such that one vertex lies in the upper half and the other in the lower half. 
    When $k_1=0$, type (ii) also includes lines that connect the source of $\alpha_1$ with a vertex in the lower half (or its parallel line).
    As in type (i), each class of type (ii) can be represented by one of the two parallel lines as in the Figure \ref{fig:Dkroots} (B).        
       
    Lines of types (iii) and (iv) are those joining the center of the $2(k-1)$-gon with one of its vertices. 
    Unlike (i) and (ii), two parallel lines in type (iii) and (iv) are not equivalent to each other. 
    For example, $\alpha_1$ and $\alpha_2$ are parallel but not equivalent.
   
    It can be seen that every line segment on the $D_k$-wheel can be classified as one of types (i)–(iv).
    All elements $\alpha$ in (i)$-$(iv) are geometric roots from the direct computation of $(\alpha,\alpha)=2$.
 
    The number of  homology classes for each type is as follows: types (i) and (ii) each contain $(k-1)(k-2)$ elements, and types (iii) and (iv) each contain $2(k-1)$ elements. 
    In total, there are $2k(k-1)$ geometric roots, which is identical to the cardinality of $D_k$ root system.
    Therefore, all $D_k$ geometric roots are represented by the above oriented line segments.  
\end{proof}


\begin{prop}\label{free_Dk}
    The monodromy action on $\Phi_{D_k}$ is free.
\end{prop}

\begin{proof}
    The monodromy image $\rho(L)$ of $L$ is given by rotating $L$ in the $D_k$-wheel with an angle of $\tfrac{k \pi}{(k-1)}$.
    Hence for even $k$, the map $\rho$ generates a cyclic rotation of $L$ of order $k-1$, and the induced monodromy action on $\Phi_{D_k}$ is a free $\Z/(k-1)\Z$-action.
    For odd $k$, the map $\rho$ generates a cyclic rotation of $L$ of order $2k-2$, and the induced monodromy action on $\Phi_{D_k}$ is a free $\Z/(2k-2)\Z$-action.
\end{proof}

\subsection{The case of $E_6$}
Analogously to the $A_k$ and $D_k$-cases, one can figure out all the relative homology classes of the oriented line segments on $E_6$-wheel, prove that they are geometric roots of $E_6$, and analyze their monodromy images by brute force.

    \begin{figure}[h]
        \centering
        \includegraphics[scale=0.5]{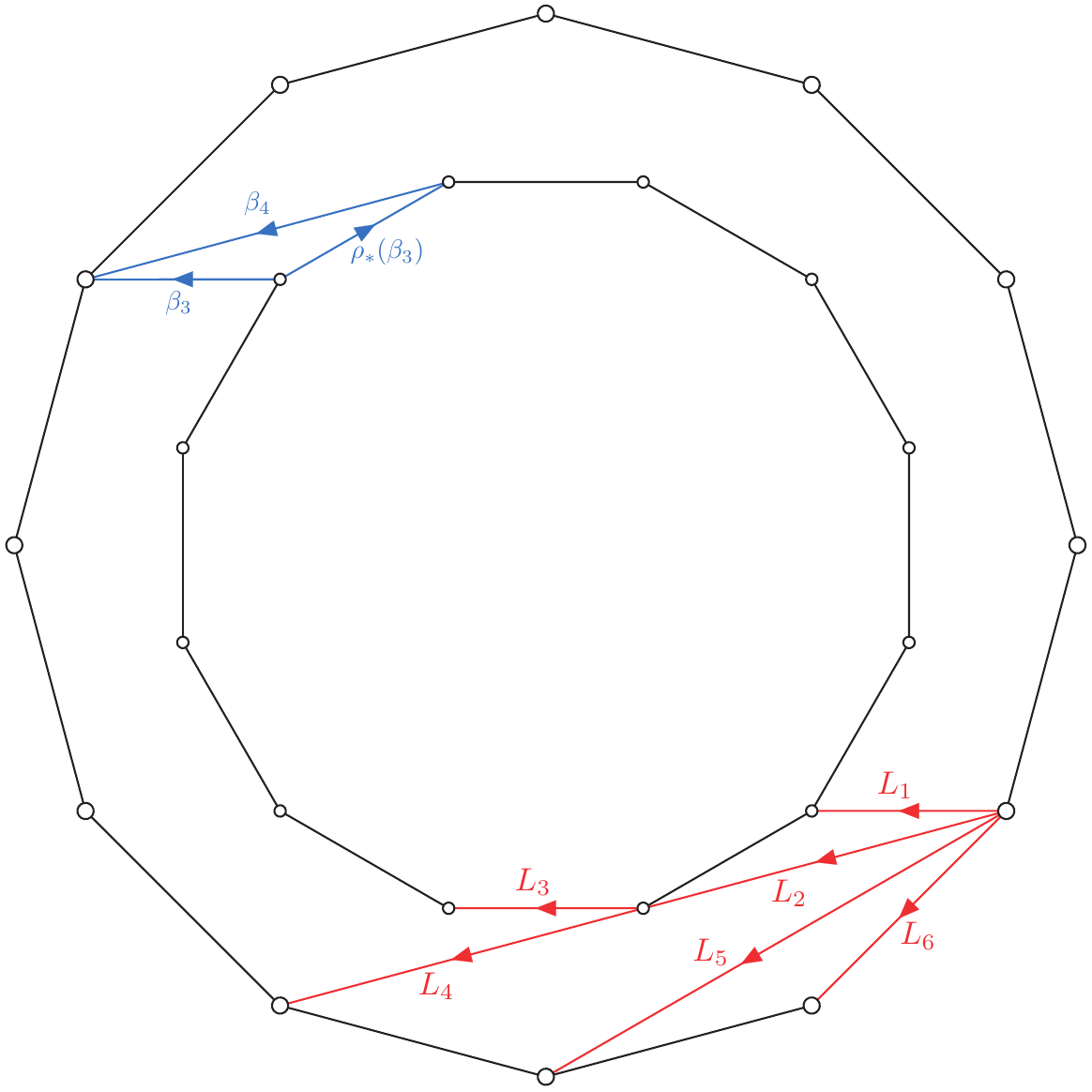}
        \caption{A projective basis in $E_{6}$-wheel.}
        \label{fig:E6spokes}
    \end{figure}

Instead, we introduce a more refined method to show that the monodromy action on $\Phi_{E_6}$ is free.
\begin{prop}\label{free_E6}
    The monodromy action on $\Phi_{E_6}$ is free.
\end{prop}

\begin{proof}
    We consider the lines $L_1,\dots,L_6$ on $E_6$-wheel as in Figure \ref{fig:E6spokes}.
    Let $\beta_1,\dots,\beta_6$ denote the relative homology classes corresponding to $L_1,\dots,L_6$, respectively.
    These are geometric roots (i.e., $(\beta_j,\beta_j)=2$ for each $j$) and this
     can be checked by presenting each $\beta_j$ by a linear combination of $\alpha_j$'s:
    \[
    \begin{array}{rl}
        \beta_1 & =\; \alpha_1 \\
        \beta_2 & =\; \alpha_1 + \alpha_2 \\
        \beta_3 & =\; \alpha_3  \\
        \beta_4 & =\; \alpha_3 + \alpha_4 \\
        \beta_5 & =\; \alpha_1 + \alpha_2 + \alpha_3 + \alpha_4 + \alpha_5  \\
        \beta_6 & =\; \alpha_1 + \alpha_2 + \alpha_3 + \alpha_4 + \alpha_5 + \alpha_6 
    \end{array}
    \quad\text{, i.e., }\quad
    \begin{pmatrix}
        \beta_1 \\ \beta_2 \\ \beta_3 \\ \beta_4 \\ \beta_5 \\ \beta_6
    \end{pmatrix}
    = 
    \begin{pmatrix}
        1 & 0 & 0 & 0 & 0 & 0 \\
        1 & 1 & 0 & 0 & 0 & 0 \\
        0 & 0 & 1 & 0 & 0 & 0 \\
        0 & 0 & 1 & 1 & 0 & 0 \\
        1 & 1 & 1 & 1 & 1 & 0 \\
        1 & 1 & 1 & 1 & 1 & 1
    \end{pmatrix}
    \begin{pmatrix}
        \alpha_1 \\ \alpha_2 \\ \alpha_3 \\ \alpha_4 \\ \alpha_5 \\ \alpha_6 
    \end{pmatrix}.
    \]
    The bilinear form $(\beta_j,\beta_j)$ is easily calculated from this.
    Moreover, the elements $\beta_1,\dots,\beta_6$ form a basis of $H_1(M,\partial M;\Z)$, since the presented matrix is invertible.

    We can calculate the action of monodromy on $\beta_j$'s explicitly.
    For example, $\rho_*(\beta_3)= \beta_3 - \beta_4$ as shwon in Figure \ref{fig:E6spokes}.
    The matrix $P$ for the monodromy operator $\rho_*$ on $H_1(M,\partial M;\Z)$ with respect to the basis $\{\beta_1,\dots,\beta_6\}$ is given by
    \[
    P = 
    \begin{pmatrix}
        1 & 1 & 0 & 0 & 1 & 1 \\
        -1 & 0 & 0 & 0 & 0 & 0 \\
        0 & 0 & 1 & 1 & 1 & 1 \\
        0 & 0 & -1 & 0 & 0 & 0 \\
        0 & -1 & 0 & -1 & -1 & -1 \\
        0 & 0 & 0 & 0 & -1 & 0
    \end{pmatrix}.
    \]
    The $j$-th column of the powered matrix $P^\lambda$ represents the element $\rho_*^\lambda(\beta_j)$ in $H_1(M,\partial M;\Z)$ with respect to the basis $\{\beta_1,\dots,\beta_6\}$.
    We can calculate the powers of this matrix $P^2,\dots,P^{12}$.
    Note that $P^{12}=\mathrm{Id}$, which is obvious due to $\rho^{12} = \mathrm{id_{M}}$.
    Let 
    \[
    \Psi_j := \{\beta_j,\rho_*(\beta_j),\dots,\rho_*^{11}(\beta_j)\}, \quad(j=1,\dots,6)
    \]
    be the orbits of $\beta_j$ under the monodromy action on $\Phi_{E_6}$.
    
    Recall that the cardinality of the $E_6$ root system is 72.
    To show that the monodromy action on $\Phi_{E_6}$ is free, it is enough to prove that
    \begin{itemize}
        \item [i)] the 6 orbits $\Psi_j$ are pairwise disjoint, and
        \item [ii)] each orbit has 12 elements.
    \end{itemize}
    These can be verified by checking the columns of the matrices $P,P^2,\dots,P^{11}$.
    None of them coincides with a column of the identity matrix; that is, none is equal to a vector of the form $(0,\dots,1,\dots,0)^t$.
    This observation establishes both (i) and (ii).
\end{proof}

Since all 72 roots are realized through the monodromy action of $L_1,\dots, L_6$, we obtain the following corollary.
\begin{cor}
    Each geometric root of type $E_6$ corresponds to an edge/spoke on the $E_6$-wheel.
\end{cor}
We will prove the converse in the next section.

\subsection{The case of $E_7$}
We proceed as in the case of $E_6$.
\begin{prop}\label{free_E7}
    The monodromy action on $\Phi_{E_7}$ is free.
\end{prop}

\begin{proof}
    \begin{figure}[h]
        \centering
        \includegraphics[scale=0.5]{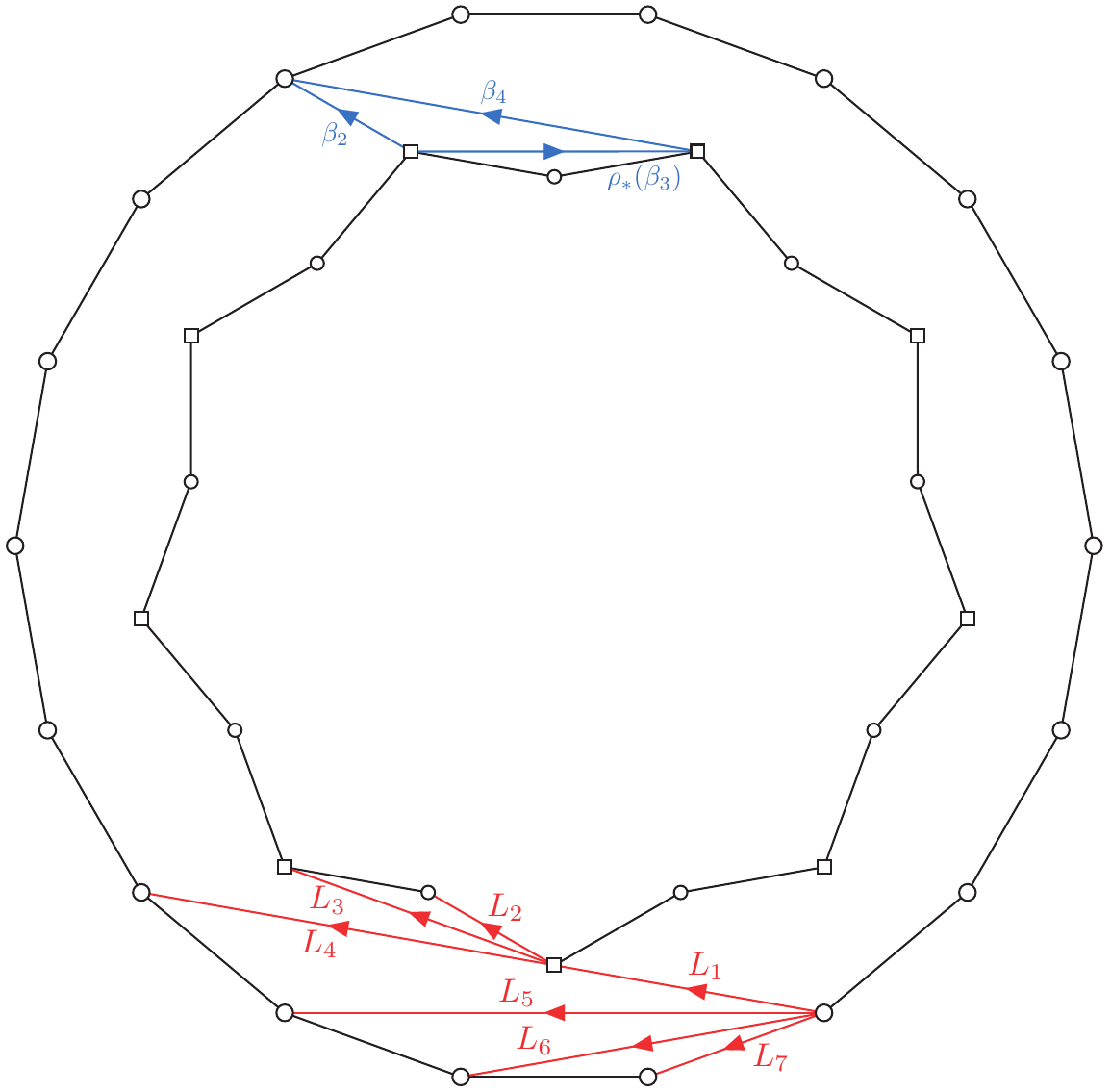}
        \caption{A projective basis in $E_{7}$-wheel.}
        \label{fig:E7spokes}
    \end{figure}
    We consider the lines $L_1,\dots,L_7$ on $E_7$-wheel as in Figure \ref{fig:E7spokes},
    and their relative homology classes $\beta_1,\dots,\beta_7$.
    These give another basis consisting of geometric roots, and this can be checked by presenting each $\beta_j$ by a linear combination of $\alpha_j$'s:
    \[
    \begin{array}{rl}
        \beta_1 & =\; \alpha_1 \\
        \beta_2 & =\; \alpha_2 \\
        \beta_3 & =\; \alpha_2 + \alpha_3 \\
        \beta_4 & =\; \alpha_2 + \alpha_3 + \alpha_4\\
        \beta_5 & =\; \alpha_1 + \alpha_2 + \alpha_3 + \alpha_4 + \alpha_5 \\
        \beta_6 & =\; \alpha_1 + \alpha_2 + \alpha_3 + \alpha_4 + \alpha_5 + \alpha_6 \\
        \beta_7 & =\; \alpha_1 + \alpha_2 + \alpha_3 + \alpha_4 + \alpha_5 + \alpha_6 + \alpha_7 \\
    \end{array}
    \quad\text{, i.e., }\quad
    \begin{pmatrix}
        \beta_1 \\ \beta_2 \\ \beta_3 \\ \beta_4 \\ \beta_5 \\ \beta_6 \\ \beta_7
    \end{pmatrix}
    = 
    \begin{pmatrix}
        1 & 0 & 0 & 0 & 0 & 0 & 0 \\
        0 & 1 & 0 & 0 & 0 & 0 & 0 \\
        0 & 1 & 1 & 0 & 0 & 0 & 0 \\
        0 & 1 & 1 & 1 & 0 & 0 & 0 \\
        1 & 1 & 1 & 1 & 1 & 0 & 0 \\
        1 & 1 & 1 & 1 & 1 & 1 & 0 \\
        1 & 1 & 1 & 1 & 1 & 1 & 1
    \end{pmatrix}
    \begin{pmatrix}
        \alpha_1 \\ \alpha_2 \\ \alpha_3 \\ \alpha_4 \\ \alpha_5 \\ \alpha_6 \\ \alpha_7 
    \end{pmatrix}.
    \]
       We can calculate the action of monodromy on $\beta_j$'s explicitly.
    For example, $\rho_*(\beta_3)= \beta_2 - \beta_4$ as shown in Figure \ref{fig:E7spokes}.
    The matrix $P$ for the monodromy operator $\rho_*$ on $H_1(M,\partial M;\Z)$ with respect to the basis $\{\beta_1,\dots,\beta_7\}$ is given by 
    \[
    P = 
    \begin{pmatrix}
        1 & 0 & 0 & 0 & 1 & 1 & 1 \\
        0 & 1 & 1 & 1 & 1 & 1 & 1 \\
        0 & -1 & 0 & 0 & 0 & 0 & 0 \\
        0 & 0 & -1 & 0 & 0 & 0 & 0 \\
        -1 & 0 & 0 & -1 & -1 & -1 & -1 \\
        0 & 0 & 0 & 0 & -1 & 0 & 0 \\
        0 & 0 & 0 & 0 & 0 & -1 & 0
    \end{pmatrix}.
    \]
    The $j$-th column of the powered matrix $P^\lambda$ represents the element $\rho_*^\lambda(\beta_j)$ in $H_1(M,\partial M;\Z)$ with respect to the basis $\{\beta_1,\dots,\beta_7\}$.
    We can calculate the powers of this matrix $P^2,\dots,P^{9}$.
    Note that $P^{9}=\mathrm{Id}$, which is obvious due to $\rho^{9} = \mathrm{id_{M}}$.
    Let 
    \[
    \Psi^+_j := \{\beta_j,\rho_*(\beta_j),\dots,\rho_*^{8}(\beta_j)\},\quad
    \Psi^-_j := \{-\beta_j,-\rho_*(\beta_j),\dots,-\rho_*^{8}(\beta_j)\}, \quad(j=1,\dots,7)
    \]
    be the orbits of $\beta_j$ and $-\beta_j$ under the monodromy action on $\Phi_{E_7}$.
    
    Recall that the cardinality of the $E_7$ root system is 126.
    To show that the monodromy action on $\Phi_{E_7}$ is free, it is enough to prove that
    \begin{itemize}
        \item [i)] the 14 orbits $\Psi^+_j,\Psi^-_j$ are pairwise disjoint, and
        \item [ii)] each orbit has 9 elements.
    \end{itemize}
    These can be verified by checking the columns of the matrices $P,P^2,\dots,P^{8}$.
    None of them coincides with a column of the identity matrix or its negative; that is, none is equal to a vector of the form $\pm (0,\dots,1,\dots,0)^t$.
    This observation establishes both (i) and (ii).
\end{proof}

\begin{cor}
    Each geometric root of type $E_7$ corresponds to an edge/spoke on the $E_7$-wheel.
\end{cor}

\subsection{The monodromy action on $\Phi_{E_8}$}
We proceed as in the case of $E_6, E_7$.
\begin{prop}\label{free_E8}
    The monodromy action on $\Phi_{E_8}$ is free.
\end{prop}

\begin{proof}
    \begin{figure}[h]
        \centering
        \includegraphics[scale=0.6]{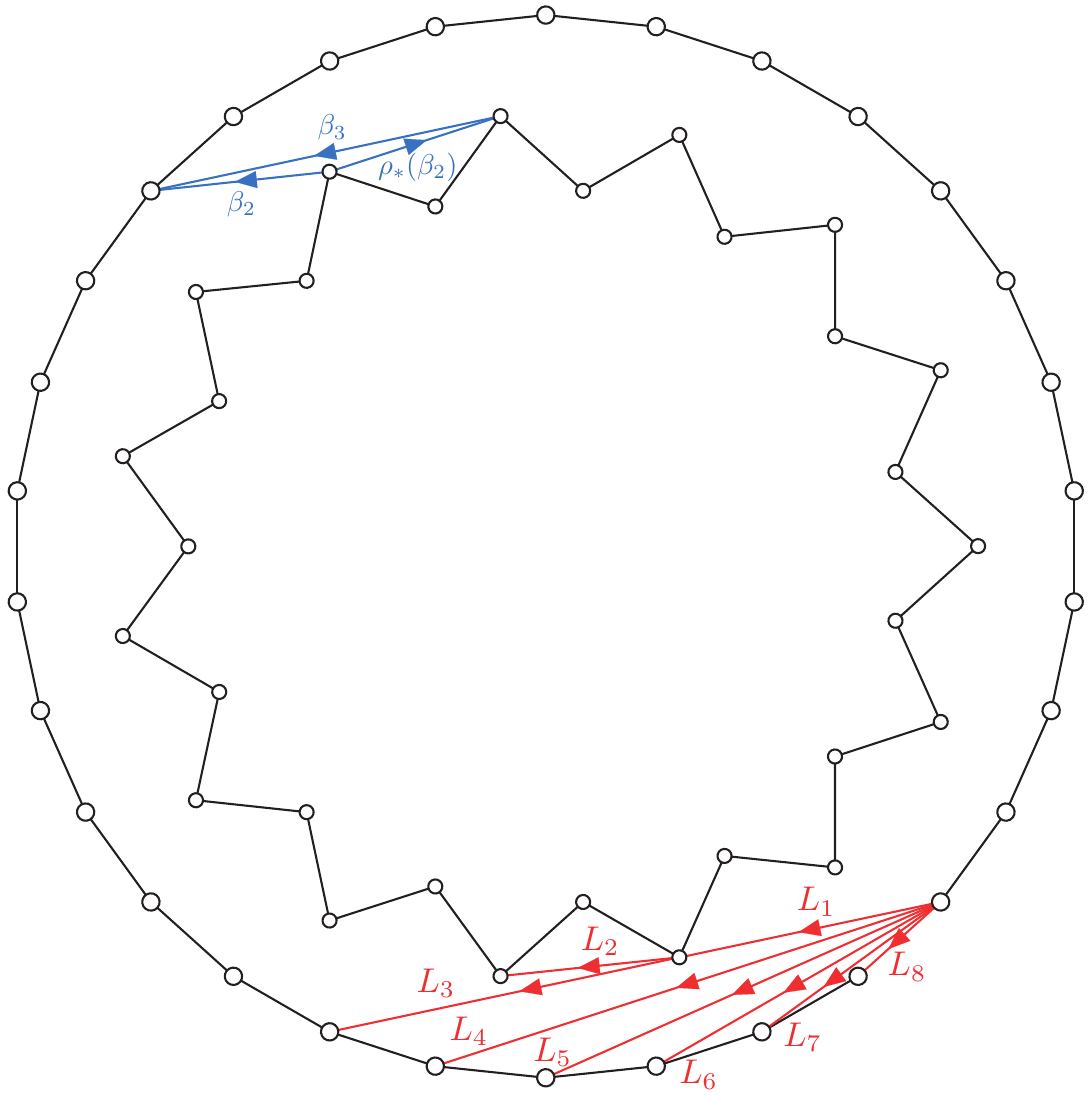}
        \caption{A projective basis in $E_{8}$-wheel.}
        \label{fig:E8spokes}
    \end{figure}
    We consider the lines $L_1,\dots,L_8$ on $E_8$-wheel as in Figure \ref{fig:E8spokes}, and their relative homology classes
    $\beta_1,\dots,\beta_8$.
     These give another basis consisting of geometric roots, and this can be checked by presenting each $\beta_j$ by a linear combination of $\alpha_j$'s:
    \[
    \begin{array}{rl}
        \beta_1 & =\; \alpha_1 \\
        \beta_2 & =\; \alpha_2 \\
        \beta_3 & =\; \alpha_2 + \alpha_3 \\
        \beta_4 & =\; \alpha_1 + \alpha_2 + \alpha_3 + \alpha_4 \\
        \beta_5 & =\; \alpha_1 + \alpha_2 + \alpha_3 + \alpha_4 + \alpha_5 \\
        \beta_6 & =\; \alpha_1 + \alpha_2 + \alpha_3 + \alpha_4 + \alpha_5 + \alpha_6 \\
        \beta_7 & =\; \alpha_1 + \alpha_2 + \alpha_3 + \alpha_4 + \alpha_5 + \alpha_6 + \alpha_7 \\
        \beta_8 & =\; \alpha_1 + \alpha_2 + \alpha_3 + \alpha_4 + \alpha_5 + \alpha_6 + \alpha_7 +\alpha_8
    \end{array}
    \quad\text{, i.e., }\quad
    \begin{pmatrix}
        \beta_1 \\ \beta_2 \\ \beta_3 \\ \beta_4 \\ \beta_5 \\ \beta_6 \\ \beta_7 \\ \beta_8
    \end{pmatrix}
    = 
    \begin{pmatrix}
        1 & 0 & 0 & 0 & 0 & 0 & 0 & 0 \\
        0 & 1 & 0 & 0 & 0 & 0 & 0 & 0 \\
        0 & 1 & 1 & 0 & 0 & 0 & 0 & 0 \\
        1 & 1 & 1 & 1 & 0 & 0 & 0 & 0 \\
        1 & 1 & 1 & 1 & 1 & 0 & 0 & 0 \\
        1 & 1 & 1 & 1 & 1 & 1 & 0 & 0 \\
        1 & 1 & 1 & 1 & 1 & 1 & 1 & 0 \\
        1 & 1 & 1 & 1 & 1 & 1 & 1 & 1
    \end{pmatrix}
    \begin{pmatrix}
        \alpha_1 \\ \alpha_2 \\ \alpha_3 \\ \alpha_4 \\ \alpha_5 \\ \alpha_6 \\ \alpha_7 \\ \alpha_8
    \end{pmatrix}.
    \]
    We can calculate the action of monodromy on $\beta_j$'s explicitly.
    For example, $\rho_*(\beta_2)= \beta_2 - \beta_3$ as shown in Figure \ref{fig:E8spokes}.
    The matrix $P$ for the monodromy operator $\rho_*$ on $H_1(M,\partial M;\Z)$ with respect to the basis $\{\beta_1,\dots,\beta_8\}$ is given by
    \[
    P = 
    \begin{pmatrix}
        1 & 0 & 0 & 1 & 1 & 1 & 1 & 1 \\
        0 & 1 & 1 & 1 & 1 &  1& 1 & 1 \\
        0 & -1 & 0 & 0 & 0 & 0 & 0 & 0 \\
        -1 & 0 & -1 & -1 & -1 & -1 & -1 & -1 \\
        0 & 0 & 0 & -1 & 0 & 0 & 0 & 0 \\
        0 & 0 & 0 & 0 & -1 & 0 & 0 & 0 \\
        0 & 0 & 0 & 0 & 0 & -1 & 0 & 0 \\
        0 & 0 & 0 & 0 & 0 & 0 & -1 & 0
    \end{pmatrix}.
    \]
    The $j$-th column of the powered matrix $P^\lambda$ represents the element $\rho_*^\lambda(\beta_j)$ in $H_1(M,\partial M;\Z)$ with respect to the basis $\{\beta_1,\dots,\beta_8\}$.
    We can calculate the powers of this matrix $P^2,\dots,P^{15}$.
    Note that $P^{15}=\mathrm{Id}$, which is obvious due to $\rho^{15} = \mathrm{id_{M}}$.
    Let 
    \[
    \Psi^+_j := \{\beta_j,\rho_*(\beta_j),\dots,\rho_*^{14}(\beta_j)\},\quad
    \Psi^-_j := \{-\beta_j,-\rho_*(\beta_j),\dots,-\rho_*^{14}(\beta_j)\},\quad(j=1,\dots,8)
    \]
    be the orbits of $\beta_j$ and $-\beta_j$ under the monodromy action on $\Phi_{E_8}$.
    
    Recall that the cardinality of the $E_8$ root system is 240.
    To show that the monodromy action on $\Phi_{E_8}$ is free, it is enough to prove that
    \begin{itemize}
        \item [i)] the 16 orbits $\Psi^+_j,\Psi^-_j$ are pairwise disjoint, and
        \item [ii)] each orbit has 15 elements.
    \end{itemize}
    These can be verified by checking the columns of the matrices $P,P^2,\dots,P^{14}$.
    None of them coincides with a column of the identity matrix or its negative; that is, none is equal to a vector of the form $\pm (0,\dots,1,\dots,0)^t$.
    This observation establishes both (i) and (ii).
\end{proof}

\begin{cor}
    Each geometric root of type $E_8$ corresponds to an edge/spoke on the $E_8$-wheel.
\end{cor}

\subsection{Simple or Projective basis}\label{ssec:simple}
We have chosen two basis of $H_1(M,\partial M;\Z)$ consisting of geometric roots.
The basis $\{\alpha_1,\dots,\alpha_k\}$ chosen in Proposition \ref{cartan matrix} is called the {\em simple basis}
and basis $\{\beta_1,\dots,\beta_k\}$ chosen in this section is called the {\em projective basis}.

The projective basis for $A_k$ is given by $\beta_j:=\alpha_1+\cdots+\alpha_j$ ($1 \le j \le k$).
The projective basis for $D_k$ is given by $\beta_1:=\alpha_1$, $\beta_2:=\alpha_2$, and $\beta_j:=\alpha_1+\cdots+\alpha_j$ ($3 \le j \le k$).
It is readily seen that the projective bases for $\Phi_{A_k}$ and $\Phi_{D_k}$ exhibit properties analogous to those of type $E$.
That is, the monodromy action on the projective basis yields a complete set of geometric roots, up to reversing orientation.

These names come from the theory of quiver representations.
Gabriel's theorem says that if a quiver has finitely many isomorphism class of indecomposable representations then
its underlying graph is  of $ADE$ type, and its indecomposable representations are in one to one correspondence with
positive roots of the root system.

Each geometric root may be regarded as an oriented Lagrangian submanifold of the Milnor fiber $M$,
which is a symplectic manifold. There is a version of Fukaya category being developed in \cite{CCJ}, \cite{CCJ2}, \cite{BCCJ}
for a singularity.  We conjecture that each basis generates this Fukaya category.
Moreover $\AI$-endomorphism algebra of $\{\beta_1,\dots,\beta_k\}$ in this Fukaya category is quasi-isomorphic to the path algebra of an $ADE$-quiver. And the hom functor with respect $\{\beta_1,\dots,\beta_k\}$ takes geometric roots to the indecomposable representations(up to shift).
Namely, equivalence classes of edges/spokes of the wheel also classify indecomposable quiver representations. 
In particular, hom functor takes basis $\{\alpha_1,\dots,\alpha_k\}$ to simple representations and takes the basis $\{\beta_1,\dots,\beta_k\}$
to the projective representations. The wheel  provides geometric models for Auslander--Reiten quivers as well (c.f., see \cite{Schbook} for well-known geometric models in $A_{k}$ and $D_{k}$-cases, also \cite{CCS06}, \cite{Sch08}, see \cite{BGMS23} and \cite{CZ24} for similar geometric models in $A_{k}$ and $D_{k}$-cases). This will be explored elsewhere by the first two authors of the paper. This is also where the current
paper is originated from.

This connection can be readily seen as follows. 
In terms of the basis $\{\alpha_1,\dots,\alpha_k\}$, we can write the homology class of each edge/spoke.
Following the notations of quiver representations theory, a homology class $\alpha_1+2\alpha_3+2\alpha_4$ will be
written as $13344$. It is difficult to write every homology classes, so we have chosen a triangulation of each Milnor fiber (hence the wheel)
and wrote the homology classes of their edges in Figure \ref{fig:tri}. The homology classes of rest of the edges/spokes can be easily deduced.
In this way, we can also check that each geometric root is either a positive or negative  linear combination of $\{\alpha_1,\dots,\alpha_k\}$,
which is a property of simple basis for a root system.
    

\section{Equivalent definition of geometric root system (proof of Thm \ref{root_equiv})} \label{sec:6}

Now we prove Theorem \ref{root_equiv}.
Let's denote the sets (ii),(iii), and (iv) in Theorem \ref{root_equiv} by $\Phi_\Gamma^{(2)}$,$\Phi_\Gamma^{(3)}$, and $\Phi_\Gamma^{(4)}$, respectively.

\begin{lemma}\label{lemma_root_subset}
The following statements hold.
\begin{enumerate}
\item
$ \Phi_\Gamma = \Phi_\Gamma^{(2)} \supseteq \Phi_\Gamma^{(3)}$.
\item  $\Phi_\Gamma^{(4)} \supseteq  \Phi_\Gamma^{(3)}.$
\end{enumerate}
\end{lemma}

\begin{proof}
$\Phi_\Gamma$ is equal to $\Phi_\Gamma^{(2)}$, since the bilinear form $(\cdot,\cdot)$ is positive-definite.

\
For the relation $\Phi \supseteq \Phi_\Gamma^{(3)}$, we claim that for a connected embedded arc $L$ satisfying $L \cap \rho(L) = \emptyset$,  $\var([L]) \bullet [L]=1$. The latter implies that  $([L],[L]) = 2$.

Consider a small perturbation $L'$ of the arc $L$ such that $L' \cap \rho(L) = \emptyset$ and $L'$ intersects $L$ once with intersection number $+1$.
(As we approach $\partial M$, $L'$ is on the right side of $L$. Since the monodromy is right-veering, $\rho(L)$ is on the left-side of $L$ and does not intersect $L'$).
 Then $\var([L])\bullet [L'] $ comes from $L \cap L'$, and hence $\var([L]) \bullet [L]=1$.
 

We prove $\Phi_\Gamma^{(4)} \supseteq  \Phi_\Gamma^{(3)}$ in two steps.
First, we claim that  given the projective basis $(\beta_1,\ldots,\beta_k)$, the collection of variation images of  $(\beta_k,\ldots, \beta_1)$
forms a distinguished collection of vanishing cycles.
Next, note that the monodromy image of an em vanishing cycle is an em vanishing cycle since one can compose the vanishing path with a loop going around all Morse critical points. In Section \ref{sec:monodromy}, we have shown that the monodromy action images of the projective basis cover all geometric roots, up to reversing orientation.
Hence this proves $\Phi_\Gamma^{(4)} \supseteq  \Phi_\Gamma^{(3)}$.

It remains to show the first claim.
We first recall the following theorem of Gusein-Zade.
\begin{thm}[{\cite[Theorem 1]{GZ80}}]\label{thm_gz80}
Let $f:(\C^n,0) \to (\C,0)$ be one of the simple singularities.  Let $(\Delta_1,\cdots,\Delta_\mu)$ is an integral basis of $H_{n-1}(M;\Z)$
and $(\nabla_1,\ldots,\nabla_\mu)$ is a dual basis of $H_{n-1}(M,\partial M;\Z)$ (via intersection pairing).
 Suppose  the matrix of the variation operator $\var$ (or $\var^{-1}$) is upper-triangular in terms of these bases. Then $(\Delta_1,\cdots,\Delta_\mu)$ is a distinguished collection of vanishing cycles.
\end{thm}

    \begin{figure}[H]
\begin{subfigure}[t]{0.47\textwidth}
\includegraphics[scale=0.55]{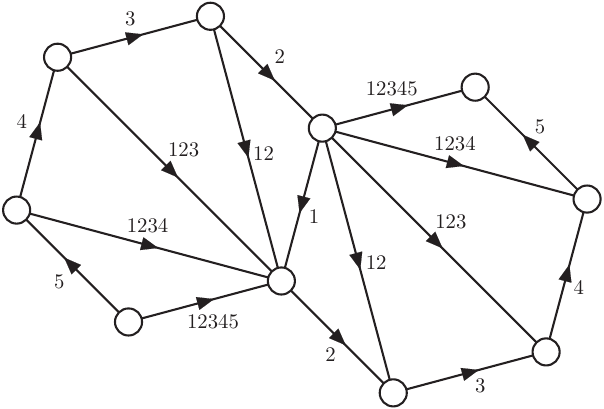}
\centering
\caption{The triangulation of $A_5$-Milnor fiber.}
\end{subfigure}
\begin{subfigure}[t]{0.47\textwidth}
\includegraphics[scale=0.6]{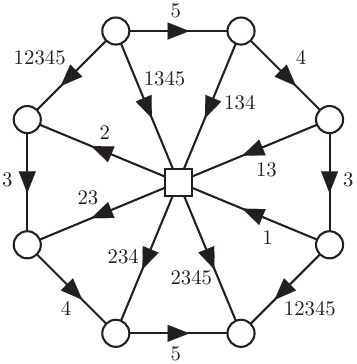}
\centering
\caption{The triangulation of $D_5$-Milnor fiber.}
\end{subfigure}

\begin{subfigure}[t]{0.47\textwidth}
\includegraphics[scale=0.35]{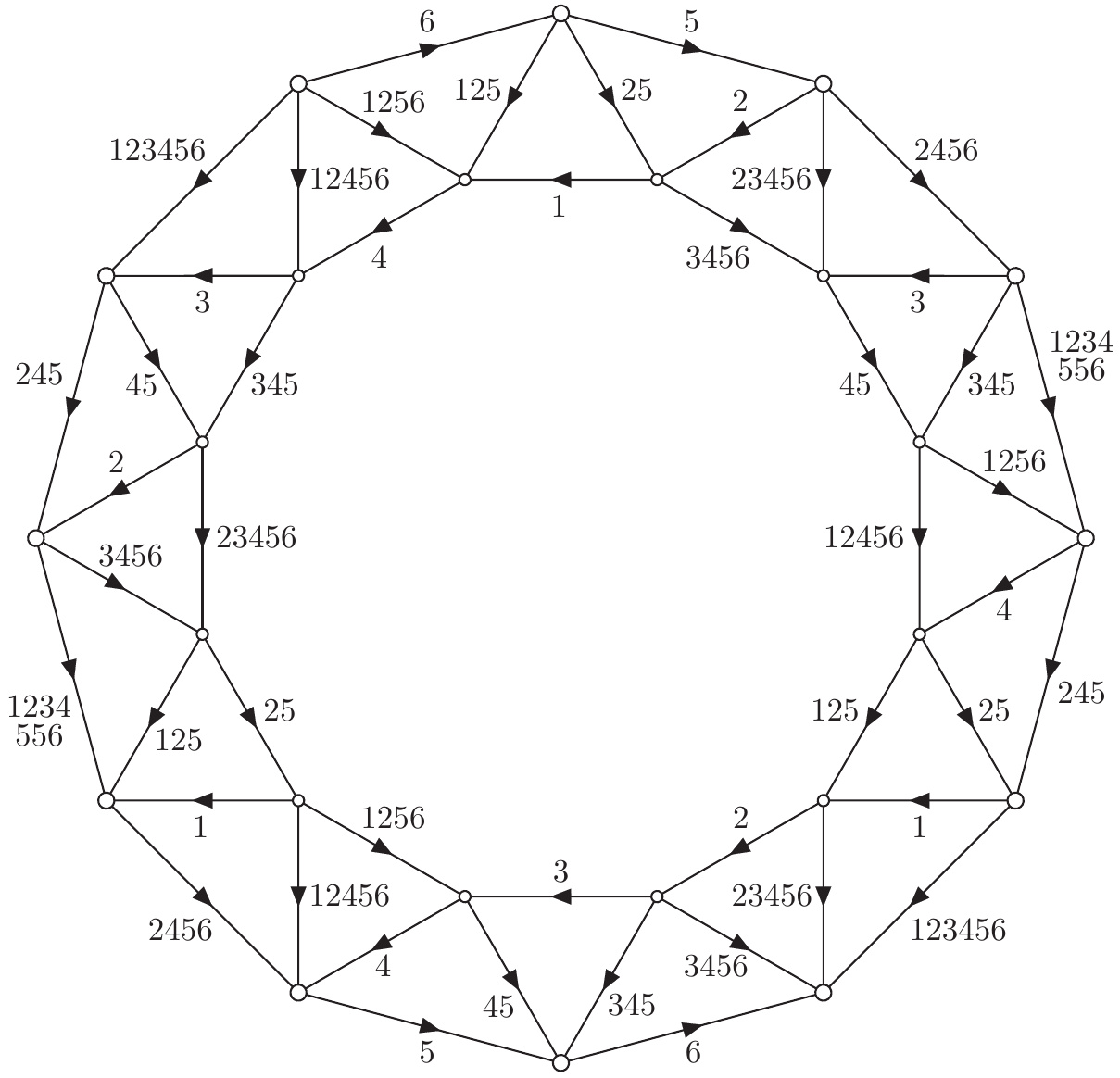}
\centering
\caption{The triangulation of $E_6$-Milnor fiber.}
\end{subfigure}
\begin{subfigure}[t]{0.47\textwidth}
\includegraphics[scale=0.35]{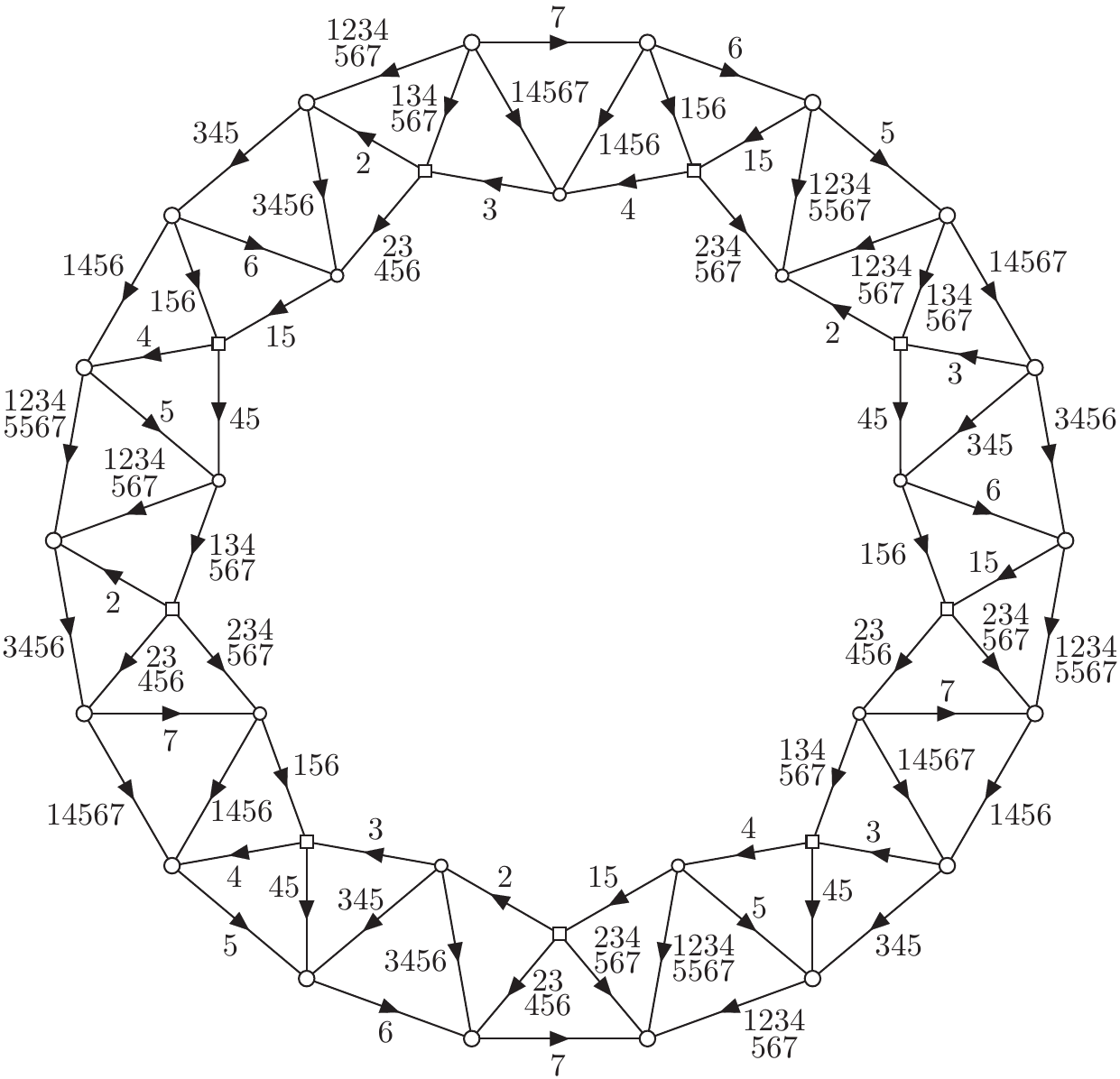}
\centering
\caption{The triangulation of $E_7$-Milnor fiber.}
\end{subfigure}

\begin{subfigure}[t]{0.6\textwidth}
\includegraphics[scale=0.5]{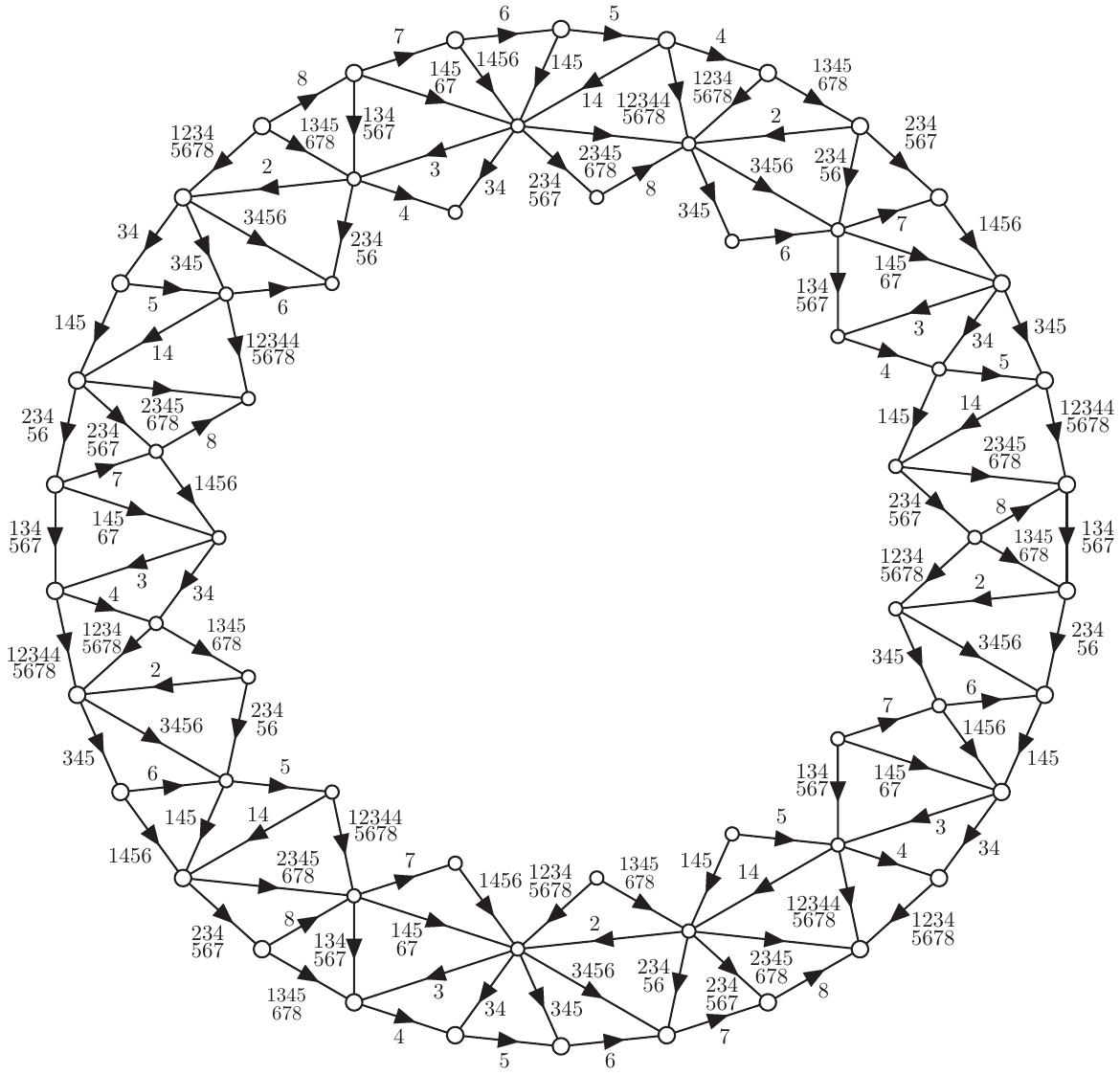}
\caption{The triangulation for $E_8$-Milnor fiber.}
\end{subfigure}
\centering
\caption{The triangulation of Milnor fibers.}
\label{fig:tri}
\end{figure}

For the simple basis $(\alpha_1,\ldots,\alpha_k)$, the matrix in the theorem is exactly $B_{ij}$ and we have checked that they are
upper-triangular. This means that simple basis is in fact a distinguished collection of vanishing cycles.
For the projective basis $(\beta_1,\ldots, \beta_k)$, one can similarly compute the matrix of variation operator and find that it is lower-triangular.
Hence after reversing the ordering as $(\beta_k,\ldots, \beta_1)$, their variation images form a distinguished collection of vanishing cycles.
%
\end{proof}
In order to show that they are all equal, we introduce another set.
\begin{defn}
Define a subset $\Phi^{(5)}_\Gamma$ of the relative homology group $H_1(M,\partial M;\Z)$ by
\begin{center}
    $\Phi^{(5)}_\Gamma$ := $\{ \alpha\in H_1(M,\partial M;\Z)$ | $\alpha$ is represented by an oriented line segment in the Coxeter wheel $\}$.
\end{center}
\end{defn}

We will prove the following lemma in the rest of the section.
\begin{lemma}\label{lem:sub}
The following statements hold.
\begin{enumerate}
\item  $\Phi_\Gamma \subseteq \Phi^{(5)}_\Gamma$.
\item  $\Phi^{(5)}_\Gamma \subseteq \Phi_\Gamma^{(3)}$.
\item $|\Phi_\Gamma^{(4)}| \le |\Phi_\Gamma|$.
\end{enumerate}
\end{lemma}
Combining above two lemmata, we have
$$\Phi_\Gamma= \Phi_\Gamma^{(2)}=\Phi_\Gamma^{(3)}= \Phi_\Gamma^{(4)}=\Phi_\Gamma^{(5)}.$$
This proves Theorem \ref{root_equiv}.

Lemma \ref{lem:sub} (1)
means that every geometric root is represented by an oriented line segment in the Coxeter wheel, and this was proved in the last section.

For Lemma \ref{lem:sub} (2), we prove the following.

\begin{prop}
    Let $L$ be an oriented line segment in a Coxeter wheel.
    Then its image $\rho(L)$ under the geometric monodromy is disjoint with $L$.
    Consequently, $\Phi^{(5)}_\Gamma \subseteq \Phi^{(3)}_\Gamma$.
\end{prop}

\begin{proof}
Although the general approach is the same for all types, we give specific justifications for each type separately.
\begin{enumerate}
    \item The case of $A_k$: 
    
From Proposition \ref{free_Ak}, for the boundary edges of the $A_k$-wheel, monodromy $\rho$ rotates them  counterclockwise by $\frac{2\pi}{k+1}$
and reverses the orientation. Hence $L$ and $\rho(L)$ are disjoint. For a line segment $L$ strictly inside $A_k$-wheel, note that the geometric monodromy $\rho$
takes the interior of the $A_k$-wheel to its complement in the $A_k$-Milnor fiber. Hence  $L$ and $\rho(L)$ are disjoint as well.

    \item  The case of $D_k$:
    
Recall that   the monodromy action on $D_k$-Milnor fiber is given by rotating the $2(k-1)$-gon counterclockwise by $\frac{k \pi}{k-1}$.
A diagonal $L$ in a regular $2(k-1)$-gon does not overlap with itself after the rotation of an angle $\frac{k \pi}{k-1}$,  unless $L$ is connecting two opposite vertices. But such $L$ is not allowed in the $D_k$-wheel due to the puncture at the center. 
For spokes to the center or boundary edges of $2(k-1)$-gon, it is easy to check that  $L \cap \rho(L) = \emptyset$ holds.

    \item The case of $E_6$:
    
We find a region $Y$ of the wheel so that the interiors of $Y$ and $\rho(Y)$ do not intersect at all.
In Figure \ref{fig:tri} (C), the region $Y$ is given by  the convex region enclosed by the lines `3',`4',`5',`6', and `3456' at the bottom of Figure.
Since  the geometric monodromy on the $E_6$-wheel is given by  the counterclockwise rotation of the wheel by $\tfrac{7}{6}\pi$,
$\rho(Y)$ is the convex region enclosed by the lines `3',`4',`125',`6', and `123456' at the upper-left part of Figure \ref{fig:tri} (C).
 
To see that  the interiors of $Y$ and $\rho(Y)$ do not intersect at all, note that the region $Y$ is composed of $B_{1,4},B_{1,3},B_{2,3},B_{2,2}$, and $A_3$, while $\rho(Y)$ is composed of $B_{2,4},B_{2,1},B_{3,3},B_{3,4}$, and $A_4$.

From this, one can see that if $L$ is a line inside the region $Y$, then we have $L \cap \rho(L) =\emptyset$. 
We can also check directly that $L \cap \rho(L) = \emptyset$ holds  when $L$ is one of the boundary edges of $Y$
    
Now the entire $E_6$-wheel can be covered by the union of regions $Y,\rho(Y),\dots,\rho^{11}(Y)$.
Moreover, every edge and spoke in the $E_6$-wheel lies in at least one of these regions. Since $\rho$ is a diffeomorphism on the $E_6$-Milnor fiber, we can deduce that $\rho^{j}(L) \cap \rho^{j+1}(L) = \emptyset$ for every integer $j$.
Therefore, we get $L \cap \rho(L) = \emptyset$ for every edge/spoke $L$.

    \item The case of $E_7$:

We proceed as in the case of $E_6$.
In Figure \ref{fig:tri} (D), the region $Y$ is defined by the convex region enclosed by the lines `2',`3',`4',`5',`6',`7', and `234567' at the bottom of Figure. Since the geometric monodromy on the $E_7$-wheel is the counterclockwise rotation of the wheel by $\tfrac{10}{9}\pi$,
$\rho(Y)$ is the convex region enclosed by the lines `2',`3',`4',`15',`6',`7', and `1234567' at the top of Figure \ref{fig:tri} (D).

The interiors of $Y$ and $\rho(Y)$ do not intersect because the region $Y$ is composed of $B_{1,8},B_{1,9},B_{1,3},$ $B_{1,4}$, and $A_3$, while $\rho(Y)$ is composed of $B_{1,6},B_{1,7},B_{1,1},B_{1,2}$, and $A_1$.
One can check that  $L \cap \rho(L) = \emptyset$ holds whenever $L$ is a line in the inside of $Y$ or at the boundary of $Y$.

The entire $E_7$-wheel can be covered by the union of regions $Y,\rho(Y),\dots,\rho^{8}(Y)$.
Every edge and spoke of the $E_7$-wheel lies within the union, except for the spoke $K$ and its monodromy action images $\rho(K),\dots,\rho^8(K)$.
Here, the spoke $K$ corresponds to the sum of `14567',`5', and `6' at the bottom-left of Figure \ref{fig:tri} (D).
It is straightforward to check that $K \cap \rho(K)= \emptyset$, and it similarly follows that $\rho^{j}(K) \cap \rho^{j+1}(K) = \emptyset$ for every integer $j$. The rest of the argument is similar to the case of $E_6$ and
we get $L \cap \rho(L) = \emptyset$ for every edge/spoke $L$ of the $E_7$-wheel.

    \item The case of $E_8$:

We proceed as in the case of $E_6$ and $E_7$, but this needs some adjustments.
In Figure \ref{fig:tri} (E),  the region $Y$ is defined by the convex region enclosed by the lines  `2',`3',`4',`5',`6',`7',`8', and `2345678' at the bottom of Figure. Since the geometric monodromy on the  $E_8$-wheel  is the counterclockwise rotation of the wheel by $\tfrac{16}{15}\pi$,
 $\rho(Y)$ is the convex region enclosed by the lines `2',`3',`14',`5',`6',`7', and `12345678' at the top of Figure \ref{fig:tri} (E).

The interiors of $Y$ and $\rho(Y)$ do not intersect, but this needs more careful check.
Note that the region $Y$ is composed of some of  $B_1,B_2,B_3$ and $A_5$ and $A_2$, while $\rho(Y)$ is composed of some part of $B_1,B_2,B_3$ and $A_3$ and $A_1$.  
Although both $Y$ and $\rho(Y)$ occupy portions of $B_1$, $B_2$, and $B_3$, one can check that they do not share any common area.
One can check that  $L \cap \rho(L) = \emptyset$ holds whenever $L$ is a line in the inside of $Y$ or at the boundary of $Y$.

The union of regions $Y,\rho(Y),\dots,\rho^{14}(Y)$ covers the $E_8$-wheel except for a certain inner area. 
Namely, the triangle enclosed by `145',`14', and `5' is not covered and there are fifteen such triangles in total.
Nevertheless, other than a few exceptions, every edge and spoke (up to equivalence) in the $E_8$-wheel lies within at least one of $Y,\dots,\rho^{14}(Y)$. 
(If one endpoint of the line $L$ lies on an innermost vertex of the $E_8$-wheel, then there exists an equivalent line for $L$ located in the outer wheel.)
The only exceptions are the spoke $K$ and its monodromy action images $\rho(K),\dots,\rho^{14}(K)$, where $K$ corresponds to the sum of `1345678',`4',`5',`6', and `7' at the bottom of Figure \ref{fig:tri} (E).
It is straightforward to check $K \cap \rho(K)= \emptyset$, and we get $L \cap \rho(L) = \emptyset$ for every edge/spoke $L$ in $E_8$. \qedhere
\end{enumerate}
\end{proof}

Now it remains to prove Lemma \ref{lem:sub} (3). For this, we will consider the set of em vanishing cycles of $f(x,y)$ and compare it with
that of its stabilization  $\Tilde{f}(x,y,z) = f(x,y)+z^2$.
\begin{prop}
Let $f(x,y)$ be a simple curve singularity of type $\Gamma$, and $\Lambda_f \subset H_1(M;\Z)$ be the set of em vanishing cycles of $f$.
Then we have  $|\Lambda_f| \le |\Phi_\Gamma|$. 
Since the variation operator is an isomorphism, it follows that $|\Phi_\Gamma^{(4)}|=|\Lambda_f| \le |\Phi_\Gamma|$.
\end{prop}
\begin{proof}
We will find that the number of vanishing cycles of $f$ is the same as that of $\Tilde{f}$ and compare the latter with  the classical roots of  type $\Gamma$.
  For the Thom--Sebastiani sum of two singularites $g_1,g_2$, Gabrielov \cite{Ga73} proved that distinguished bases of vanishing cycles for $g_1$ and $g_2$ induce a distinguished basis of vanishing cycles for the Thom--Sebastiani sum.
    More precisely, let $\{\Delta^{(1)}_i\}$ be a distinguished basis of $\widetilde{H}_{n_1}(M_{g_1};\Z)$ and $\{\Delta^{(2)}_j\}$ be a distinguished basis of $\widetilde{H}_{n_2}(M_{g_2};\Z)$.
    According to Gabrielov, the elements
    \[
    \Tilde{\Delta}_{i,j} := \iota_* ( \Delta^{(1)}_i \otimes \Delta^{(2)}_j )
    \]
    form a distinguished basis of $\widetilde{H}_{n_1+n_2-1}(M_{g_1\oplus g_2};\Z)$ with a certain ordering of the index $(i,j)$.
    
    We apply this for $g_1=f(x,y)$ and $g_2 = z^2$.
    Since $M_{z^2}$ consists of two points, we may consider an embedding $\iota$ of the suspension of $M_f$ into $M_{\Tilde{f}}$, which is a homotopy equivalence.
    The induced map $\iota_*$ can be considered as an isomorphism from $\widetilde{H}_{1}(M_f;\Z)$ to $\widetilde{H}_2(M_{\Tilde{f}};\Z)$. 
    Gabrielov's theorem asserts that if $\{\Delta_i\}$ is a distinguished basis of $\widetilde{H}_1(M_f;\Z)$, then the elements $\Tilde{\Delta}_i:= \iota_*(\Delta_i)$ form a distinguished basis of $\widetilde{H}_2(M_{\Tilde{f}};\Z)$.

    Now let $\Delta$ be an em vanishing cycle in $H_1(M_f;\Z)$. 
    Then $\Delta$ can be considered as one of among the elements of a distinguished basis $\{\Delta_1,\dots,\Delta_k\}$, so $\iota_*(\Delta)$ is a vanishing cycle in $H_2(M_{\Tilde{f}};\Z)$.
    Hence, we have constructed an injection from the set of em vanishing cycles for $f(x,y)$ to the set of vanishing cycles for $\Tilde{f}(x,y,z)$.
    Recall that the set of vanishing cycles of 3-variable simple singularity of type $\Gamma$ form a root system of type $\Gamma$.
    Therefore, $|\Lambda_f| \le |\Phi_\Gamma|$ is proved.
\end{proof}

\section{Construction of Lie algebras} \label{sec:7}
We have given a geometric definition of Lie algebras in Definition \ref{Lie bracket} using
vanishing arcs, Seifert form and variation operator for each simple singularity $f(x,y)$.
In this section, we verify that they indeed define simple Lie algebras that are isomorphic to the classical ones of the corresponding type.

\begin{thm}\label{thm:Lie bracket}
For $ADE$ curve singularities, $\mathfrak{g}_{geo}$ and the Lie bracket defined in Definition \ref{Lie bracket} define a Lie algebra, which is isomorphic to the simple Lie algebra of the corresponding type.
\end{thm}
\begin{proof}
Let us first check the skew-symmetry of the Lie bracket.  
It is enough to check the following two cases.
For $ \alpha \in \Phi_\Gamma$,  linearity of variation operator implies that
$$[g_{-\alpha},g_\alpha] = -\mathrm{var}(-\alpha)= + \mathrm{var}(\alpha) = - [g_\alpha, g_{-\alpha}].$$
Take two geometric roots $\alpha$ and $\beta$.  If $\alpha+\beta$ is not a geometric root, then the bracket vanishes and skew-symmetry holds automatically.
If $\alpha+\beta \in \Phi_\Gamma$, we have $(\alpha,\beta)=-1$ from the identity 
    $$(\alpha,\beta) = \tfrac{1}{2}\left( (\alpha+\beta,\alpha+\beta)-(\alpha,\alpha)-(\beta,\beta) \right).$$
Then the following identity for the sign $N_{\alpha,\beta}$ of the bracket implies the skew symmetry  $[g_\alpha,g_\beta] = -[g_\beta,g_\alpha]$:
   $$ N_{\alpha,\beta} N_{\beta,\alpha}=(-1)^{\mathrm{var}(\beta)\bullet \alpha + \mathrm{var}(\alpha)\bullet \beta} = (-1)^{(\alpha,\beta)} = -1. $$
 
Let us check the Jacobi Identity:
$$[[A,B],C]+[[B,C],A]+[[C,A],B]=0.$$
 We classify the cases according to the number of inputs in $\mathfrak{h}_{geo}$.
      \begin{enumerate}
        \item Suppose that $A,B,C\in \mathfrak{h}_{geo}$: $[[A,B],C]=[[B,C],A]=[[C,A],B]=0$.
        \item Suppose that two of $A,B,C$ are in $\mathfrak{h}_{geo}$, say $A$ and $B$: Since variation operator is an isomorphism, we can find $\alpha, \beta \in H_1(M,\partial M;\Z)$
        that $A=\mathrm{var}(\alpha)$, $B=\mathrm{var}(\beta)$. From the multilinearity, we may assume that  $\alpha,\beta,\gamma \in \Phi_\Gamma$ and $C=g_\gamma$.
        Then Jacobi identity follows from the following identities:
        \begin{align*}
            [[\mathrm{var}(\alpha),\mathrm{var}(\beta)],g_\gamma]& = 0, \\
            [[\mathrm{var}(\beta),g_\gamma],\mathrm{var}(\alpha)]& = [(\beta,\gamma) g_\gamma, \mathrm{var}(\alpha)] = -(\alpha,\gamma)(\beta,\gamma) g_\gamma, \\
            [[g_\gamma,\mathrm{var}(\alpha)],\mathrm{var}(\beta)]& = [-(\alpha,\gamma) g_\gamma, \mathrm{var}(\beta)] = (\alpha,\gamma)(\beta,\gamma) g_\gamma.
        \end{align*}
        
        \item Suppose that one of $A,B,C$ is in $\mathfrak{h}_{geo}$, say $C=\mathrm{var}(\gamma)$,  $A=g_\alpha$ and $B=g_\beta$ for  $\alpha,\beta,\gamma\in \Phi_\Gamma$:

         If $\alpha+\beta$ is non-zero and not a geometric root, then we get $[[A,B],C]=[[B,C],A]=[[C,A],B]=0$ and Jacobi identity holds. 
         If $\alpha+\beta=0$, then 
         $$[[A,B],C]+ [[B,C],A]+[[C,A],B]= 0  -(\gamma,-\alpha) (-\mathrm{var}(-\alpha)) - (\gamma,\alpha) \mathrm{var}(\alpha) =0.$$
         Assume that $\alpha+\beta \in \Phi_\Gamma$. Then Jacobi identity follows from the following identities:
        \begin{align*}
            [[g_\alpha,g_\beta],\mathrm{var}(\gamma)]& = [N_{\alpha,\beta} g_{\alpha+\beta}, \mathrm{var}(\gamma)] = -(\gamma,\alpha+\beta) N_{\alpha,\beta} g_{\alpha+\beta}, \\
            [[g_\beta,\mathrm{var}(\gamma)],g_\alpha]& = [-(\gamma,\beta) g_\beta, g_\alpha] = (\gamma,\beta) N_{\alpha,\beta} g_{\alpha+\beta}, \\
            [[\mathrm{var}(\gamma),g_\alpha],g_\beta]& = [(\gamma,\alpha) g_\alpha, g_\beta] = (\gamma,\alpha) N_{\alpha,\beta} g_{\alpha+\beta} .
        \end{align*}

        \item Suppose that none of $A,B,C$ is in $\mathfrak{h}_{geo}$, say $A=g_\alpha$, $B=g_\beta$, and $C=g_\gamma$ for $\alpha,\beta,\gamma\in \Phi_\Gamma$. If none of $\alpha+\beta$, $\beta+\gamma$, and $\gamma+\alpha$ are  geometric roots, then we have  $[[A,B],C]=[[B,C],A]=[[C,A],B]=0$ and Jacobi identity holds. Without loss of generality, assume that $\alpha+\beta \in \Phi_\Gamma$. If $\alpha+\beta+\gamma$ is not a geometric root, then we again have  $[[A,B],C]=[[B,C],A]=[[C,A],B]=0$.
        Hence let us assume further that $\alpha+\beta+\gamma \in \Phi_\Gamma$. 
        \begin{itemize}
            \item[i)] Suppose that at least one of $\beta+\gamma$ and $\gamma+\alpha$ is a geometric roots. Without loss of generality, assume that $\beta+\gamma \in \Phi_\Gamma$. 
            Then, we observe that $\gamma+\alpha$ is not a root as $(\gamma,\alpha) \neq 1$:             \begin{align*}
                2(\gamma,\alpha) & = (\alpha+\beta+\gamma,\alpha+\beta+\gamma) - (\alpha,\alpha) - (\beta,\beta) - (\gamma,\gamma) - 2 (\alpha,\beta) - 2 (\beta,\gamma) \\
                & = 2-2-2-2+2+2 = 0.
            \end{align*}
            From this, we have
            \begin{align*}
                [[g_\alpha,g_\beta],g_\gamma]& = [N_{\alpha,\beta} g_{\alpha+\beta}, g_\gamma] = N_{\alpha,\beta} N_{\alpha+\beta,\gamma} g_{\alpha+\beta+\gamma}, \\
                [[g_\beta,g_\gamma],g_\alpha]& = [N_{\beta,\gamma} g_{\beta+\gamma}, g_\alpha] = N_{\beta,\gamma} N_{\beta+\gamma,\alpha} g_{\alpha+\beta+\gamma}, \\
                [[g_\gamma,g_\alpha],g_\beta]& = 0 .
            \end{align*}
            For Jacobi identity, it is enough to show that  $N_{\alpha,\beta} N_{\alpha+\beta,\gamma}N_{\beta,\gamma} N_{\beta+\gamma,\alpha} = -1$. This follows from
            \begin{align*}
                & \mathrm{var}(\beta) \bullet \alpha + \mathrm{var}(\gamma) \bullet (\alpha+\beta) + \mathrm{var}(\gamma) \bullet \beta + \mathrm{var}(\alpha) \bullet (\beta+\gamma) \\
                \equiv \;\; & \mathrm{var}(\beta) \bullet \alpha + \mathrm{var}(\gamma) \bullet \alpha + \mathrm{var}(\alpha) \bullet (\beta+\gamma) \\
                \equiv \;\; & (\alpha,\beta) + (\gamma,\alpha) \equiv -1+0 \equiv 1 \mod{2}.
            \end{align*} 

            \item[ii)] Suppose that both $\beta+\gamma$ and $\gamma+\alpha$ are not geometric roots. Then
$$ \quad\quad\quad\quad\quad (\beta+\gamma,\beta+\gamma) + (\gamma+\alpha,\gamma+\alpha)  = (\alpha+\beta+\gamma,\alpha+\beta+\gamma) + (\alpha,\alpha) + (\beta,\beta) + (\gamma,\gamma) - (\alpha+\beta,\alpha+\beta) = 6.$$
           Note that $(\delta,\delta) = 2 \; \mathrm{var}(\delta) \bullet \delta$ is a positive  even integer for every nonzero $\delta \in H_1(M,\partial M;\Z)$. 
                         Since both $(\beta+\gamma,\beta+\gamma)$ and $(\gamma+\alpha,\gamma+\alpha)$ are not equal to $2$, one of them should be $0$, and the other should be $6$. Without loss of generality, assume that $(\beta+\gamma,\beta+\gamma)=0$ which implies that $\beta+\gamma=0$. Then,
            \begin{align*}
                [[g_\alpha,g_\beta],g_\gamma]& = [N_{\alpha,\beta} g_{\alpha+\beta}, g_\gamma] = N_{\alpha,\beta} N_{\alpha+\beta,\gamma} g_{\alpha} = - g_\alpha, \\
                [[g_\beta,g_\gamma],g_\alpha]& = [-\mathrm{var}(\beta), g_\alpha] = -(\beta,\alpha) g_{\alpha} = + g_\alpha, \\
                [[g_\gamma,g_\alpha],g_\beta]& = 0 .
            \end{align*}
            Here $N_{\alpha,\beta} N_{\alpha+\beta,\gamma} = -1$ follows from
            \begin{align*}
                \mathrm{var}(\beta) \bullet \alpha + \mathrm{var}(-\beta) \bullet (\alpha+\beta) \equiv - \mathrm{var}(\beta) \bullet \beta \equiv 1 \mod{2}.
            \end{align*}
        \end{itemize}
              
    \end{enumerate}
    
 Now, we prove that the complexified geometric Lie algebra $\mathfrak{g}_{geo}\otimes \C$ is isomorphic to the simple Lie algebra of the corresponding type.
 First, to show that the geometric Lie algebra is simple, we use the following theorem.
\begin{thm}[{\cite[Theorem 14.2]{KacL}}]
Let $\mathfrak{g}$ be a finite dimensional Lie algebra over an algebraically closed field $\mathbb{k}$ of characteristic zero.
Suppose $\mathfrak{g}$ admits a decomposition
 $$\mathfrak{g} = \mathfrak{h} \oplus \left(\bigoplus_{\alpha \in \Phi} \mathfrak{g}_{\alpha}\right)$$
into a direct sum of subspaces such that the following properties hold:
\begin{enumerate}
\item $\mathfrak{h}$ is an abelian subalgebra and $\dim \mathfrak{g}_\alpha = 1$ for all $\alpha \in \Phi$,
where  $\mathfrak{g}_\alpha =\{ x \in \mathfrak{g} \mid [h,x]=\alpha(h) x \; \forall h \in \mathfrak{h}\}$.
\item $[g_\alpha, g_{-\alpha}] \in \mathbb{k} h_\alpha$, 
where $h_\alpha \in \mathfrak{h}$ is such that $\alpha(h) \neq 0$.
\item $\mathfrak{h}^*$ is spanned by $\Phi$. 
\end{enumerate}
Then $\mathfrak{g}$ is a semi-simple Lie algebra. Moreover, if $\Phi$ is indecomposable, then $\mathfrak{g}$ is simple.
\end{thm}
For our geometric Lie algebra, (1) holds by definition, (2) holds because
$\alpha(\mathrm{var}(\alpha))$ is non-zero because $\alpha \bullet \mathrm{var}(\alpha) \neq 0$ for any $\alpha \in \Phi_\Gamma$.
(3) follows from Proposition \ref{cartan matrix}.  Hence $\mathfrak{g}_{geo} \otimes \C$ is a semi-simple Lie algebra. 
Since the root system for   $\mathfrak{g}_{geo} \otimes \C$  and that of the classical simple Lie algebra are isomorphic, two Lie algebras are isomorphic to each other. 
This finishes the proof of Theorem \ref{thm:Lie bracket}.
\end{proof}

\section{Geometry of Lie brackets} \label{sec:8}
\subsection{$A_2$-case}
Let us illustrate the construction in the case of $A_2$.
The Lie algebra $\mathfrak{g}$ is given by
$$\mathfrak{sl}(3;\Z)= \{ A\in \mathfrak{gl}(3;\Z) \;|\; \tr A = 0 \} = \Z \langle W_1 , W_2 , X_1 , X_2 , X_3 , Y_1 , Y_2 , Y_3 \rangle,$$
        $$
        W_1 = \begin{pmatrix}
            1 & 0 & 0 \\ 0 & -1 & 0 \\ 0 & 0 & 0
        \end{pmatrix}, \quad
        W_2 = \begin{pmatrix}
            0 & 0 & 0 \\ 0 & 1 & 0 \\ 0 & 0 & -1
        \end{pmatrix}, \quad
        X_1 = \begin{pmatrix}
            0 & 1 & 0 \\ 0 & 0 & 0 \\ 0 & 0 & 0
        \end{pmatrix}, \quad
        X_2 = \begin{pmatrix}
            0 & 0 & 0 \\ 0 & 0 & 1 \\ 0 & 0 & 0
        \end{pmatrix}, \quad
        X_3 = \begin{pmatrix}
            0 & 0 & 1 \\ 0 & 0 & 0 \\ 0 & 0 & 0
        \end{pmatrix},$$
        $$
        Y_1 = \begin{pmatrix}
            0 & 0 & 0 \\ -1 & 0 & 0 \\ 0 & 0 & 0
        \end{pmatrix}, \quad
        Y_2 = \begin{pmatrix}
            0 & 0 & 0 \\ 0 & 0 & 0 \\ 0 & -1 & 0
        \end{pmatrix}, \quad
        Y_3 = \begin{pmatrix}
            0 & 0 & 0 \\ 0 & 0 & 0 \\ -1 & 0 & 0
        \end{pmatrix}.$$
    The Lie bracket between basis elements is given by
       \begin{longtable}{ccc}
            & $[W_1,W_2]=0$, & \\[6pt]
            \, $[W_1,X_1] = 2X_1$ ,  & $[W_1,X_2] = -X_2$ , & $[W_1,X_3] = X_3$ , \\
            \, $[W_2,X_1] = -X_1$ , & $[W_2,X_2] = 2X_2$ ,  & $[W_2,X_3] = X_3$ , \\[6pt]
            \, $[W_1,Y_1] = -2Y_1$ , & $[W_1,Y_2] = Y_2$ ,  & $[W_1,Y_3] = -Y_3$ , \\
            \, $[W_2,Y_1] = Y_1$ ,  & $[W_2,Y_2] = -2Y_2$ , & $[W_2,Y_3] = -Y_3$, \\[6pt]
            \, $[X_1,Y_1] = -W_1$ , & $[X_2,Y_2] = -W_2$ , & $[X_3,Y_3] = -W_1-W_2$, \\[6pt]
            \, $[X_1,X_2] = X_3$ ,  & $[X_1,X_3] = 0$ , & $[X_2,X_3] = 0$, \\
            \, $[X_1,Y_2] = 0$ , & $[X_1,Y_3] = -Y_2$ , & $[X_2,Y_3] = Y_1$, \\
            \, $[Y_1,X_2] = 0$ , & $[Y_1,X_3] = -X_2$ ,  & $[Y_2,X_3] = X_1$, \\
            \, $[Y_1,Y_2] = Y_3$ , & $[Y_1,Y_3] = 0$ , & $[Y_2,Y_3] = 0$.
        \end{longtable}
The geometric Lie algebra from $A_2$-singularity $x^3+y^2$ is constructed as follows.
The Milnor fiber  $M$, defined by $x^3+y^2=1$, is a once-punctured torus (obtained by gluing the opposite edges of the punctured hexagon in Figure \ref{fig:A2roots}) (B).
The  regular triangle formed by three punctures in Figure \ref{fig:A2roots} (B) is the $A_2$-wheel, and 
six geometric roots are homology classes of arcs $\alpha_1,\alpha_2,\alpha_3$ and those of their orientation reversals $-\alpha_1,-\alpha_2,-\alpha_3$.
%
Monodromy images and variation images of geometric roots are illustrated in Figure \ref{fig:A_2 intersections}.
We take $\{ \mathrm{var}(\alpha_1), \mathrm{var}(\alpha_2) \}$ as the basis of $H_1(M;\Z)$.
  
\begin{lemma}
The geometric Lie algebra for $A_2$-singularity is defined on the vector space
$$\mathfrak{g}_{geo} = H_1(M;\Z) \oplus \Z \langle g_{\alpha_1},g_{\alpha_2},g_{\alpha_3},g_{-\alpha_1},g_{-\alpha_2},g_{-\alpha_3} \rangle.$$
The map $\Phi: \mathfrak{g}_{geo} \to \mathfrak{sl}_3(\mathbb{Z})$ defined by
$$\Phi(\mathrm{var}(\alpha_i)) =W_i,\; \Phi(g_{\alpha_i}) = X_i,\mbox{ and }  \Phi(g_{-\alpha_i}) = Y_i$$
is an isomorphism of Lie algebras.
\end{lemma}
  
%
%
%
%

\begin{proof}    From (1) of Definition \ref{Lie bracket}, we get 
    \begin{align*}
        [h,h']=0
    \end{align*}
    for every $h,h'\in H_1(M;\Z)$.
    
    (2) of Definition \ref{Lie bracket} presents
    \begin{align*}
        [\var(\alpha_i),g_{\pm \alpha_j}] = (\alpha_i,\pm \alpha_j) \cdot g_{\pm \alpha_j} = \pm \left(\var(\alpha_i) \bullet \alpha_j + \var(\alpha_i) \bullet \rho_*(\alpha_j) \right) \cdot g_{\pm \alpha_j}.
    \end{align*}
    The intersection number between the arcs and the variation images of arcs can be explicitly computed as described in Figure \ref{fig:A_2 intersections}.
    We get
    \begin{align*}
        [\var(\alpha_1),g_{\alpha_1}] = 2g_{\alpha_1} , \quad [\var(\alpha_1),g_{\alpha_2}] = -g_{\alpha_2} , \quad [\var(\alpha_1),g_{\alpha_3}] = g_{\alpha_3} , \\
        [\var(\alpha_2),g_{\alpha_1}] = -g_{\alpha_1} , \quad [\var(\alpha_2),g_{\alpha_2}] = 2g_{\alpha_2} , \quad [\var(\alpha_2),g_{\alpha_3}] = g_{\alpha_3},
    \end{align*}
    and similar results for $g_{-\alpha_j}$'s. 

    \begin{figure}[h]
    \centering
    \begin{subfigure}{0.32\textwidth}
    \includegraphics[scale=0.7]{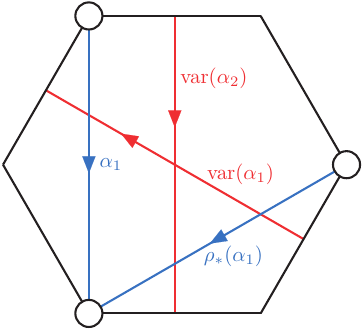}
    \centering
    \caption{$[\var(\alpha_1),g_{\alpha_1}]$ and $[\var(\alpha_2),g_{\alpha_1}]$.}
    \end{subfigure}
    \begin{subfigure}{0.32\textwidth}
    \includegraphics[scale=0.7]{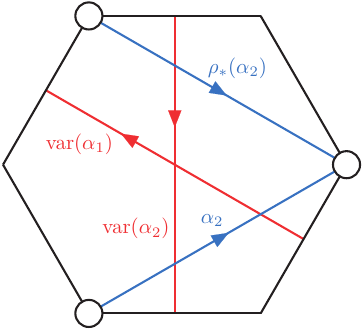}
    \centering
    \caption{$[\var(\alpha_1),g_{\alpha_2}]$ and $[\var(\alpha_2),g_{\alpha_2}]$.}
    \end{subfigure}
    \begin{subfigure}{0.32\textwidth}
    \includegraphics[scale=0.7]{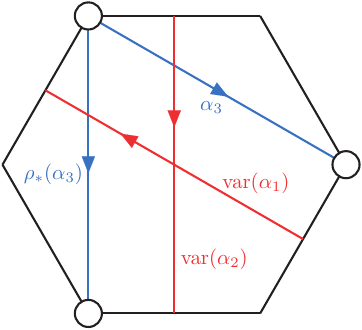}
    \centering
    \caption{$[\var(\alpha_1),g_{\alpha_3}]$ and $[\var(\alpha_2),g_{\alpha_3}]$.}
    \end{subfigure}
    \caption{Computations of Lie brackets.}
        \label{fig:A_2 intersections}
    \end{figure}
    

    (3) of Definition \ref{Lie bracket} presents
    \begin{align*}
        [g_{\alpha_1},g_{-\alpha_1}] = -\mathrm{var}(\alpha_1) , \quad
        [g_{\alpha_2},g_{-\alpha_2}] = -\mathrm{var}(\alpha_2) , \quad
        [g_{\alpha_3},g_{-\alpha_3}] = -\mathrm{var}(\alpha_3) = -\var(\alpha_1) - \var(\alpha_2).
    \end{align*}
    
    Finally, (4) of Definition \ref{Lie bracket} presents
    \begin{align*}
        [g_{\alpha_i},g_{\alpha_j}] = 
        \begin{cases}  
        (-1)^{\var(\alpha_j) \bullet \alpha_i} g_{\alpha_i+\alpha_j} & \textrm{if } \; \alpha_i + \alpha_j \in \Phi_\Gamma, \\ 
        0 & \textrm{if } \;\alpha_i + \alpha_j \notin \Phi_\Gamma. 
        \end{cases}
    \end{align*}
    The intersection number can be computed from Figure \ref{fig:A_2 intersections}. 
    We get
    \begin{align*}
        \begin{array}{lll}
            \, [g_{\alpha_1},g_{\alpha_2}] = g_{\alpha_3} ,  & [g_{\alpha_1},g_{\alpha_3}] = 0 , & [g_{\alpha_2},g_{\alpha_3}] = 0, \\
            \, [g_{\alpha_1},g_{-\alpha_2}] = 0 ,  & [g_{\alpha_1},g_{-\alpha_3}] = -g_{-\alpha_2} , & [g_{\alpha_2},g_{-\alpha_3}] = g_{-\alpha_1}, \\
            \, [g_{-\alpha_1},g_{\alpha_2}] = 0 ,  & [g_{-\alpha_1},g_{\alpha_3}] = -g_{\alpha_2} , & [g_{-\alpha_2},g_{\alpha_3}] = g_{\alpha_1}, \\
            \, [g_{-\alpha_1},g_{-\alpha_2}] = g_{-\alpha_3} ,  & [g_{-\alpha_1},g_{-\alpha_3}] = 0 , & [g_{-\alpha_2},g_{-\alpha_3}] = 0.
        \end{array}
    \end{align*}
   \end{proof} 
Now the signs appearing in type (4) in the proof have the following geometric interpretation of the right hand rule.
Lie bracket is non-zero if two roots can be added to another root.  When it does, we need to determine the sign $\pm 1$.
The claim is that for type $A$, the sign is given by the right hand rule.  Namely,  the Lie bracket $[g_{\alpha_1},g_{\alpha_2}] = g_{\alpha_3}$ is realized by the vector sum $\alpha_3=\alpha_1+\alpha_2$ in Figure \ref{fig:A2roots} (B), and the sign is positive as the orientation given by the pair $(\alpha_1,\alpha_2)$ agrees with the (complex) orientation of the Milnor fiber.
In contrast, the sign in the Lie braket $[g_{-\alpha_1},g_{\alpha_3}] = -g_{\alpha_2}$ is negative as the pair $(-\alpha_1,\alpha_3)$ does not agree with the orientation.
This sign rule holds only for $A_k$-cases. 

\subsection{$A_k$-case}   
There is a straight-forward generalization of $A_2$-case to the $A_k$-case. 
The Lie algebra $\mathfrak{sl}_{k+1}$
is decomposed into the Cartan subalgebra generated by $h_1,\dots, h_k$ (where $h_i = E_{i,i}-E_{i+1,i+1}$), and 
the root spaces $X_{ij}=(E_{i,j})$ and $Y_{ji}= (-E_{j,i})$ for $i<j$.
Milnor fiber $M$ of $A_k$-singularity is given by $\{x^{k+1}+y^2=1\}$, or a punctured $2(k+1)$-gon with $\pm 3$ boundary edge identifications.
$A_k$-wheel is given by a regular $(k+1)$-gon whose vertices are the punctures labeled by $v_{1}, \dots, v_{k+1}$.
Let $\alpha_{ij}$ be an arc given by the line segment connecting $v_{i}$ and $v_{j}$.
Variation images of $\alpha_{12}, \alpha_{23}, \dots, \alpha_{k(k+1)}$ generate $H_1(M;\mathbb{Z})\cong \Z^k$.
The isomorphism of Lie algebras is obtained by matching $\mathrm{var}(\alpha_{i (i+1)})$ with $h_i$, $g_{\alpha_{ij}}$ with $X_{ij}$, and $g_{-\alpha_{ij}}$ with $Y_{ji}$
for $i<j$.

Concatenation of arcs $\alpha_{ij}, \alpha_{jl}$ to $\alpha_{il}$ for distinct $i,j,l$ gives the nontrivial Lie bracket
$$[g_{\alpha_{ij}}, g_{\alpha_{jl}}] = \pm g_{\alpha_{il}}$$ 
and the sign is $+$ if vertices $(i,j,l)$ are ordered counter-clockwise on the $(k+1)$-gon and $-$ otherwise. Hence, the sign is given by the
right hand rule.

\subsection{$D_k$-case}
There is a simple geometric way to compute the Lie bracket for $D_k$-cases by looking at the
triangle that is formed by the two geometric roots.

Recall that $D_k$-wheel is given by a regular $2(k-1)$-gon whose vertices  $v_1,\dots v_{k-1}, v_{-1}, \dots, v_{-(k-1)}$ (labelled cyclically) are punctures with an additional puncture $v_0$ at the center.
$D_k$-Milnor fiber is obtained by identifying opposite edges of the $D_k$-wheel.
Now the relative homology classes of oriented arcs $L_{i,j}$ (line segment connecting from $v_i$ to $v_j$) are geometric roots.
With this notation, the consequence of Lemma \ref{Dk roots} can be rewritten as follows.
\begin{lemma}
Let $[L_{i,j}]\in H_1(M,\partial M;\Z)$ denote the relative homology class of the line segment $L_{i,j}$.
\begin{itemize}
    \item $[L_{i,j}]$ $(i,j \in \{-(k-1),\dots,k-1\})$ is contained in the $D_k$ geometric root system if and only if $i\ne \pm j$.
    \item Two distinct line segments $L_{i_1,j_1}$ and $L_{i_2,j_2}$ have same relative homology $[L_{i_1,j_1}]=[L_{i_2,j_2}]$ if and only if $i_1= -j_2$, $j_1=-i_2$, and $i_1,i_2,j_1,j_2\ne 0$.
\end{itemize}
\end{lemma}

 
It is not difficult to determine when the sum of two distinct roots in $\Phi_{D_k}$ is again a root.

\begin{lemma}
    Let $\alpha_1,\alpha_2$ be distinct geometric roots in $\Phi_{D_k}$.
    Then $\alpha_1+\alpha_2$ is again a geometric root if and only if $\alpha_1$ and $\alpha_2$ have representatives $L_{i_1,j_1}$ and $L_{i_2,j_2}$, respectively, such that $j_1=i_2$ and $i_1 \ne \pm j_2$.
\end{lemma}
In particular, $[L_{i,j}]$ and $[L_{j,-i}]$ $(i\ne 0)$ are not summable due to the center puncture of $D_k$-wheel.
Thus, $\alpha_1=[L_{i,j}]$ and $\alpha_2=[L_{j,l}]$ are summable when $L_{i,j},L_{j,l},L_{l,i}$ compose a triangle in the $D_k$-wheel.
Now we describe the sign of Lie bracket $[g_{\alpha_1},g_{\alpha_2}]$ in this case. 
 
    \begin{figure}[h]
    \centering
    \begin{subfigure}{0.47\textwidth}
    \includegraphics[scale=0.8]{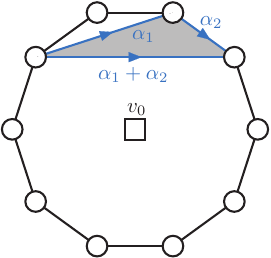}
    \centering
    \caption{$[g_{\alpha_1},g_{\alpha_2}]=-g_{\alpha_1+\alpha_2}$ since $\epsilon=-1$.}
    \end{subfigure}
    \begin{subfigure}{0.47\textwidth}
    \includegraphics[scale=0.8]{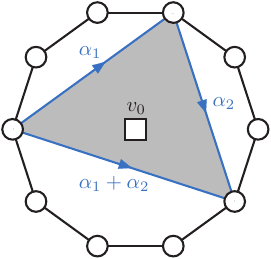}
    \centering
    \caption{$[g_{\alpha_1},g_{\alpha_2}]=-(-1)g_{\alpha_1+\alpha_2}$.}
    \end{subfigure}
    \caption{Illustration of  $D_k$ Lie bracket in Proposition \ref{prop:dksign}.}
        \label{fig:Dk sign}
    \end{figure}
    
\begin{prop}\label{prop:dksign}
Let $\alpha_1,\alpha_2$ be distinct geometric roots, and suppose that $\alpha_1+\alpha_2$ is a geometric root.
Let $\epsilon_{\alpha_1,\alpha_2}$ be the sign given by the right hand rule ($+1$ if $(\alpha_1,\alpha_2)$ agrees with the orientation of $M$).
Then, 
$$[g_{\alpha_1},g_{\alpha_2}] =  \begin{cases} 
\epsilon_{\alpha_1,\alpha_2} g_{\alpha_1+\alpha_2} & \textrm{if the triangle is formed by}\;  \alpha_1,\alpha_2,\mbox{ and }\alpha_1+\alpha_2  \textrm{ does not contain} \;\; v_0, \\
- \epsilon_{\alpha_1,\alpha_2} g_{\alpha_1+\alpha_2} & \textrm{if the triangle is formed by}\;  \alpha_1,\alpha_2,\mbox{ and }\alpha_1+\alpha_2  \textrm{ contains} \;\; v_0,  \end{cases}$$
\end{prop}
\begin{proof}
The sign is given by the parity of the intersection $\mathrm{var}(\alpha_2) \bullet \alpha_1$. 
This can be checked through case by case analysis: 
first case is when the triangle $\alpha_1,\alpha_2,\alpha_1+\alpha_2$ does not meet $v_0$, second case is when the triangle contains $v_0$ as a vertex,
and last case is when the triangle contains $v_0$. We leave the details to the reader.
\end{proof}
Hence the sign is given by the right hand rule, with the additional $(-1)$ sign if three geometric roots form a triangle that contains the center $v_0$ in the interior.
Signs for $E$-types are more subtle. One can specify the shape and size of the triangles formed by three vectors that requires additional $(-1)$ signs as in $D_{k}$-case, but omit them as there are too many to list.

\section{Description of Weyl group reflections and a Coxeter element} \label{sec:9}

Root system $\Phi$ on $\mathbb{R}^k$ with the pairing $(\cdot, \cdot)$ has an associated Weyl group, which 
is generated by $\{s_\alpha\}_{\alpha \in \Phi}$. 
Here, $s_\alpha$ is a reflection of $\mathbb{R}^k$ with respect to the hyperplane that is perpendicular to $\alpha$, and this restricts to an action on the set of the roots.
$$s_\alpha(\beta) := \beta -2 \tfrac{(\alpha,\beta)}{(\alpha,\alpha)}\alpha.$$
In this section, we explain the Weyl group action on the set of geometric roots and find the relation between the Coxeter element and the monodromy of the singularity.
Recall that our pairing is the negative symmetrized Seifert form. 
For a simple singularity of three variables, this is the same as intersection form, and $s_\alpha$ can be then identified with Picard--Lefschetz transformation.
In particular, the corresponding Weyl group can be interpreted as the monodromy group.
For a simple singularity of two variables, on the other hand, the pairing is not the same as intersection form and $s_\alpha$ cannot be then identified with Picard--Lefschetz transformation, so we need an alternative description for the Weyl group.
We will see that it corresponds to certain ``combinatorial operation'' of flipping roots.

The following is easy to check.
\begin{lemma}
Let $\alpha,\beta \in \Phi_\Gamma$ be geometric roots. Then
\begin{itemize}
    \item [(i)] $s_\alpha(\beta) = \beta+\alpha$\quad if $\beta+\alpha\in \Phi_\Gamma$, or equivalently $(\alpha,\beta)=-1$,
    \item [(ii)] $s_\alpha(\beta) = \beta-\alpha$\quad if $\beta-\alpha\in \Phi_\Gamma$, or equivalently $(\alpha,\beta)=1$,
    \item [(iii)] $s_\alpha(\beta) = \alpha$\quad\quad\;\, if $(\alpha,\beta)=0$,
    \item [(iv)] $s_\alpha(\pm \alpha) = \mp \alpha.$ 
\end{itemize}
\end{lemma}

Here is a geometric interpretation of Weyl group action on geometric roots.
$s_\alpha$ flips $\alpha$ (as an arrow) to $-\alpha$ (switching source and target).
Then for a geometric root $\beta$ that meets $\alpha$ in the wheel, imagine that
$\beta$ arrow is tied to the  the source/target node of $\alpha$ so that when $\alpha$ arrow is flipped,
$\beta$ arrow is switched accordingly (if the resulting arrow is still a root).
See Figure \ref{fig:Weyl_example} for an illustration.

   \begin{figure}[h]
        \centering
        \includegraphics[scale=0.9]{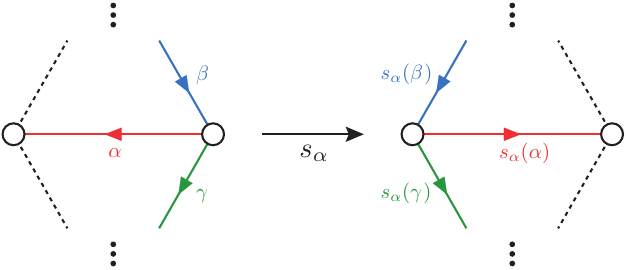}
        \caption{Local picture of a simple reflection $s_{\alpha}$.}
        \label{fig:Weyl_example}
    \end{figure}
%

We illustrate an example of the Weyl group action on $\Phi_{D_{5}}$ in Figure \ref{fig:Weyl_D5}.
First, $s_\alpha(\beta_1)$ did not change from $\beta_1$ since $\beta_1$ and $\pm \alpha$ are not summable in the geometric root system $\Phi_{D_5}$. 
For $s_\alpha(\beta_2)$, recall that $\alpha$ has two geometric representatives given by parallel ones
as in Figure \ref{fig:Dkroots} (B). This parallel arrow meets
$\beta_2$ and flipping this arrow gives $s_\alpha(\beta_2)$ in the figure. The same applies for $\beta_3$ and $\beta_4$.
    \begin{figure}[h]
        \centering
        \includegraphics[scale=0.8]{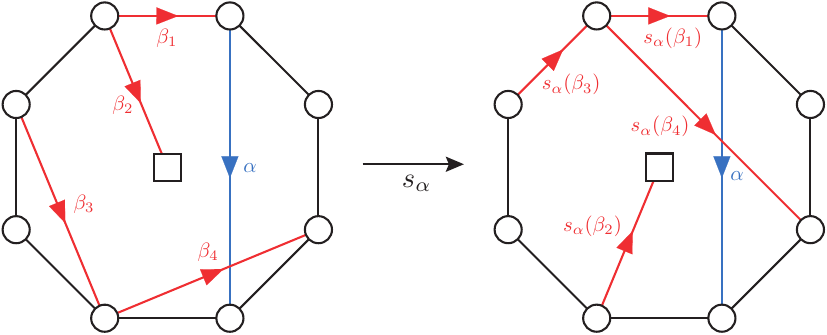}
        \caption{Example of the Weyl group action on $\Phi_{D_{5}}$.}
        \label{fig:Weyl_D5}
    \end{figure}
 
Therefore, we may say that $s_\alpha$ flips the roots involved with $\alpha$.
We remark that  the relation  $s_\alpha s_{\alpha'} s_\alpha = s_{\alpha'} s_\alpha s_{\alpha'}$
(when $(\alpha,\alpha')=-1$) in the Weyl group can be easily observed.

Now, we will find a geometric interpretation of Coxeter element and its action on geometric roots.
Recall that we have chosen a simple basis $\{\alpha_1, \dots, \alpha_k\}$  of $\Phi_\Gamma$ in Proposition \ref{cartan matrix}.
From Theorem \ref{thm_gz80}, the variation image $\{\Delta_1,\dots,\Delta_k\}$ of the simple basis $\{\alpha_1, \dots, \alpha_k\}$  forms a distinguished collection of vanishing cycles of the simple singularity $f(x,y)$.

Given a simple basis $\{\alpha_1, \dots, \alpha_k\}$, corresponding Coxeter element $c$ is defined as 
$$c = s_{\alpha_1} \circ \dots \circ s_{\alpha_k}.$$

Proposition 2.4 of Hubery--Krause \cite{HKra} says that
 the action of Coxeter element $c$ on geometric roots equals the operator $\overline{\rho}_*$, monodromy operator composed with an orientation reversal. This  is a variant of a classical result concerning skew-symmetric Milnor lattices, as found on page 36 of \cite{GZ95}.

In particular, this together with the discussions in Section \ref{sec:monodromy}, provides a geometric realization of Kostant's theorem (Theorem \ref{Kostant_original}) in terms of the Coxeter wheel.

\begin{thm}\label{kostant}
    The group $\langle \;\overline{\rho}_* \rangle$ acts freely on $\Phi_\Gamma$.
    The order of $\;\overline{\rho}_*$ and the number of its orbits are summarized as follows.
    \begin{center}
    \begin{tabular}{ c||c|c|c|c|c } 
    Type $(\Gamma)$ & $A_k$ & $D_k$ & $E_6$ & $E_7$ & $E_8$ \\
    \hline
    Number of roots  & $k(k+1)$ & $2k(k-1)$ & $72$ & $126$ & $240$ \\
    \hline
    Order of action  & $k+1$ & $2(k-1)$ & $12$ & $18$ & $30$ \\ 
    \hline
    Number of orbits & $k$ & $k$ & $6$ & $7$ & $8$
\end{tabular}
\end{center}
In particular, the order of action is equal to the Coxeter number $h$ of the Coxeter group of type $\Gamma$, and the number of orbits is equal to $k$, the dimension of the homology group $H_1(M,\partial M;\Z)$.
\end{thm}

\begin{proof}
The proof of this theorem is an adaptation of that of Section \ref{sec:monodromy}.
Orientation reversal and monodromy commutes, but finding the precise order of $\;\overline{\rho}_*$ is more subtle.
\begin{enumerate}
    \item The case of $A_k$

    Recall the rule of monodromy action on $A_k$-wheel (see, e.g., Proposition \ref{free_Ak}).
    If $\alpha\in \Phi_\Gamma$ is represented by an oriented line $L$, then $\rho_*(\alpha)$ is represented by the line obtained by rotating $L$ through an angle of $\tfrac{2\pi}{k+1}$ and inverting its orientation.
    Thus, $\overline{\rho}_*(\alpha)$ can be considered as the rotation of $\alpha$ through an angle of $\tfrac{2\pi}{k+1}$.
    Hence, the action of $\;\overline{\rho}_*$ on $\Phi_{A_k}$ is a free $\Z/(k+1)\Z$-action.

    \item The case of $D_k$

    Recall the rule of monodromy action on $D_k$-wheel (see, e.g., Proposition \ref{free_Dk}).
    $\overline{\rho}_*(\alpha)$ can be considered as the rotation of $\alpha$ through an angle of $\tfrac{k \pi}{k-1}$ followed by an orientation reversal.

    First suppose that $k$ is an even number.
    From Proposition \ref{free_Dk}, recall that $\rho_*^{k-1}(\alpha)=\alpha$ for every $\alpha\in \Phi_\Gamma$.
    Thus, we have $\overline{\rho}_*^{2k-2}(\alpha)=\alpha$ and $\overline{\rho}_*^{k-1}(\alpha)=-\alpha$.
    The representatives of monodromy images $\overline{\rho}_*(\alpha),\dots,\overline{\rho}_*^{k-2}(\alpha)$, $\overline{\rho}_*^{k}(\alpha),\dots,\overline{\rho}_*^{2k-3}(\alpha)$ are not parallel to the representative of $\alpha$, so they are different to $\alpha$.
    Therefore, the action of $\;\overline{\rho}_*$ on $\Phi_{D_k}$ ($k$ even) is a free $\Z/(2k-2)\Z$-action.

    Now suppose that $k$ is an odd number.
    From Proposition \ref{free_Dk}, recall that $\rho_*^{2k-2}(\alpha)=\alpha$ for every $\alpha\in \Phi_\Gamma$.
    Thus, we have $\overline{\rho}_*^{2k-2}(\alpha)=\alpha$.
    The representatives of monodromy images $\overline{\rho}_*(\alpha),\dots,$ $\overline{\rho}_*^{k-2}(\alpha)$, $\overline{\rho}_*^{k}(\alpha),\dots,\overline{\rho}_*^{2k-3}(\alpha)$ are not parallel to the representative of $\alpha$, so they are different to $\alpha$.
    Since $k$ is odd, we also have $\overline{\rho}_*^{k-1}(\alpha)= \rho_*^{k-1}(\alpha) \ne \alpha$.
    Therefore, the action of $\;\overline{\rho}_*$ on $\Phi_{D_k}$ ($k$ even) is a free $\Z/(2k-2)\Z$-action.

    \item The case of $E_6$

    Recall the matrix $P$ which represents the monodromy operator $\rho_*$ with respect to the basis $\{\beta_1,\dots,\beta_6\}$ of $\Phi_{E_6}$ from the proof of Proposition \ref{free_E6}.
    The matrix for the Coxeter element $\;\overline{\rho}_*$ is equal to $\overline{P}:=-P$.
    We have $\overline{P}^{12} = P^{12} = \mathrm{Id}$.
    As we did in Proposition \ref{free_E6}, it can be checked that none of the columns of $\overline{P},\dots,\overline{P}^{11}$ coincides with a column of the identity matrix.
    Therefore, the action of $\;\overline{\rho}_*$ on $\Phi_{E_6}$ is a free $\Z/12\Z$-action.
    
    \item The case of $E_7$

    Recall the matrix $P$ which represents the monodromy operator $\rho_*$ with respect to the basis $\{\beta_1,\dots,\beta_7\}$ of $\Phi_{E_7}$ from the proof of Proposition \ref{free_E7}.
    The matrix for the Coxeter element $\;\overline{\rho}_*$ is equal to $\overline{P}:=-P$.
    We have $\overline{P}^{9} = -P^{9} = -\mathrm{Id}$, but $\overline{P}^{18} = P^{18} = \mathrm{Id}$.
    Note that the orbit of $\beta_j$ under the action of $\;\overline{\rho}_*$ is same with $\Psi^+_j \cup \Psi^-_j$.
    Each orbit has 18 elements, and hence the action of $\;\overline{\rho}_*$ on $\Phi_{E_7}$ is a free $\Z/18\Z$-action.

    \item The case of $E_8$:

    Recall the matrix $P$ which represents the monodromy operator $\rho_*$ with respect to the basis $\{\beta_1,\dots,\beta_8\}$ of $\Phi_{E_8}$ from the proof of Proposition \ref{free_E8}.
    The matrix for the Coxeter element $\;\overline{\rho}_*$ is equal to $\overline{P}:=-P$.
    We have $\overline{P}^{15} = -P^{15} = -\mathrm{Id}$, but $\overline{P}^{30} = P^{30} = \mathrm{Id}$.
    Note that the orbit of $\beta_j$ under the action of $\;\overline{\rho}_*$ is same with $\Psi^+_j \cup \Psi^-_j$.
    Each orbit has 30 elements, and hence the action of $\;\overline{\rho}_*$ on $\Phi_{E_8}$ is a free $\Z/30\Z$-action. \qedhere
\end{enumerate}
\end{proof}


\bibliographystyle{amsalpha}
\bibliography{FukayaSing}

\end{document}